\documentclass{amsart}
\usepackage{amsmath,amsthm, amssymb,mathabx, mathrsfs, epsfig, xypic, mathtools}
\usepackage{cite,  mathscinet}
\usepackage{xr-hyper}
\usepackage{enumitem}
\usepackage{tikz, tikz-cd}
\usetikzlibrary{arrows, calc, patterns, cd, shapes, trees, matrix}
\usepackage[colorlinks]{hyperref}
 \let\underbrace\LaTeXunderbrace

\DeclareMathOperator*{\colim}{colim}
\begin{document}

\newcommand{\actsonr}{\mathrel{\reflectbox{$\righttoleftarrow$}}}
\newcommand{\actsonl}{\mathrel{\reflectbox{$\lefttorightarrow$}}}

\newcommand{\floor}[1]{\lfloor #1 \rfloor}

\newcommand{\isoeq}{\cong}
\newcommand{\cech}{\vee}
\newcommand{\dsum}{\mathop{\oplus}}

\newcommand{\Tr}{\mathrm{Tr}}
\newcommand{\Gal}{\mathrm{Gal}}
\newcommand{\tr}{\mathrm{tr}}

\newcommand{\Conf}{C}
\newcommand{\Sym}{\mathrm{Sym}}
\newcommand{\End}{\mathrm{End}}

\newcommand{\Var}{\mathrm{Var}}
\newcommand{\van}{\mathrm{van}}
\newcommand{\Loc}{\mathrm{Loc}}
\newcommand{\VHS}{\mathrm{VHS}}
\newcommand{\HS}{\mathrm{HS}}
\newcommand{\MHM}{\mathrm{MHM}}

\newcommand{\calF}{\mathcal{F}}
\newcommand{\calQ}{\mathcal{Q}}
\newcommand{\calL}{\mathcal{L}}
\newcommand{\calO}{\mathcal{O}}
\newcommand{\calM}{\mathcal{M}}
\newcommand{\calG}{\mathcal{G}}

\newcommand{\frakm}{\mathfrak{m}}

\newcommand{\bbN}{\mathbb{N}}
\newcommand{\bbZ}{\mathbb{Z}}
\newcommand{\bbQ}{\mathbb{Q}}
\newcommand{\bbR}{\mathbb{R}}
\newcommand{\bbC}{\mathbb{C}}
\newcommand{\bbP}{\mathbb{P}}
\newcommand{\bbA}{\mathbb{A}}
\newcommand{\bbL}{\mathbb{L}}
\newcommand{\bbF}{\mathbb{F}}

\newcommand{\bc}{\mathbf{c}}

\newcommand{\Spec}{\mathrm{Spec}\,}
\newcommand{\motcomp}{\widehat{\calM_\bbL}} 
\newcommand{\mot}{\calM}

\numberwithin{equation}{subsubsection}
\theoremstyle{plain}

\newtheorem{maintheorem}{Theorem} 
\renewcommand{\themaintheorem}{\Alph{maintheorem}} 
\newtheorem{maincorollary}[maintheorem]{Corollary}
\newtheorem{mainconjecture}[maintheorem]{Conjecture}

\newtheorem*{theorem*}{Theorem}

\newtheorem{theorem}[subsubsection]{Theorem}
\newtheorem{corollary}[subsubsection]{Corollary}
\newtheorem{conjecture}[subsubsection]{Conjecture}
\newtheorem{proposition}[subsubsection]{Proposition}
\newtheorem{lemma}[subsubsection]{Lemma}

\theoremstyle{definition}

\newtheorem{example}[subsubsection]{Example}
\newtheorem{definition}[subsubsection]{Definition}
\newtheorem{remark}[subsubsection]{Remark}
\newtheorem{question}[subsubsection]{Question}

\theoremstyle{remark}

\newtheorem{notation}[subsubsection]{Notation}

\newcommand{\calP}{\mathcal{P}}
\newcommand{\calA}{\mathcal{A}}
\newcommand{\calB}{\mathcal{B}}

\newcommand{\M}{\mathcal{M}}
\newcommand{\LL}{\bbL}
\newcommand{\bbK}{\mathbb{K}}
\renewcommand{\i}{\underline{i}}
\renewcommand{\j}{\underline{j}}
\renewcommand{\k}{\underline{k}}
\newcommand{\s}{\underline{s}}
\renewcommand{\t}{\underline{t}}
\newcommand{\PConf}{\mathrm{PConf}}
\newcommand{\QCoh}{\mathrm{QCoh}}

\newcommand{\Noeth}{\mathrm{Noeth}}
\newcommand{\cons}{\mathrm{cons}}
\newcommand{\Map}{\mathrm{Map}}

\newcommand{\Coh}{\mathrm{Coh}}
\newcommand{\bbV}{\mathbb{V}}
\newcommand{\id}{\mathrm{id}}
\newcommand{\red}{\mathrm{red}}
\newcommand{\Top}{\mathrm{Top}}
\newcommand{\bs}{\backslash}
\newcommand{\RS}{\mathrm{RS}}
\newcommand{\PW}{\mathrm{PW}}

\newcommand{\bbE}{\mathbb{E}}

\newcommand{\Kap}{\mathrm{Kap}}
\newcommand{\Fil}{\mathrm{Fil}}
\newcommand{\supp}{\mathrm{supp}}

\title{Motivic Euler products in motivic statistics}
\author{Margaret Bilu}
\author{Sean Howe}

\maketitle

\begin{abstract} We formulate and prove an analog of Poonen's finite-field Bertini theorem with Taylor conditions that holds in the Grothendieck ring of varieties. This gives a broad generalization of the work of Vakil-Wood, who treated the case of smooth hypersurface sections. In fact, our techniques give analogs in motivic statistics of all known results in arithmetic statistics that have been proven using Poonen's sieve, including work of Bucur-Kedlaya on complete intersections and Erman-Wood on semi-ample Bertini theorems.  A key ingredient is the use of motivic Euler products, as introduced by the first author, to write down candidate motivic probabilities. We also formulate a conjecture on the uniform convergence of zeta functions that unifies motivic and arithmetic statistics for varieties over finite fields. 
\end{abstract}

\tableofcontents
\section{Introduction}

We begin by providing some context and motivation for our results. Readers interested in seeing the statement of our main theorem first may wish to skip immediately to section \ref{subsec.results}. 

\subsection{Motivation} 
\subsubsection{Arithmetic statistics}
The field of arithmetic statistics is concerned with the study of counting problems in arithmetic geometry through a statistical lens. A simple but fundamental example of a problem in arithmetic statistics is the following: what is the probability that a randomly chosen positive integer is square-free? 

This question can be made precise by considering the probability for integers up to a fixed $N$, and then taking the limit as $N \rightarrow \infty$. There is a simple heuristic which provides a candidate answer: an integer $m$ is squarefree if and only if, for every prime number $p$, $p^2$ does not divide $m$. When $N$ is large relative to $p$, the proportion of positive integers $< N$ that are divisible by $p^2$ is approximately $1/p^2$. If we assume these conditions are independent at different primes, we obtain the candidate probability
\[ \prod_{p\,\mathrm{prime}} (1-1/p^2)= \zeta(2)^{-1} \]
where here $\zeta$ denotes the Riemann zeta function. By controlling the error terms, this heuristic can be made precise, and the asymptotic probability that a random positive integer is squarefree is indeed $\zeta(2)^{-1}$. 

\subsubsection{Poonen's Bertini theorem} Arguing through a similar heuristic, Poonen \cite{poonen:bertini} proved the following analog for varieties over a finite field $\bbF_q$:
\begin{theorem*}[Poonen] Let $X \subset \bbP^n_{\bbF_q}$ be a smooth projective variety. The probability that a random hypersurface of degree $d$ intersects $X$ transversely approaches $\zeta_X(\dim X + 1)^{-1}$ as $d \rightarrow \infty$, where $\zeta_X$ is the Hasse-Weil zeta function of $X$.
\end{theorem*}

\begin{remark}
If the variety $X$ in Poonen's theorem is a curve, then a hypersurface intersects $X$ transversely if its defining equation has only simple zeroes when restricted to $X$. By analogy, if we think of the prime numbers as being the points on a curve and the integers as giving functions on the prime numbers (i.e. we consider the arithmetic curve $\Spec\, \bbZ$), then the statement that an integer is squarefree can be thought of as saying that it has only simple zeroes on the set of primes. 
\end{remark}

The special value of the zeta function in Poonen's theorem arises through a similar heuristic: to say a hypersurface intersects $X$ transversely is a condition that can be checked point by point in $X$, and at each point transversality is a simple condition on the first order Taylor expansion of a defining equation for the hypersurface. The zeta function arises when we take the product of these local contributions under the assumption that they are independent. Controlling the error term in order to verify the heuristic, however, is considerably more difficult in this case; Poonen accomplishes it via a careful sieving argument. 

\subsubsection{Poonen's sieve in arithmetic statistics}
Instead of considering transverse intersections, one could also consider more general conditions on Taylor series: for example, imposing singularities of a prescribed type, or different values at certain points (or, returning to our original example of the integers, one could change the local conditions by instead considering, e.g., cube-free integers, or squarefree integers in a fixed residue class mod 7). The candidate probability is still given by an infinite Euler product of the local densities, however, this product cannot, in general, be related to a special value of a zeta function. Nonetheless, Poonen's sieving technique is flexible enough to show that for a very general class of conditions on the Taylor series, the asymptotic probability that a hypersurface section satisfies them is given by the Euler product predicted by the heuristic.

In this generality, Poonen's sieve has proven to be a versatile tool in arithmetic statistics, and has led to a variety of other results in the field (cf. e.g. \cite{bucur-kedlaya:complete-intersections, erman-wood:semiample-bertini, gunther:random-hypersurfaces, howe:mrv1, howe:mrv2}). 

\subsubsection{From counting points to topology}
The space of hypersurfaces of degree $d$ intersecting $X$ transversely has the structure of an algebraic variety over $\bbF_q$, and Poonen's result is a statement about the number of points on that variety. By the Grothendieck-Lefschetz fixed point formula, the number of points on an algebraic variety over $\bbF_q$ can be computed from the cohomology of that variety. Thus, Poonen's result can be thought of as describing an asymptotic (as $d\rightarrow \infty$) topological property of this sequence of varieties. 

It is then natural to ask, if we work instead over the complex numbers and consider a smooth projective variety $X/\bbC$, can we make any useful asymptotic statements about the topology of the complex manifold parameterizing degree $d$ hypersurfaces that intersect $X$ transversely? This question was studied by Vakil-Wood \cite{vakil-wood:discriminants}, who showed the answer is yes, launching the field of \emph{motivic statistics}. 

\subsubsection{The Grothendieck ring of varieties}

Vakil and Wood work in $K_0(\Var/\bbC)$, the Grothendieck ring of complex algebraic varieties\footnote{In fact, they work over an arbitrary base field $\bbK$. However, there is a gap in their proof in positive characteristic, which we will address below.}. Additively, this is the free abelian group on generators $[Y]$ for $Y/\bbC$ an isomorphism class of variety, modulo the relations $[Y]=[U]+[Z]$ whenever $Z$ is a closed subvariety of $Y$ with open complement $U$. The multiplication is induced by Cartesian product: ${[X][Y]=[X \times Y].}$ 

The main result of Vakil-Wood is a statistical statement in this ring. In order to compare probabilities along a sequence of varieties of increasing dimension, it will be necessary to renormalize (just as in Poonen's theorem, where, rather than speaking of the number of transverse hypersurface sections, we instead speak of the probability that a hypersurface section is transverse). To accomplish this renormalization, we must invert the class $\bbL := [\bbA^1]$ to obtain the ring \[\calM_{\bbC}=K_0(\Var/\bbC)[\bbL^{-1}]. \]
Furthermore, to make an asymptotic statement it is necessary to take a limit, so we also complete $\calM_{\bbC}$ with respect to the dimension topology, where classes of ``very negative dimension", are considered to be small. We write $\widehat{\calM}_\bbC$ for this completion.

The class of a variety $Y/\bbC$ in $\widehat{\calM}_\bbC$ encodes important information about the topology of the complex manifold $Y(\bbC)$. For example, when $Y$ is smooth and projective, all of the rational cohomology groups of $Y$, with their Hodge structures, can be recovered from the class $[Y] \in \widehat{\calM}_\bbC$. 

\begin{remark} Results in $K_0(\Var/\bbC)$ and its variants are often described with the adjective \emph{motivic}. This is after Grothendieck's idea of motives as a universal cohomology theory for algebraic varieties. The Grothendieck ring does not provide a universal cohomology theory, however, the assignment $Y \mapsto [Y]$ can be thought of as the universal compactly supported Euler characteristic. 
\end{remark}

\subsubsection{Vakil-Wood's motivic Bertini theorem}
Given $X \subset \bbP^n_{\bbC}$ a smooth projective variety, Vakil and Wood study the \emph{motivic probability} that a random degree $d$ hypersurface is transverse to $X$. Concretely, let $V_d$ be the space of homogeneous polynomials of degree $d$ in $n+1$ variables, and let $U_d \subset V_d$ be the open subvariety of polynomials corresponding to hypersurfaces in $\bbP^n$ intersecting $X$ transversely. Note that for $N_d:=\binom{d+n}{n}$, $V_d \cong \bbA^{N_d}$, so that we can form the quotient
\[ \frac{[U_d]}{[V_d]}=\frac{[U_d]}{\bbL^{N_d}} \in \widehat{\calM}_\bbC. \]
This quotient is the motivic probability that a random degree $d$ hypersurface intersects $X$ transversely; note that if we were working over a finite field, then replacing the class in the Grothendieck ring with the number of points in this formula would recover the genuine probability appearing in Poonen's theorem.  

In order to formulate a motivic analog of Poonen's theorem, one must also give a natural motivic expression for the zeta value. For this, recall the  following well-known fact: if $X /\bbF_q$ is a variety, then we can compute the power series expansion for the zeta function\footnote{The zeta function $\zeta_X$ appearing earlier is related by the change of variables $\zeta_X(s)=Z_X(q^{-s})$.}
\[ Z_X(t) := \prod_{x \in |X|} \frac{1}{1-t^{\deg x}} = 1 + \#X(\bbF_q) t + \# \Sym^2 X(\bbF_q) t^2 +  \# \Sym^3 X(\bbF_q) t^3 + \ldots \]
where for every $n\geq 1$, $\Sym^nX$ is the $n$-th symmetric power of~$X$. 
Although there is no obvious way to make sense of the Euler product defining $Z_X(t)$ when $X/\bbC$, the power series has a natural motivic analog: for $X/\bbC$ a quasi-projective variety, we define the Kapranov zeta function by
\[ Z_X^\Kap(t) = 1 + [X]t + [\Sym^2 X]t^2 + [\Sym^3 X]t^3 + \ldots \in 1 + t K_0(\Var/\bbC)[[t]]. \]
We can now state one of the main results of \cite{vakil-wood:discriminants}:
\begin{theorem}[Vakil-Wood]\label{theorem.vakil-wood} Let $X \subset \bbP^n_\bbC$ be a smooth projective variety. The motivic probability that a random hypersurface of degree $d$ intersects $X$ transversely is asymptotically equal to $1/Z_X^\Kap\left(\bbL^{-(\dim X + 1)}\right)$ in $\widehat{\calM}_\bbC$. 
\end{theorem}
\noindent Concretely, in the notation used above, Theorem \ref{theorem.vakil-wood} says that 
\[ \lim_{d \rightarrow \infty} \frac{[U_d]}{[V_d]}=\left(Z_X^\Kap\left(\bbL^{-(\dim X + 1)}\right)\right)^{-1}. \]

\subsubsection{Motivic statistics}\label{subsubsec:motivic statistics}
While Poonen's Bertini theorem paved the way for many further applications to arithmetic statistics, the applications of Vakil-Wood's method have so far been much more limited. Indeed, it has been extended in essentially only two ways:
\begin{enumerate}
\item In \cite{vakil-wood:discriminants}, Vakil-Wood also compute the asymptotic motivic probability that a hypersurface has exactly $m$ points of non-transverse intersection with $X$.
\item In \cite{howe:mrv2}, the second author computes the motivic statistics of the universal family of transverse hypersurface sections of $X$. The results can be summarized by saying that the universal family is a motivic binomial random variable with parameters 
\[ N=[X],\; p=\frac{\bbL^{\dim X} - 1}{\bbL^{\dim X +1} -1},\]
a statement that is made precise using moment generating functions.   
\end{enumerate}

In hindsight, the common theme of these two generalizations is that the relevant conditions on the Taylor series differ at only a finite set of points from the condition of being smooth (though that finite set of points is allowed to vary, so that  we can ``add up" probabilities across different configurations). In particular, the relevant probabilities in the arithmetic analogs are given by taking a special value of a zeta function multiplied by a polynomial, so that they have natural motivic analogs given by multiplying the Kapranov zeta function by a polynomial (and also ``adding up" over the varying sets of points, which is more complicated in the motivic setting where we must integrate over a configuration space rather than a finite set of configurations). 

\subsubsection{Motivic Euler products and motivic statistics}

One of the main obstructions to extending Vakil-Wood's work to give motivic analogs of other results in arithmetic statistics is the lack of candidate motivic probabilities when the Euler product on the arithmetic side cannot be directly related to a zeta function. 

In this work, we remove that obstruction by using the theory of motivic Euler products developed by the first author in her work on motivic height zeta functions \cite{bilu:thesis, bilu:motiviceulerproducts}. Combining this theory with a careful formalization of the motivic inclusion-exclusion method introduced by Vakil-Wood and some insights ported from probability theory using motivic random variables, we are able to prove motivic Bertini theorems that allow for very general Taylor conditions (cf. Theorems \ref{theorem.taylorconditions} and \ref{maintheorem.general} below).

As applications, we obtain motivic analogs of essentially all known results in arithmetic statistics that have been proven using Poonen's sieve. Moreover, this proliferation of results in motivic statistics complementing the corresponding results in arithmetic statistics leads us to a conjecture on the analytic convergence of zeta functions that unifies motivic and arithmetic statistics over finite fields. 

\begin{remark}
In the second author's work \cite{howe:mrv2} on the motivic statistics of the universal transverse hypersurface section, motivic Euler products already play a key role in the proof of the main theorem. However, they appear only in an ad hoc role, and the proof given in \cite{howe:mrv2} is divorced from any conceptual understanding of the result. The approach of the present work offers a substantial conceptual improvement, and in a sequel \cite{bilu-howe:motivic-random-variables} we will use these insights to clarify and expand the theory of motivic random variables introduced in \cite{howe:mrv1, howe:mrv2}.
\end{remark}

\subsubsection{Significance}
Our work shows that the study of motivic statistics is not merely a curiosity existing in a few special cases, but rather a subject of breadth comparable to that of arithmetic statistics and worthy of study in its own right. We provide strong tools to carry out that study, and examples indicating how they can be used. These tools have already been applied in exciting new ways that we did not predict: for example, in his thesis, Ljungberg \cite{ljungberg:thesis} uses our results to prove some cases of the motivic Tamagawa number conjecture. Finally, our conjecture relating motivic and arithmetic statistics provides a concrete path towards a deeper unity of these two subjects which so far have only been related by analogy.  

\subsection{Results}\label{subsec.results}

\subsubsection{Poonen's Bertini theorem with Taylor conditions.}
Let $X \subset \bbP^n_{\bbF_q}$ be a subvariety. As explained in the previous section, the condition that a hypersurface in $\bbP^n_{\bbF_q}$ intersects $X$ transversely can be described by prescribing the allowable first-order Taylor expansions for a defining equation at each point. It is thus natural to instead consider an arbitrary set of allowable Taylor expansions at each point of $\bbP^n_{\bbF_q}$, and then ask for the probability that a hypersurface (or rather its defining equation) has an allowable Taylor expansion at every point. 

In this more general setup, one can still compute a candidate probability as an infinite product of the local densities, and the general version of Poonen's Bertini theorem \cite[Theorem 1.2]{poonen:bertini} says that, under some reasonable assumptions about the set of allowable Taylor expansions (for example, the infinite product should converge!), the asymptotic probability that a hypersurface of degree $d$ has an allowable Taylor expansion at each point of $\bbP^n_{\bbF_q}$ is indeed equal to this candidate probability.

Our main result is a motivic analog of this more general theorem. To state it, we first introduce some notation. 

\subsubsection{Taylor conditions}
\newcommand{\Taylor}{\mathrm{Taylor}}
Let $V_d := \Gamma(\bbP^n, \calO(d))$ be the space of degree $d$ homogeneous polynomials in $n+1$ variables. Let $\calP^r_d / \bbP^n_{\bbC}$ be the vector bundle of $r$-principal parts of $\calO(d)$: at a point $x \in \bbP^n(\bbC)$, the fiber $\calP^r_d|_x$ is canonically identified with the the set $\calO(d)_x / \mathfrak{m}_x^{r+1}$ of $r$-infinitesimal expansions of sections of $\calO(d)$.  Restriction of global sections to infinitesimal neighborhoods induces a natural Taylor expansion map of vector bundles over $\bbP^n$,
\[ \Taylor^r: V_d \times \bbP^n \rightarrow \calP^r_d. \]

\newcommand{\everywhere}{\mathrm{everywhere}}
\begin{definition}
An \emph{($r$-infinitesimal) Taylor condition (on $\calO(d)$)} is a constructible subset $T \subset \calP^r_d$.  
\end{definition} 
\noindent Given a Taylor condition $T$, we define the subset $V_{d}^{T-\everywhere} \subset V_d$ of equations satisfying $T$ everywhere by
\[ V_{d}^{T-\everywhere} := \{ F \in V_d(\bbC)\, | \, \forall x \in \bbP^n(\bbC),\, \Taylor^r(F,x)\in T\}. \]
An elementary argument with Chevalley's theorem shows that $V_{d}^{T-\everywhere}$ is a constructible subset of $V_d$. In particular, it has a well-defined class 
\[ [V_{d}^{T-\everywhere}] \in K_0(\Var/\bbC). \]
Thus it makes sense to compute the motivic probability that an element of $V_d$ satisifies $T$ at every point as 
\[ \frac{[V_{d}^{T-\everywhere}]}{[V_d]} \in \calM_{\bbC}. \]

\subsubsection{Motivic Euler products}
If we fix an $x \in \bbP^n(\bbC)$, it is easy to see that the motivic probability that $F \in V_d$ satisfies $T$ at $x$ is given by 
\[ \frac{[T]_x}{[\calP^r_d]_x}\, \left( = \frac{[T|_x]}{[\calP^r_d |_x]} \right). \]
However, there is no obvious way in which to form the uncountable product
\[ `` \prod_{x \in \bbP^n(\bbC)} \frac{[T]_x}{[\calP^r_d]_x}. " \]
We overcome this difficulty by using the theory of motivic Euler products, as introduced by the first author \cite{bilu:thesis, bilu:motiviceulerproducts}. Briefly, for $X/\bbC$ a variety and $\{a_i\}_{i \in \bbZ_{\geq 1}}$ a set of classes in the relative Grothendieck ring $K_0(\Var/X)$ (or its localization $\calM_X := K_0(\Var/X)[\bbL^{-1}]$), motivic Euler products give a systematic way to make sense of the expression
\[ \prod_{x \in X} \left(1 + (a_1)_x  t + (a_2)_x t^2 + \ldots \right)\]
as an element of $ 1+ tK_0(\Var/\bbC)[[t]]$ in a way that behaves consistently with the notation -- in other words, one can make most of the formal manipulations one would expect for products.

\begin{definition}\label{def:intro-m-admissible}
A Taylor condition $T$ is $M$-\emph{admissible} if there is a finite subset $S \subset \bbP^n(\bbC)$ with $|S| \leq M$ such that
\[ T|_{\bbP^n \backslash S} \]
has codimension at least $n+1$ in $\calP^r_d|_{\bbP^n \backslash S}$. A Taylor condition $T$ is \emph{admissible} if it is $M$-admissible for some $M \geq 0$.
\end{definition}

The importance of this notion is justified by
\begin{lemma}\label{lemma:intro-admissible-converge} If $T$ is an admissible Taylor condition, then the motivic Euler product 
\begin{equation}\label{eqn.mot-prob-product} \prod_{x \in \bbP^n} \left(1 - \frac{[T]_x}{[\calP^r_d]_x} t \right)\end{equation}
converges at $t=1$ in $\widehat{\calM}_\bbC.$
\end{lemma}

In particular, suppose $T$ is a Taylor condition such that the complementary condition $T^c := P_{d,s} \backslash T$ is admissible. We then make the nonsense formal manipulation
\begin{align*}
 \left.\prod_{x \in \bbP^n}\left( 1 - \frac{[T^c]_x}{[\calP^r_d]_x} t\right)\right|_{t=1} & ``= \prod_{x \in \bbP^n}\left(1 - \frac{[T^c]_x}{[\calP^r_d]_x} \right) "\\
 & `` = \prod_{x \in \bbP^n} \left(1 - \frac{[\calP^r_d]_x - [T]_x}{[\calP^r_d]_x}\right) " \\
 & ``= \prod_{x \in \bbP^n} \frac{[T]_x}{[\calP^r_d]_x}.  "\end{align*}
Here the quotation marks are placed to indicate that there is no actual mathematical sense to the manipulations because the infinite products without a variable have no meaning in this context. Nonetheless, these naive manipulations combined with Lemma \ref{lemma:intro-admissible-converge} suggest that, for $T$ a Taylor condition with admissible complement, 
\[ \left.\prod_{x \in \bbP^n} \left(1 - \frac{[T^c]_x}{[\calP^r_d]_x} t\right)\right|_{t=1} \]
can be thought of as the product of the local densities $\frac{[T]_x}{[\calP^r_d]_x}$ over all points $x \in \bbP^n$. 

\subsubsection{A motivic Bertini theorem with Taylor conditions}
We can now state our motivic analog of Poonen's finite field Bertini theorem with Taylor conditions. Because there is no canonical way to identify Taylor conditions on $\calO(d)$ for varying~$d$, we state our theorem using asymptotic ``big $O$" notation rather than a limit. 

\begin{maintheorem}\label{theorem.taylorconditions} Fix $r$, $n$, and $M$. Then, as $T$ ranges over all $r$-infinitesimal Taylor conditions with $M$-admissible complement on the line bundles $\calO(d)$ on $\bbP^n$, 
\[  \frac{[V_{d}^{T-\everywhere}]}{[V_d]} = \left.\prod_{x \in \bbP^n} \left(1 - \frac{[T^c]_x}{[\calP^r_d]_x} t\right)\right|_{t=1} + O\left(\bbL^{-\frac{d}{r+1}}\right) \textrm { in } \widehat{\calM}_\bbC. \]
\end{maintheorem} 

\noindent Concretely, this means that there is some $N$ such that for any $d$ and any $r$-infinitesimal Taylor condition $T$ on $\calO(d)$ with $M$-admissible complement, 
\[ \frac{[V_{d}^{T-\everywhere}]}{[V_d]} \equiv \left.\prod_{x \in \bbP^n}\left( 1 - \frac{[T^c]_x}{[\calP^r_d]_x} t\right)\right|_{t=1} \mod \Fil_{N - \floor{\frac{d}{r+1}}},\]
where $\Fil_k$ denotes the subgroup of $\widehat{\calM}_\bbC$ of elements of dimension $\leq k$. 

\begin{remark}\label{remark:bound-in-big-o} In fact, here one can take $N=M$. \end{remark}

\begin{example}\label{example.poonen-taylor-conditions} Fix a finite set of locally closed smooth subvarieties $X_i \subset \bbP^n$ and a finite $S \subset \bbP^n(\bbC)$. Then, if $T$ is a Taylor condition such that for all $x$ in $\bbP^n(\bbC) \backslash S$, $T_x$ contains the Taylor expansions that are transverse to each of the $X_i$ at $x$, then $T^c$ is $|S|$-admissible. Thus, our theorem applies to the natural analogs of the Taylor conditions considered by Poonen in \cite[Theorem 1.2]{poonen:bertini}, except that we only work at a fixed finite level of the Taylor expansion and require that $T$ be constructible. The condition of constructibility cannot be weakened in our setup, however, one could likely remove the finite level condition by passing to an inverse limit. However, as all of the applications we have in mind can be described at a finite level of the Taylor expansion, we leave such a generalization to the motivated reader. 
\end{example}

\subsubsection{Generalization}
For applications, one often needs a much more general statement: for example, to prove a semi-ample Bertini theorem analogous to the finite field result of \cite{erman-wood:semiample-bertini}, one needs to handle sections of coherent sheaves. To ``add up" over conditions, e.g. for applications to motivic random variables as in \cite{howe:mrv2} or in reproving Vakil-Wood's result on hypersurface sections singular at exactly $m$ points, one must allow an arbitrary base variety $S$. To compare directly with point counting results one must allow for varieties over a field of characteristic $p$. 

To these ends, we prove a more general version of Theorem \ref{theorem.taylorconditions} where
\begin{enumerate}
\item $\bbP^n$ is replaced with an arbitrary projective variety $X$ over an arbitrary base variety $S$ over an arbitrary field $K$.
\item $\calO(d)$ is replaced with $\calF(d):=\calF \otimes \calL^d$ where $\calF$ is a coherent sheaf on $X$ and $\calL$ is a relatively ample line bundle on $X/S$. 
\item The Grothendieck ring is replaced with a better behaved quotient in characteristic $p$ (note that this is already necessary even for Vakil-Wood's result: although the results of \cite{vakil-wood:discriminants} are claimed over any field, there is a gap in the proof in positive characteristic -- cf. Remark \ref{remark:vakil-wood-error} below for more details).
\end{enumerate}

It is not a priori clear how one should formulate such a generalization, because a coherent sheaf is not in general represented by a variety. To overcome this difficulty, we first introduce a formalism that attaches to a coherent sheaf $\calA$ on a variety $Y$ a geometric realization $\bbV(\calA) / Y$, which is a disjoint union of geometric vector bundles over a stratification of $Y$. 

This geometric realization gives a functor from coherent sheaves on $Y$ to the localization of the category of varieties over $Y$ along piecewise isomorphisms ${Y'\rightarrow Y}$. It turns out that this localization is a reasonable category for us to work in: it preserves the notion of constructible sets, points, etc., and passage to the Grothendieck ring of varieties over $Y$ factors through the localization functor.

Letting $\calP^r_{/S} \calF$ denote the sheaf of $r$-principal parts (i.e. $r$-infinitesimal expansions) of $\calF$, we obtain a well-defined notion of an $r$-infinitesimal Taylor condition on $\calF$ as a constructible subset of $\bbV(\calP^r_{/S} \calF)$. We are also able to formulate a good notion of admissibility in this context (Definition \ref{def:m-admissible}). Moreover, denoting by $f:X\to S$ the structural morphism, the space $\bbV(f_* \calF(d))$, a variety over $S$, is a reasonable space of global sections, and the setup is sufficiently robust so that a simple argument with Chevalley's theorem still implies that the subset 
\[ \bbV(f_* \calF(d))^{T-\mathrm{everywhere}} \subset \bbV(f_* \calF(d)) \] 
of sections satisfying the Taylor condition $T$ everywhere is constructible. We obtain

\begin{maintheorem}\label{maintheorem.general}
Fix $f:X \rightarrow S$, a map of varieties over a field $K$, $\calF$ a coherent sheaf on $X$, $\calL$ a relatively ample line bundle on $X$, and $r, M \geq 0$. Then, there is an $\epsilon >0$ such that as $T$ ranges over all $r$-infinitesimal Taylor conditions on $\calF(d)=\calF \otimes \calL^d$ with $M$-admissible complement, 
\[ \frac{[\bbV\left(f_* \calF(d)\right)^{T-\mathrm{everywhere}}]}{[\bbV(f_* \calF(d))]} = \left. \prod_{x \in X/S} \left(1 - \frac{[T^c]_x}{[\bbV(\calP^r_{/S}\calF)]_x}\right)\right|_{t=1} + O(\bbL^{-\epsilon d}) \textrm{ in } \widehat{\widetilde{\calM}}_X.\]
\end{maintheorem}

When $K$ is of characteristic zero, the Grothendieck ring $\widehat{\widetilde{\calM}}_X$ appearing in Theorem \ref{maintheorem.general} is equal to $\widehat{\calM}_X$, the completion of $\calM_X = K_0(\Var/X)[\bbL^{-1}]$ for the dimension filtration. In positive characteristic, it is a modification obtained by first taking the quotient $\widetilde{K}_0(\Var/X)$ of $K_0(\Var/X)$ by radicial surjective maps before inverting $\bbL$ and completing. 

\begin{remark}
If $\calL^A$ is very ample on each fiber, then one can take $\epsilon = \frac{1}{A(r+1)}$ in Theorem \ref{maintheorem.general}. If $\calF(D)$ is globally generated then one can also obtain an explicit bound on the constant appearing in the big O notation in terms of $D$ as in Remark~\ref{remark:bound-in-big-o}. 
\end{remark}

\begin{remark} The main ingredient in the proof of Theorem \ref{maintheorem.general}, Theorem \ref{theorem.general}, can also be applied to more general spaces of global sections; this allows us to treat, e.g., motivic probability problems over varieties that are not proper.
\end{remark}

\subsubsection{Applications}
In section \ref{sect.applications}, we give the following applications of Theorems~\ref{theorem.taylorconditions}, \ref{maintheorem.general}, and~\ref{theorem.general}:
\begin{enumerate}
\item We compute the asymptotic motivic probability that a hypersurface section of a smooth projective variety has exactly $m$ singular points. This result is due to Vakil-Wood \cite{vakil-wood:discriminants}, but our approach significantly clarifies the argument, bringing it much closer to the simple proof of the arithmetic analog, due to Gunther \cite{gunther:random-hypersurfaces}, via Poonen's Bertini theorem with Taylor conditions. 
\item We compute the asymptotic motivic probability that a complete intersection in $\bbP^n$ is smooth, giving a motivic analog of the arithmetic result of Bucur-Kedlaya \cite{bucur-kedlaya:complete-intersections}. In this case, numerical computations of Bucur-Kedlaya suggest that the probability in the arithmetic case cannot be expressed using special values of zeta functions, so that a motivic analog seemed inaccessible before our work.
\item We compute the asymptotic motivic probability that a surface containing a fixed curve in $\bbP^3$ is smooth, giving a motivic analog of the arithmetic result of Gunther \cite{gunther:random-hypersurfaces}.
\item We compute the asymptotic motivic probability that a map from a direct sum of line bundles to a fixed vector bundle on a curve is a surjection. This computation plays an important role in work of Ljungberg \cite{ljungberg:thesis} proving cases of the motivic Tamagawa number conjecture.
\item We prove a motivic semi-ample Bertini theorem, giving a motivic analog of the arithmetic result of \cite{erman-wood:semiample-bertini}.

\end{enumerate}
Our results can be used to prove more general versions of {(1)-(4)}, but in the interest of clarity we have preferred to focus on concrete examples.

Theorem \ref{maintheorem.general} can also be used to give a new proof of the main result on motivic random variables in \cite{howe:mrv2}; this will be carried out in \cite{bilu-howe:motivic-random-variables}, where we also clarify and expand the theory of motivic random variables developed by the second author in \cite{howe:mrv1, howe:mrv2}. Similar to our new proof of Vakil-Wood's theorem on hypersurface sections with exactly $m$ singular points, this new proof is much closer to the simple argument for the arithmetic analog \cite[Theorem C]{howe:mrv2} given by adding up applications of Poonen's Bertini theorem with Taylor conditions. We will also use Theorem~\ref{maintheorem.general} in \cite{bilu-howe:motivic-random-variables} to give new examples of interesting motivic random variables. 

\subsection{Questions}
Our work leads to some natural questions:
\begin{enumerate}
\item After making a suitable modification to the Grothendieck ring (as in Theorem \ref{maintheorem.general}), Theorem \ref{theorem.taylorconditions} also holds over a finite field $\bbF_q$. Thus, over $\bbF_q$ we have two stabilization theorems: Poonen's arithmetic theorem for point-counts and our motivic theorem in the Grothendieck ring. Neither theorem implies the other, however, it is natural to ask if there is a stronger result that implies both. In Section~\ref{sect.conjecture} below, we formulate a conjecture about the analytic convergence of a sequence of zeta functions that would accomplish this unification (in the interest of brevity, we state the conjecture only in the case of smooth hypersurface sections, but it can be generalized).  
\item Can one prove an arithmetic analog of Theorem \ref{maintheorem.general} (even over $S=\Spec \bbF_q$)? It seems likely to the authors that the answer is yes, and doing so could clarify the literature by neatly packaging arguments involving Poonen's sieve, which are often implemented ad hoc. 
\item As indicated in Example \ref{example.poonen-taylor-conditions}, it should be possible to replace the finite level Taylor conditions in Theorem \ref{theorem.taylorconditions} by working carefully with the inverse limit of the bundles of principal parts at finite level. Are there interesting questions to ask in motivic statistics that require this level of generality? Is the formalism used for studying motivic integration on arc-schemes sufficient to carry out such a generalization, or will new tools be needed? A precise formulation could be interesting, even without applications in mind. 
\end{enumerate}

\subsection{Outline}
In section \ref{section.grothendieck-ring} we introduce our notation for Grothendieck rings, along with some convenient intermediary localizations of the category of varieties that can be used to facilitate arguments involving repeated passage to stratifications or computations with geometric points. 

In section \ref{section.configuration-spaces}, we summarize the basic properties and notation we will need for dealing with configuration spaces. We note that this work makes extensive use of labeled configuration spaces, and thus our notation needs to accommodate these concisely and without confusion -- this leads us to some non-standard choices.

In section \ref{section.geometric-realization} we construct our geometric realization functor for coherent sheaves and establish some basic structural results for its image. 

In section \ref{section.taylor-expansions} we recall and extend some basic results on sheaves of principal parts and Taylor expansion maps for coherent sheaves; in particular, we explain how Taylor expansion maps interact with configuration spaces and geometric realizations.

In section \ref{sect.motivic-euler-products} we summarize the construction of and basic results on motivic Euler products following \cite{bilu:motiviceulerproducts}, and provide some complements that will be necessary in the present work. 

In section \ref{sect.motivic-inclusion-exclusion} we formalize and carefully explain the motivic inclusion-exclusion principle first introduced in a special case by Vakil-Wood \cite{vakil-wood:discriminants}. We give two proofs of the principle; one following \cite{vakil-wood:discriminants}, and the other using motivic Euler products. 

In section \ref{section.indepedent-events} we develop some basic results in motivic probability theory for studying $m$-wise (e.g. pairwise, triplewise, \ldots) independent events -- this is the key conceptual input from probability theory used in the proof of Theorems \ref{theorem.taylorconditions}~and~\ref{maintheorem.general}.

In section \ref{section.motivic-prob-coherent-sheaves} we prove our main theorems. To do so, we first give a general formalism for motivic probability problems defined by maps of coherent sheaves, and prove an approximation result (Theorem \ref{theorem:main-theorem-coherent}) for these problems when the defining condition has admissible complement and the sheaf of global sections is sufficiently generating. We then apply this result to Taylor expansion maps to obtain Theorems \ref{theorem.general}, \ref{theorem.taylorconditions}, and  \ref{maintheorem.general}.  

In section \ref{sect.applications} we give specific applications to motivic statistics, and in section~\ref{sect.conjecture} we explain our conjecture on the analytic convergence of zeta functions unifying motivic and arithmetic statistics. 

\subsection{Acknowledgements}
We thank Antoine Chambert-Loir, Ronno Das, Benson Farb,  Ben  Fayyazuddin Ljungberg, Javier Fres\'an, Aaron Landesman, Daniel Litt and Ravi Vakil  for helpful conversations. 

The second author was supported during the preparation of this work by the National Science Foundation under Award No. DMS-1704005 and by the National Science Foundation under Grant No. DMS-1440140 while in residence at the Mathematical Sciences Research Institute in Berkeley, California during Spring 2019. 

\section{The Grothendieck ring of varieties}\label{section.grothendieck-ring}

A \emph{variety} over a field $K$ is a reduced, separated $K$-scheme of finite type (we do not require that it be irreducible or connected). We denote by $\Var/K$ the category of varieties over $K$, and if $S$ is a variety over $K$, we denote by $\Var/S$ the slice category of varieties over $S$, whose objects are morphisms $X\rightarrow S$ in $\Var/K$.

The Grothendieck ring of varieties over $S$, $K_0(\Var/S)$, is defined at the level of groups as the free abelian group on the isomorphism classes $[X/S]$ of $X/S \in \Var/S$, modulo the relations 
\[ [X/S] = [Z/S] + \left[ \left(X\bs Z\right) /S \right] \]
whenever $Z\subset X$ is a closed subvariety. The ring multiplication is induced by
\[ [X_1] \cdot [X_2] = [(X_1 \times_S X_2)_\red].\]

Given an object $X/S \in \Var/S$, it follows that if $X=\bigsqcup X_i$ is a decomposition of~$X$ into locally closed subvarieties (here the disjoint union is of sets), then 
\[ [X/S]=\sum [X_i/S] \in K_0(\Var/S).\]
In particular, in making arguments in the Grothendieck ring, one is often free to pass to an arbitrary decomposition. However, some arguments may involve several steps of passing to finer and finer decompositions while still working with actual varieties before finally passing to the Grothendieck ring where the different decompositions become equivalent. To make this kind of argument cleanly and carefully, it is thus preferable to introduce an intermediate step between $\Var/S$ and $K_0(\Var/S)$. 

This can be accomplished by localizing the category $\Var/S$ at piecewise isomorphisms to obtain a new category $(\Var/S)_{\PW}$, such that the natural localization functor induces an isomorphism  
\[ K_0(\Var/S) = K_0( (\Var/S)_{\PW} ). \] 

\begin{remark} As we will see below, the definition of $K_0((\Var/S)_\PW)$ is simpler than that of $K_0(\Var/S)$: one only needs to quotient by relations coming from disjoint unions in $(\Var/S)_\PW$. \end{remark}

If $K$ is of characteristic zero, then $f: X \rightarrow Y$ is a piecewise isomorphism if and only if $f$ induces a bijection on $L$-points for any algebraically closed $L/K$. If $f$ satisfies the latter property, we also say it is radicial surjective, so that in characteristic zero being a piecewise isomorphism is equivalent to being radicial surjective. In characteristic $p$, however, being radicial surjective is weaker than being a piecewise isomorphism -- for example, the Frobenius map $x \mapsto x^2$ from $\bbA^1_{\bbF_2} \rightarrow \bbA^1_{\bbF_2}$ is radicial surjective but not a piecewise isomorphism. 

It turns out that in characteristic $p$, in order to obtain a Grothendieck ring that is well-behaved for our purposes, it will be preferable to instead localize at all radicial surjective morphisms to obtain the category $(\Var/S)_{\RS}$. We will write \[ \widetilde{K}_0(\Var/S) := K_0( (\Var/S)_{\RS} ). \]
In characteristic zero, $\widetilde{K}_0(\Var/S)=K_0(\Var/S)$, and in general the kernel of the natural map $K_0(\Var/S)\rightarrow \widetilde{K}_0(\Var/S)$ is generated by the relations $[X/S]=[X'/S]$ whenever there is a radicial surjective map $f:X/S \rightarrow X'/S$ in $\Var/S$. In this guise, $\widetilde{K}_0(\Var/S)$ was introduced by Mustata \cite{mustata:zeta} and also studied by Chambert-Loir-Nicaise-Sebag \cite{chambert-loir-et-al:motivic-integration}, who gave an equivalent formulation (cf. Remark \ref{remark:K0-rs-other-definitions} below). 

\begin{remark}
Roughly speaking, the reason we must localize at radicial surjective morphisms is that it allows us to make combinatorial arguments with configuration spaces using only their geometric points, which are easily understood. This is a subtle but serious point -- for example, there is a gap in the proof of \cite[Theorem 1.13]{vakil-wood:discriminants} when $K$ is of characteristic $p$, because the motivic inclusion-exclusion principle can fail in $K_0(\Var)$ in the presence of inseparable extensions of residue fields -- specifically, equation (3.1) of \cite{vakil-wood:discriminants} is not obviously valid. The gap is removed by passing to the quotient $\widetilde{K}_0(\Var/S)$; cf. section \ref{sect.motivic-inclusion-exclusion} below for more details on this point and the failure of motivic inclusion-exclusion in characteristic $p$. 
\end{remark}

In the remainder of the section, we will carefully define $(\Var/S)_\PW$ and $(\Var/S)_\RS$ and study their Grothendieck rings; we point out here that the family of morphisms being inverted is well-behaved in both cases, so that morphisms in the localized category are given by spans. We also introduce a third variant, $(\Var/S)_{S-\PW}$, where we localize only at piecewise isomorphisms induced by decompositions of $S$. This is useful for studying the classes in Grothendieck rings defined by constructible subsets, and will also be used in Section \ref{section.geometric-realization} below when we define a functor $\bbV$ from coherent sheaves on a variety $S$ to $(\Var/S)_{S-\PW}$ generalizing the construction attaching the representing geometric vector bundle to a locally free sheaf. 

\subsection{Piecewise isomorphisms and radicial surjective morphisms}

\begin{definition}\label{def:piecewiseisoandradicial} A morphism $f:X \rightarrow Y$ of varieties over $K$ is
\begin{enumerate} 
\item a \emph{piecewise isomorphism} if there exist decompositions $X=\bigsqcup X_i$ and $Y=\bigsqcup Y_i$  into locally closed subvarieties such that $f$ induces an isomorphism between $X_i$ and $Y_i$ for each $i$.
\item \emph{radicial surjective} if for every algebraically closed field $L/K$, the induced map $f: X(L) \rightarrow Y(L)$ 
is a bijection. 
\end{enumerate}
\end{definition}

\begin{remark}\label{remark:radicial-injective} In (2), if we replace ``bijection'' by ``injection'', we get the notion of radicial morphism. A morphism $f:X\to Y$ is radicial surjective (in the sense of definition \ref{def:piecewiseisoandradicial}) if and only if it is both radicial and surjective. 
\end{remark}

\begin{proposition}\hfill
\begin{enumerate}
\item A piecewise isomorphism is radicial surjective. If ${\mathrm{char} K= 0}$, a radicial surjective morphism is a piecewise isomorphism. 
\item A morphism $f:X \rightarrow Y$ is radicial surjective if and only if it is bijective and for any $x \in X$, the extension of residue fields $\kappa(x)/\kappa(f(x))$ is purely inseparable. 
\end{enumerate}
\end{proposition}
\begin{proof} 
\begin{enumerate} \item The first statement is clear. For the second statement, see \cite[Chapter 2, Proposition 1.4.11]{chambert-loir-et-al:motivic-integration}.
\item Use \cite[\href{https://stacks.math.columbia.edu/tag/01S2}{Tag 01S2}]{stacks-project}, more precisely the equivalence between  (1) and (3) in Lemma 28.10.2, together with Remark \ref{remark:radicial-injective}.
\end{enumerate}

\end{proof}

\begin{remark}\label{remark:rad-surj-factorization} 
Recall that a map of varieties over $K$ is a universal homeomorphism if and only if it is radicial, surjective, and finite (see \cite[Corollaire (18.12.11)]{EGAIVd}). Given a radicial surjective map of varieties $f:X \rightarrow Y$, one can thus find a piecewise isomorphism $Y' \rightarrow Y$ such that the induced map 
\[ X \times_Y Y' \rightarrow Y' \]
is a universal homeomorphism -- indeed, a radicial surjective map of varieties is quasi-finite, so it suffices to choose a stratification $Y'$ along which it becomes finite. 
\end{remark}

The following lemma ensures that morphisms in the localized category have a simple description. Here we follow the notation of  \cite[\href{https://stacks.math.columbia.edu/tag/04VB}{Section 04VB}]{stacks-project}.

\begin{lemma}\label{lemma:calculus-of-fractions}
The classes $\PW$ of piecewise isomorphisms and $\RS$ of radicial surjective morphisms in $\Var/S$ are right multiplicative systems (cf. \cite[Definition 04VC]{stacks-project}).
\end{lemma}
\begin{proof}
It is clear that isomorphisms are in $\PW$ (resp. $\RS$), and the composition of two morphisms in $\PW$ (resp.~$\RS$) is again in $\PW$ (resp.~$\RS$), so that $\PW$ (resp. $\RS$) satisfies condition RMS1 of  \cite[\href{https://stacks.math.columbia.edu/tag/04VC}{Definition 04VC}]{stacks-project}. 

Condition RMS2 is verified because given $f: X \rightarrow Z$ in $\PW$ (resp.~$\RS$) and an arbitrary morphism $Y \rightarrow Z$, the map $(X \times_Z Y)_\red \rightarrow Y$ is in $\PW$ (resp.~$\RS$). 

Condition RMS3 is verified because given $v: Y \rightarrow Z$ in $\PW$ (resp. $\RS$) and $f, g: X \rightarrow Y$ such that $v \circ g = v \circ f$, we have $f = g$. To show this, it suffices to check the case $v \in \RS$. We first observe that $f=g$ on the underlying topological space since $v$ is a homeomorphism. Thus, if $\eta$ is a generic point of an irreducible component of $X$, $f(\eta)=g(\eta)$, and we write the common images as $\eta'$. Then, to show $f=g$, it suffices to verify the maps $f^*, g^*: \kappa(\eta') \rightarrow \kappa(\eta)$ agree for any such $\eta$. Writing $v^*: \kappa(v(\eta'))\rightarrow \kappa(\eta')$, we are given $f^* \circ v^* = g^* \circ v^*$. The field extension $\kappa(\eta')/\kappa(v(\eta'))$ is purely inseparable (because $v$ is radicial surjective), and any map of fields with domain $\kappa(v(\eta'))$ admits at most one extension to $\kappa(\eta')$, so we conclude $f^*=g^*$. 
\end{proof}

\subsection{The localized categories}

For $W$ one of $\PW$ or $\RS$, we denote by $(\Var/S)_W$ the localized category. Because $W$ has a calculus of right fractions, it admits the following concrete description (\cite[Tag 04VK]{stacks-project}: the objects of $(\Var/S)_W$ are the same as those of $(\Var/S)$, and the maps are given by
\[ \Map_{(\Var/S)_W}(X,Y) = \colim_{f: X' \rightarrow X \textrm { in } W} \Map_{\Var/S}(X', Y). \]
That is, every morphism $X \rightarrow Y$ is represented by a span or ``hat":
\[ \xymatrix{ & X ' \ar[dl]_{f} \ar[dr]^{g} & \\
X & & Y } \]
with $f \in W$ and $g$ arbitrary; the calculus of fractions allows us to compose spans. 

\begin{remark}
When $K$ is perfect, Liu-Sebag have constructed a natural fully faithful embedding $X \mapsto X^\cons$ of $(\Var/K)_\PW$ in the category of zero-dimensional reduced schemes over $K$. The underlying set of $X^\cons$ is the same as $X$, but it is equipped with the constructible topology.
\end{remark}

Because $\PW \subset \RS$, the natural localization functor $\Var/S \rightarrow (\Var/S)_\RS$ factors through $(\Var/S)_\PW \rightarrow (\Var/S)_\RS.$ 

We will also consider the category $(\Var/S)_{S-\PW}$, where we localize only at piecewise isomorphisms induced by decompositions of $S$. Concretely: the objects of $(\Var/S)_{S-\PW}$ are varieties over $S$, and, if we write $S' \sim S$ to denote a decomposition $S' = \bigsqcup S_i$ of $S$ into locally closed subvarieties $S_i$,  
\[ \Map_{(\Var/S)_{S-\PW}}(X, Y) = \colim_{S' \sim S} \Map_{\Var/S}(X_{S'}, Y). \]

\subsection{Constructible subsets}
Let $X/K$ be a variety, and let $C$ be a constructible subset of the underlying topological space of $X$. Then, we can write $C$ as a set-theoretic disjoint union of locally closed subsets $C=\bigsqcup C_i$, and thus we can equip $C$ with the structure of a variety over $X$ by letting each $C_i$ have its induced subvariety structure and letting $C$ be the disjoint union of these varieties.  This structure depends on the chosen decomposition, however, if we consider the resulting object of $(\Var/X)_{X-\PW}$, there is a unique isomorphism between the objects resulting from two different choices. 

Thus, it makes sense to speak of ``the" object of $(\Var/X)_{X-\PW}$ defined by the constructible subset $C$. If $X$ is a variety over $S$, then by composition with the natural map we also obtain an object $C/S$ of $(\Var/S)_W$ where $W=\PW \textrm{ or } \RS$, determined up to canonical isomorphism.  

\begin{remark}[Chevalley's theorem for morphisms in the localized category]
Given varieties $X$ and $Y$ over $K$, a radicial surjective map ${f:X \rightarrow Y}$ induces a bijection $f:|X|\rightarrow |Y|$ of underlying scheme-theoretic sets. Thus, the functor sending a variety to its underlying scheme-theoretic set factors through $(\Var/K)_\RS$. In particular, given a subset $A \subset |X|$, and $f \in \Map_{(\Var/K)_{\RS}}(X,Y)$, it makes sense to speak of $f(A) \subset |Y|$. 

Moreover, Chevalley's theorem holds in this context, so that if $A$ is constructible then $f(A)$ is constructible: Indeed, if we write $X^\cons$ for the underlying set equipped with the constructible topology, then a map $f:X \rightarrow Y$ of varieties over $K$ induces a continuous open map (the latter by Chevalley's theorem) $X^\cons \rightarrow Y^\cons$. If $f$ is radicial surjective then, as indicated above, the map on sets is also bijective, so a radicial surjective map induces a homeomorphism $f:X^\cons \rightarrow Y^\cons$. Thus, 
\[ \Var/K \rightarrow \Top,\; X \mapsto X^\cons \]
factors through $(\Var/K)_{\RS}$. Because for a variety $X$, the constructible subsets of $X$ are the quasi-compact opens of $X^\cons$, we find that Chevalley's theorem holds for an element $f \in \Map_{(\Var/K)_{\RS}}(X,Y)$: such an $f$ induces an open map ${X^\cons \rightarrow Y^\cons}$, and thus sends quasi-compact opens to quasi-compact opens. Similarly, Chevalley's theorem holds for morphisms in $(\Var/K)_{\PW}$, $(\Var/S)_\RS$, et cetera. 
\end{remark} 
\subsection{Pointwise arguments}
\newcommand{\Set}{\mathrm{Set}}

If $f: X\rightarrow Y$ is a piecewise isomorphism, then for any field $L/K$, 
\[ f: X(L) \rightarrow Y(L) \]
is a bijection. Thus, there is a functor of $L$-points
$$\begin{array}{ccc}  (\Var/K)_{\PW}& \rightarrow& \Set\\
 X &\mapsto & X(L). \end{array}$$
Similarly, a radicial surjective $f$ induces a bijection on $L$ points when $L/K$ is algebraically closed, so for $L/K$ algebraically closed, we obtain a functor of $L$-points,
$$\begin{array}{ccc}  (\Var/K)_{\RS}& \rightarrow& \Set\\
 X &\mapsto & X(L). \end{array}$$
factorizing the functor of $L$-points on $(\Var/K)_{\PW}$. 

\subsubsection{Recognition principles}
Almost by definition, isomorphisms in $(\Var/K)_{\RS}$ can be recognized on geometric points (i.e. points in algebraically closed $L/K$):
\begin{lemma}\label{lemma:rec-princ-rs} A map $f \in \Map_{(\Var/K)_\RS}(X,Y)$ is an isomorphism if and only if it induces a bijection on $L$-points for any algebraically closed $L/K$. 
\end{lemma}
For constructible sets, this recognition principle holds already in $(\Var/X)_{X-\PW}$:
\begin{lemma}\label{lemma:rec-princ-cons}
Let $C_1$ and $C_2$ be constructible subsets of $X$. The following are equivalent:
\begin{enumerate}
\item $C_1=C_2$ as subsets of $X$,
\item $C_1 \cong C_2$ in $(\Var/X)_{X-\PW}$, 
\item For every $L/K$, $C_1(L)=C_2(L)$ as subsets of $X(L)$, and
\item For every algebraically closed $L/K$, $C_1(L)=C_2(L)$ as subsets of $X(L)$. 
\end{enumerate}
\end{lemma}

\subsubsection{Useful constructions with geometric points}
The following basic facts about the behavior of geometric points will be helpful to keep in mind for use with the recognition principles of Lemmas \ref{lemma:rec-princ-rs} and \ref{lemma:rec-princ-cons}: 

\begin{lemma}\label{lemma:geometric-points-quotient} Let $X/K$ be a quasi-projective variety and let $G$ be a finite group acting on $X$. Then, the quotient $X/G$ exists, and, for $L/K$ an algebraically closed field, the natural map 
\[ X(L) \rightarrow \left(X/G\right)(L) \]
induces a bijection
\[ X(L)/G \rightarrow \left(X/G\right)(L) \]
\end{lemma}

\begin{lemma}\label{lemma:geometric-points-image} Let $f: X \rightarrow Y$ be a map of $K$-varieties, and consider the constructible set $f(X)$. Then, for any algebraically closed field $L/K$,
\[ f(X)(L) = f(X(L)) \subset Y(L). \]
\end{lemma}

\subsection{Grothendieck rings}
For $W$ one of $\PW$ or $\RS$, we define $K_0\left( (\Var/S)_W \right)$ as follows: additively, it is the free abelian group on isomorphism classes of objects in $(\Var/S)_W$, modulo the relations
\[ [X_1/S \sqcup X_2/S] = [X_1/S] + [X_2/S] \]
where here $\sqcup$ denotes disjoint union (which is the categorical coproduct). The ring structure is again induced by 
\[ [X_1/S]\cdot [X_2/S] = [ (X_1 \times_S X_2)_\red / S ]. \]

Recall that $K_0(\Var/S)$ is the free abelian group on isomorphism classes of varieties over $S$, modulo the relations $[X/S] = [ Z / S] + [X\bs Z / S]$ whenever $Z \subset X$ is a closed subvariety. We define $K_0( (\Var/S)_{S-\PW})$ in the same way, but starting with isomorphism classes in $(\Var/S)_{S-\PW}$. The localization functors induce equalities
\[ K_0(\Var/S)=K_0( (\Var/S)_{S-\PW} ) = K_0( (\Var/S)_{\PW} ). \] 

We will also write $\widetilde{K}_0(\Var/S)$ for $K_0((\Var/S)_\RS)$. The localization functor thus induces a map
\begin{equation}\label{eqn:map-K0-locs} K_0( \Var/S ) \rightarrow \widetilde{K}_0( \Var/S ) 
\end{equation}
which is clearly surjective. 

\begin{lemma}\label{lemma:groth-ring-relations}
The kernel of the morphism (\ref{eqn:map-K0-locs}) is generated by the relations ${[X]=[X']}$ whenever there exists a radicial surjective map $f: X'/S \rightarrow X/S$. In particular, if $K$ is of characteristic zero, then 
\[ K_0( \Var/S ) = \widetilde{K}_0( \Var/S ). \] 
\end{lemma}

\begin{remark}\label{remark:K0-rs-other-definitions}
The definition of $\widetilde{K}_0( \Var/S )$ as a quotient of $K_0(\Var/S)$ was introduced by Mustata \cite[7.2]{mustata:zeta} when $S=\Spec K$. Combined with Remark \ref{remark:rad-surj-factorization}, we see that it is also equal to the quotient $K^{\mathrm{uh}}_0( \Var/S )$ considered by Chambert-Loir, Nicaise and Sebag \cite[Chapter 2, 4.4]{chambert-loir-et-al:motivic-integration}. 
\end{remark}

\begin{remark}\label{remark:RSvsPW}
The map $\Spec \bbF_2(t^{1/2}) \rightarrow \Spec \bbF_2(t)$ is radicial surjective, but not a piecewise isomorphism, so the natural map 
\[ (\Var/\Spec \bbF_2(t))_\PW \rightarrow (\Var / \Spec \bbF_2(t) )_\RS \]
is not an equivalence. Similar examples can be constructed using the Frobenius over any field of positive characteristic (including perfect fields of positive characteristic). 

However, two varieties can have the same class in the Grothendieck ring without being piecewise isomorphic, and it is an open question to determine, e.g., whether $K_0(\Var/\bbF_2(t))$ is equal to $\widetilde{K}_0(\Var/ \bbF_2(t) )$. To our knowledge, no specific example in any Grothendieck ring in characteristic $p$ involving inseparable extensions of residue fields has been settled: for example, we do not know if  $[ \Spec \bbF_2(t^{1/2}) / \Spec \bbF_2(t) ]$ is equal to $1$ in $K_0(\Var/\Spec \bbF_2(t))$ (cf. \cite[Remark 4.4.8]{chambert-loir-et-al:motivic-integration}). 
\end{remark}

\subsubsection{Fiberwise computation in relative Grothendieck rings}

The following will be helpful when it is more convenient to argue over a field: 

\begin{lemma} \label{lemma:injection-relative-product}
 For $S$ a variety over $K$, the natural maps
\begin{align*}
 K_0( \Var/S  ) & \rightarrow \prod_{s \in S} K_0( \Var/ \kappa(s) ), \textrm{ and} \\ 
 \widetilde{K}_0( \Var/S ) & \rightarrow \prod_{s \in S} \widetilde{K}_0\left( \Var/ \kappa(s) \right) 
\end{align*}
are injective.  
\end{lemma}
\begin{proof} If a class $x=\sum_i a_i [X_i/S]$ is in the kernel, then, because simplifying an expression to zero involves only a finite number of relations, we may spread out from generic points and then conclude by Noetherian induction.
\end{proof}

\subsection{Dimension}

Given a map of varieties $f: X \rightarrow S$, and a point $x \in X$, we define
\[ \dim_{/S, x} (X) := \dim_x (X_{f(x)}). \]
and
\[ \dim_{/S} (X) := \max_{x \in X} \dim_{/S,x} (X) = \max_{s \in S} \dim X_s. \]

\begin{lemma}
If $f: X \rightarrow Y$ is a radicial surjective map of varieties over $S$, then $\dim_{/S} X = \dim_{/S} Y$.  
\end{lemma} 
\begin{proof}

Note that relative dimension is a purely topological notion. Thus, invoking \ref{remark:rad-surj-factorization}, we find that it suffices to treat the case where $f$ is a piecewise isomorphism. We may also suppose $S$ is the spectrum of a field, since $f$ induces a piecewise isomorphism on each fiber, and that $Y$ is irreducible. Then, $f$ induces an isomorphism between open subsets of $X$ and $Y$, so $\dim X \geq \dim Y$. On the other hand, it is clear that $\dim X \leq \dim Y$ because any chain of specializations in $X$ is also a chain of specializations in $Y$, and we conclude. 
\end{proof}

In particular, we find that the function $\dim_{/S}$ is well-defined on isomorphism classes in $(\Var/X)_\PW$ and $(\Var/X)_\RS$.

We denote by $\mathcal{M}_S$ the ring $K_0(\Var/S)[\LL^{-1}]$, and $\widetilde{\mathcal{M}}_S$ the ring $\widetilde{K}_0(\Var/S)[\LL^{-1}]$. For every $d\in\bbZ$, we define $\Fil_d\calM_S$ (resp. $\Fil_d \widetilde{\calM}_S$) to be the subgroup of $\M_S$ (resp. $\widetilde{\calM}_S$) generated by elements of the form $[X]\LL^{-n}$ where $X$ is an $S$-variety, $n\in\bbZ$ is an integer and $\dim_{/S} X - n\leq d$.

This gives us an increasing and exhaustive filtration on the ring $\M_S$ (resp. $\widetilde{\calM}_S$). We denote by $\widehat{\M}_S$ (resp. $ \widehat{\widetilde{\calM}}_S$) the completion of $\M_S$ (resp. $\widetilde{\calM}_s$) with respect to this filtration. We will also use the notation $\Fil^d\M_S = \Fil_{-d}\M_S$. 
Define a function
$$\dim_{/S}:\widehat{\M}_S\to \bbZ\cup\{-\infty\}$$
by sending a class $a$ to $\inf\{d\in\bbZ,\ a\in \Fil_d\widehat{\M}_S\}$.

The following lemma will be a useful tool for controlling the position of certain classes in the dimension filtration. 
\begin{lemma}\label{lemma:dimension-fibration}
Let $X \rightarrow Y$, and $f:Y \rightarrow S$, and suppose there is a constructible $g: Y \rightarrow \bbZ$ such that for each $y \in Y$, $\dim X_y \leq g(y) - \dim_y Y_{f(y)}$. Then 
\[ \sum_i \left[ \left(X \times_Y g^{-1}(i)\right)/S \right] \bbL^{-i} \in \Fil^{0} \calM_S.  \]
\end{lemma}
\begin{proof}
Using \cite[Theorem 13.1.3]{EGAIVc}, we may decompose $Y=\bigsqcup Y_j$ into irreducible subvarieties such that the fibers of $Y_j$ over $S$ are equidimensional of constant dimension over $S$, and then further refine the decomposition so that $g$ is constant with value $k_j$ on each $Y_j$. Then 
\[ \sum_i [X \times_Y g^{-1}(i)/S] \bbL^{-i} = \sum_j [X_{Y_j}] \bbL^{-k_j}. \]
For all $y \in Y_j$ we have that $\dim_y (Y_j)_{f(y)} = \dim_{/S} Y_j $ and $g$ is constant on $Y_j$ with value $k_j$, so 
\begin{align*} \dim_{/Y_j} X_{Y_j} = \max_{y \in Y_j} \dim X_y & \leq \max_{y \in Y_j} \left( g(y) - \dim_y Y_{f(y)} \right) \\ 
& \leq \max_{y \in Y_j} \left( g(y) - \dim_y Y_{j,f(y)} \right)\\
 & \leq k_j - \dim_{/S} Y_j \end{align*}
Thus, 
\[ \dim_{/S} (X_{Y_j}) \leq \dim_{/Y_j}(X_{Y_j}) + \dim_{/S} Y_j \leq k_j, \]
so each term is in $\Fil^0$ and we conclude.  
\end{proof}

\section{Configuration spaces} \label{section.configuration-spaces}
\subsection{Partitions} \label{subsect:partitions}Let $I$ be an abelian semigroup, and $\mathcal{P}(I)$ the free commutative monoid generated by the elements of $I$. Following Vakil-Wood \cite{vakil-wood:discriminants}, elements of $\mathcal{P}(I)$ are called \textit{(generalized) partitions}. Such a partition $\mu$ may be denoted $\mu = (m_i)_{i\in I}$ where $m_i$ is the multiplicity of the element $i$, so that $m_i = 0$ for all $i$ but a finite number.  For $\mu = (m_i)_{i\in I} \in \mathcal{P}(I)$, we define $\sum \mu = \sum im_i$, $|\mu|=\sum m_i$ and $|| \mu || = \# \{i \; | \; m_i > 0 \}$. 

We denote by $\calQ_0$ the set $\mathcal{P}(\bbZ_{>0})$ of partitions in the traditional sense. For non-negative integers $m_1, \ldots, m_k$, we will denote by $(m_1, m_2, \ldots, m_k)$ or $1^{m_1} 2^{m_2} \ldots k^{m_k}$ the partition $(m_1, m_2, \ldots, m_k, 0, 0, \ldots)\in \calQ_0$. 
We denote by $\calQ$ the subset of $\calQ_0$ of partitions $(m_1, m_2, \ldots )$ such that for some non-negative integer $k$, $m_i = 0$ for all $i > k$ and $m_i > 0$ for all $i \leq k$. Such an element $\mu$ of $\calQ$ may be interpreted as an \textit{ordered partition} of the integer $|\mu|$. 

\subsection{Generalized configuration spaces}\label{sect.confspacesdefinition}
Let $S$ be a variety over a field~$K$. All products in this section are taken relatively to~$S$. For a quasi-projective variety $X$ over $S$, we denote by  $\Sym^n_{/S}(X) = X^n/\mathfrak{S}_n$ the symmetric power of $X$ relatively to~$S$. In what follows, to shorten notations, we will often omit $S$ and write simply $\Sym^n(X)$, unless there is some risk of confusion.

For a generalized partition $\mu = (m_i)_{i\in I}$, we may consider the natural quotient morphism $$p: \prod_{i \in I}X^{m_i}\to \prod_{i\in  I}\Sym^{m_i}X,$$
which is finite and surjective.  Let $U := \left(\prod_{i\in I}X^{m_i}\right)_{*,X/S}\subset \prod_{i\in I}X^{m_i}$ be the complement of the diagonal, that is, the open subset of points with all coordinates being distinct. 

\begin{definition}\label{definition:confspace} We denote by $\left(\prod_{i\in I}\Sym^{m_i}X\right)_{*,X/S}$, or $\Conf^{\mu}_{/S}(X)$, the open subset $p(U)$ of the variety $\prod_{i\in I}\Sym^{m_i}X.$ 
\end{definition}
The space $\Conf^{\mu}_{/S}(X)$ parameterizes $\mu$-configurations relative to~$S$, that is, finite collections of distinct points of~$X$ consisting of $m_i$ (unordered) points labeled by $i$ for every $i\in I$ and lying above the same point in~$S$.  In particular, if the partition $\mu$ is of the form $\mu = 1^n$ for some positive integer~$n$, we recover the definition of the usual configuration space $\Conf^n_{/S}(X)$ of~$n$ unordered points on~$X$ relatively to $S$. When $S$ is a (clearly specified) field $K$, we may omit it in the notation and write $\Conf^{\mu}(X)$. 

\begin{notation} Note that in what follows, when $\mu\in\calQ_0$ is a partition, geometric points of $\Conf^{\mu}_{/S}(X)$ will be written simply in the form $(x_1,\ldots,x_{|\mu|})$ for $x_1,\ldots,x_{|\mu|}$ distinct geometric points of $X$ lying above the same point of $S$. In particular, it is implicit in the notation that that $x_1,\ldots,x_{m_1}$ are unordered, $x_{m_1 + 1}, \ldots, x_{m_1 + m_2 -1}$ are unordered, etc. 
\end{notation}


More generally, if $\mathscr{X}=(X_{i})_{i\in I}$, is a family of $X$-varieties and $\mu = (m_{i})_{i\in I}$ is a generalized partition, we may consider the product
$$\prod_{i\in I}\Sym^{m_{i}}X_{i}$$
above $\prod_{i\in I}\Sym^{m_{i}}X$ and restrict it, as above, to the image $$\left(\prod_{i\in I}\Sym^{m_{i}}X\right)_{*,X/S}$$ of the complement of the diagonal. 

\begin{definition} \begin{enumerate} 
\item We denote by 
 $\Conf^{\mu}_{X/S}(\mathscr{X})$ or $\left(\prod_{i\in I}\Sym^{m_{i}}X_{i}\right)_{*,X/S}$ the resulting variety. By construction, it comes with a morphism to $\Conf^{\mu}_{/S}(X)$.
 \item If all the elements~$X_{i}$ of the family $\mathscr{X}$ are equal to the same variety $Y$, we write $\Conf^{\mu}_{X/S}(\mathscr{X}) =: \Conf^{\mu}_{X/S}(Y)$. 

\end{enumerate} 
\end{definition}

\begin{notation} As a general rule, we will use the $(Z)_{*,X/S}$ notation  (or $(Z)_{*}$ if there is no risk of confusion)  for a variety $Z$ which has an obvious morphism to a product of symmetric powers of $X$, to denote the restriction to the image of the complement of the diagonal. 
\end{notation}

\begin{remark}\label{remark:geometric-points-of-confspaces} Let $L$ be an algebraically closed field containing $K$. Then using lemmas \ref{lemma:geometric-points-image} and \ref{lemma:geometric-points-quotient}, we see that  the $L$-points of $\Conf^{\mu}_{X/S}(\mathscr{X})$ are given by families $(c_i)_{i\in I}$ where each $c_i = (x_{i,1},\ldots, x_{i,m_i})$ is a configuration of $m_i$ unordered $L$-points of $X_i$, such that $x_{i,j}$ for all $i,j$ lie above the same $L$-point of $S$. 
\end{remark}

\begin{remark} 
 Note that the definition of $\Conf^{\mu}_{X/S}(\mathscr{X})$ only uses the elements $X_i$ of the family~$\mathscr{X}$ with index $i$ such that $m_i>0$. More precisely, this construction defines a functor 
 \[ \Conf_{X/S}^\mu: (\Var/X)^{||\mu||} \rightarrow \Var/\Conf^{\mu}_{/S}(X).\]
 \end{remark}

\begin{remark} Assume that all of the elements of the family $\mathscr{X}$ are equal to $X$. Then the notations
$\Conf^{\mu}_{X/S}(\mathscr{X})$, $\Conf^{\mu}_{X/S}(X)$ and $\Conf^{\mu}_{/S}(X)$ (from definition \ref{definition:confspace}) refer to the same space and may be used interchangeably. 
\end{remark}



\begin{remark} In \cite{bilu:thesis}, the space $\Conf^{\mu}_{/S}(X)$ (resp. $\Conf^{\mu}_{X/S}(\mathscr{X})$) was denoted $S^{\mu}(X)$ (resp. $S^{\mu}(\mathscr{X})$). We choose to use the notation $\Conf_{X/S}^{\mu}$ here to make it clear that there is a restriction to the complement of the diagonal at the level of the base variety $X$,  that we are working relatively to $S$, and that in special cases we recover usual configuration spaces. On the other hand, for a family $\mathscr{X} = (X_i)_{i\in I}$ of $X$-varieties we will sometimes use here the notation $\Sym^{\mu}_{X/S}(\mathscr{X})$ to denote the product $\prod_{i\in I} \Sym^{m_i}X_i,$ so that $\Conf^{\mu}_{X/S}(\mathscr{X}) = (\Sym^{\mu}_{X/S}(\mathscr{X}))_{*,X/S}.$
\end{remark}


\subsection{Localization}

The functor 
\[ \Conf^{\mu}_{/S}: \Var/S \rightarrow \Var/S \]
induces a functor 
\[ \Conf^{\mu}_{/S}: (\Var/S)_W \rightarrow (\Var/S)_W \]
for $W=\PW,\,\RS,\textrm{ or } S-\PW$. Indeed, to verify this it suffices to check that in each case $\Conf^{\mu}_{/S}$ sends a morphism in $W$ to a morphism in $W$. 

Similarly, for a fixed $X/S$, the labeled configuration space functor
\[ \Conf_{X/S}^\mu: (\Var/X)^{||\mu||} \rightarrow \Var/\Conf^{\mu}_{/S}(X) \]
induces a functor
\[ \Conf^{\mu}_{X/S}: (\Var/X)_{X-\PW}^{||\mu||} \rightarrow (\Var/\Conf^{\mu}_{/S}(X))_{\Conf^{\mu}_{/S}(X)-\PW} \]
and functors
\[ \Conf^{\mu}_{X/S}: (\Var/X)_{W}^{||\mu||} \rightarrow (\Var/\Conf^{\mu}_{/S}(X))_{W} \]
for $W=\PW, \RS$. 

\subsection{Quasi-coherent sheaves on configuration spaces}\label{subsection.configuration-coherent-sheaves-functor}
Consider a partition $\mu \in \calQ$ and a variety $X/S$. We write $n=|\mu|$, $m=||\mu||$, $\mu=(l_1, \ldots, l_m)$, and $\mathfrak{S}_\mu$ for the product of symmetric groups $\mathfrak{S}_{l_1} \times \ldots \times\mathfrak{S}_{l_m}$. Then, denoting by $\PConf^n_{/S}(X)$ the ordered configuration space of $n$ points on~$X$ relatively to~$S$, there is a natural map
\[ \pi_{\mu}: \PConf^n_{/S}(X) \rightarrow \Conf^{\mu}_{/S}(X) \]
which is a finite \'{e}tale Galois $\mathfrak{S}_\mu$-cover. We denote by $\QCoh(X)$ the category of quasi-coherent sheaves on~$X$. Using this cover, we obtain a functor
\[ \Conf^\mu_{X/S}: \QCoh(X)^{||\mu||} \rightarrow  \QCoh(\Conf^{\mu}_{/S}(X)) \]
defined as follows: we write $p_k$ for the $k$th projection map 
\[ p_k: \PConf^n_{/S}(X) \rightarrow X.\] 
Then, given quasi-coherent sheaves $\calF_1, \ldots, \calF_m$ on $X$ we form the $\mathfrak{S}_\mu$-equivariant sheaf 
\[ \bigoplus_{i=1}^m \left( \bigoplus_{k=l_1 + \ldots + l_{i-1} + 1}^{l_1+ \ldots + l_i} p_k^* \calF_i \right). \]
To obtain $\Conf_{X/S}^{\mu}(\calF_1, \ldots, \calF_m)$ we descend via $\pi_\mu$ (i.e. apply ${(\pi_{\mu})}_{*}(\bullet)^{\mathfrak{S}_\mu}$). 

Note that if we restrict to coherent sheaves, we obtain a functor
\[ \Conf^\mu_{X/S}: \Coh(X)^{||\mu||} \rightarrow  \Coh(\Conf^{\mu}_{/S}(X)). \]
Going forward, we will usually want to apply $\Conf^\mu_{X/S}$ with the same argument $\calF$ in each spot; we write this as $\Conf^{\mu}_{X/S}(\calF)$.  

\section{The geometric realization of a coherent sheaf}\label{section.geometric-realization}
\newcommand{\rk}{\mathrm{rk}}

Let $S/K$ be a variety. If $\calF$ is locally free of finite rank on $S$, then it is naturally the sheaf of sections of the geometric vector bundle $\bbV(\calF):=\Spec_{/S} \Sym^\bullet (\calF^*)$, and the assignment sending a locally free sheaf to the representing geometric vector bundle is a functor from locally free sheaves of finite rank on $S$ to $\Var/S$.

In particular, this allows us to define a class $[\bbV(\calF)] \in K_0(\Var/S)$ attached to a locally free sheaf $\calF$ of finite rank; moreover, by choosing a decomposition of $S$ into locally closed subvarieties $S_i$ such that $\calF|_{S_i}$ is free, we find this class is equal to $\bbL^{\rk(\calF)}$ (which makes sense as $\rk(\calF)$ is locally constant). It is not necessary to pass all the way to the Grothendieck ring to obtain this simple form: the same argument shows already that
\begin{equation}\label{equation.structure-theorem-loc-free} \bbV(\calF) \cong \bbA^{\rk(\calF)} \textrm{ in } (\Var/S)_{S-\PW}.\end{equation}

In this section, we explain how extend the functor $\bbV$ on locally free sheaves to
\[ \bbV: \Coh(S) \rightarrow (\Var/S)_{S-\PW}. \]
We then give a structure theorem for $\bbV(\calF)$ analogous to (\ref{equation.structure-theorem-loc-free}), and a structure theorem describing the image under $\bbV$ of a surjective map of coherent sheaves. We also explain how $\bbV$ interacts with the functor
\[ \Conf^\mu_{X/S}: \Coh(X)^{||\mu||} \rightarrow \Coh\left(\Conf_{/S}^\mu(X)\right) \]
defined in \ref{subsection.configuration-coherent-sheaves-functor}. 

\subsection{The canonical flattening decomposition}

While an arbitrary coherent sheaf $\calF$ cannot generally be represented as the sheaf of sections of a map $F \rightarrow S$, the Fitting ideals of $\calF$ can be used (cf. \cite[05P8]{stacks-project}) to construct a canonical decomposition $S_\calF \sim S$ such that $\calF|_{S_{\calF}}$ is locally free of finite rank: more precisely, $S_\calF=\bigsqcup_{i \geq 0} S_{\calF, i}$ is a finite decomposition into locally closed subvarieties such that that $\calF|_{S_{\calF, i}}$ is locally free of rank $i$, and $S_{\calF,i}=Z_i \bs Z_{i-1}$ where $Z_i$ is the closed subvariety determined by the $i$th Fitting ideal. We will call $S_\calF \sim S$ the canonical flattening decomposition for $\calF$.

\subsection{Definition of $\bbV$}
  Thus, we can define the functor $\bbV$ on objects by
 \[ \calF  \mapsto \bbV(\calF|_{S_{\calF}})/S \]
where here the structure map $\bbV(\calF|_{S_{\calF}}) \rightarrow S$ is the composition of
\[ \bbV(\calF|_{S_{\calF}}) \rightarrow S_{\calF} \rightarrow S. \]
Given a morphism $\phi: \calF \rightarrow \calG$, the element
\[\bbV(\phi) \in \Map_{(\Var/S)_{\PW}}\left(\bbV(\calF), \bbV(\calG)\right) \]
is obtained by passing to a common refinement $S'$ of $S_\calF$ and $S_\calG$ and then taking $\bbV(\phi)$ to be the image in 
\[ \Map_{(\Var/S)_{S-\PW}}\left(\bbV(\calF), \bbV(\calG) \right) \]
of the element of 
\[ \Map_{(\Var/S)}\left( \bbV(\calF) \times_{S} {S'}, \bbV(\calG) \right) \]
obtained by composition of the natural maps
\[ \bbV(\calF) \times_{S} {S'} = \bbV(\calF|_{S'}) \xrightarrow {\bbV(\phi|_{S'})} \bbV(\calG|_{S'}) = \bbV(\calG) \times_S S' \rightarrow \bbV(\calG). \]
Here the map $\bbV(\phi|_{S'})$ makes sense because $\phi|_{S'}$ is a map of locally free sheaves, where we already have a functor $\bbV$ defined. 

\subsection{Structure theorems for the image of $\bbV$} Objects in the image of $\bbV$ have simple models in $(\Var/S)_{S-\PW}$: 
Let $g$: $|S| \rightarrow \bbZ_{\geq 0}$ be a function with bounded image such that for each $i$, $g^{-1}(i)$ is constructible. Then, we can equip $g^{-1}(i)$ with the structure of a variety over $S$ by writing it as the disjoint union of locally closed subsets of $S$. Thus, it makes sense to consider 
\[ \bbA^g = \bigsqcup_i \bbA^{i} \times g^{-1}(i) \]
as an object in $(\Var/S)_{S-\PW}$. In particular, for $\calF \in \Coh(S)$ if we let $\dim \calF$ denote the function $s \mapsto \dim_{/\kappa(s)} \calF|_s$ then we obtain directly from the definition:

\begin{theorem}\label{theorem:V-structure} For $\calF \in \Coh(S)$, $\bbV(\calF)$ is isomorphic to $\bbA^{\dim \calF}$ in $(\Var/S)_{S-\PW}.$ 
\end{theorem}

We also obtain a simple normal form for the image of a surjection under $\bbV$. This will be an important ingredient in later arguments. 

\begin{theorem}\label{lemma:presentation-of-surjection}
Let $\phi: \calF \rightarrow \calG$ be a surjection in $\Coh(S)$. Then, there exists an isomorphism of arrows in $(\Var/S)_{S-\PW}$
\[\xymatrix{ \bbV( \calF )  \ar[rr] \ar[d]^\sim && \bbV( \calG ) \ar[d]^\sim \\
\bbA^{\dim \calF - \dim \calG} \times_S \bbA^{\dim \calG} \ar[rr]_{(0, \id)} && \bbA^{\dim \calG}. }\] 
\end{theorem} 
\begin{proof}
We can find a decomposition $S'\sim S$ such that both $\calF|_{S'}$ and $\calG|_{S'}$ are locally free and trivializable, and such that the sequence
\[ 0\rightarrow \ker\left(\phi|_{S'}\right) \rightarrow \calF|_{S'} \rightarrow \calG|_{S'}\rightarrow 0 \]
is split. Then $\ker (\phi|_{S'})$ is also locally free, and for any further refinement $S''$, $\ker \phi|_{S''} = (\ker (\phi|_{S'}))|_{S''}$. We can thus choose $S''$ so that $\ker (\phi|_{S''})$ is also trivializable, and we obtain the result.  
\end{proof}

If we localize further to $(\Var/S)_\PW$, we obtain the following corollary describing the pre-image of a constructible set under a surjective map of coherent sheaves:
\begin{corollary}\label{corollary:presentation-of-surjection}
Let $\phi: \calF \rightarrow \calG$ be a surjection in $\Coh(S)$ and let $C$ be a constructible subset of $\bbV(\calG)$. Then, in $(\Var/S)_\PW$,
\[ \bbV(\phi)^{-1} (C) \cong \bbA^{\dim \calF - \dim \calG} \times_S C. \]
In particular,
\[ [\bbV(\phi)^{-1}(C)] = [\bbA^{\dim\calF - \dim \calG}] \cdot [C] \in K_0(\Var/ S) \]
and
\[  [\bbV(\phi)^{-1}(C)] = \frac{[\bbV(\calF)]}{[\bbV(\calG)]} \cdot [C]  \in \calM_S. \]
\end{corollary}

\subsection{Configurations and geometric realization}

For any partition $\mu$, geometric realization transforms the functor  
\[ \Conf^{\mu}_{X/S}:  \Coh(X)^{||\mu||} \rightarrow \Coh\left( \Conf^{\mu}_{/S}(X) \right) \]
defined in \ref{subsection.configuration-coherent-sheaves-functor} into a labeled configuration space functor:

\begin{lemma}\label{lemma:compatibility-conf-realization}For any partition $\mu$, there is a natural equivalence
\[ \Conf^{\mu}_{X/S} ( \bbV( \bullet ), \ldots, \bbV(\bullet) ) = \bbV( \Conf^{\mu}_{X/S}(\bullet, \ldots, \bullet)) \]
of functors 
\[ \Coh(X)^{||\mu||} \rightarrow (\Var/\Conf^{\mu}_{/S}(X))_{\Conf^{\mu}_{/S}(X)-\PW}. \]
\end{lemma}
\begin{proof}
We write $|\mu|=n$, $||\mu||=m$, $m(\mu)=(l_1, \ldots, l_m).$ Because $\Conf^{\mu}_{X/S}(\calF_1, \ldots, \calF_m)$ was defined by descent, there is a canonical isomorphism of $\mathfrak{S}_\mu$-equivariant sheaves on $\PConf^n_{/S}(X)$
\[ \bigoplus_{i=1}^m \left( \bigoplus_{k=l_1 + \ldots + l_{i-1} + 1}^{l_1+ \ldots + l_i} p_k^* \calF_i \right) = \pi_\mu^* \Conf^{\mu}_{X/S}(\calF_1, \ldots, \calF_m). \]
Applying $\bbV$ (which turns pullbacks into fiber products), we obtain a canonical $\mathfrak{S}_\mu$-equivariant isomorphism 
\[ \prod_{i=1}^m \bbV(\calF_i)^{l_i}|_{\PConf^n_{/S}(X)} = \bbV\left(\Conf^{\mu}_{X/S}(\calF_1, \ldots, \calF_m)\right) \times_{\Conf^{\mu}_{/S}(X)} \PConf^n_{/S}(X) \]
where the $\mathfrak{S}_{\mu}$-action on the left is the obvious one by permutation of coordinates and the $\mathfrak{S}_{\mu}$-action on the right is only on the factor $\PConf^n_{/S}(X)$. Passing to the quotient by the $\mathfrak{S}_{\mu}$-action then gives the result.
\end{proof}

\section{Taylor expansions}\label{section.taylor-expansions}

We develop here some material on Taylor expansions for global sections of coherent sheaves that will be important in the rest of this work. 

\subsection{The sheaf of principal parts}

We recall here some notation and results on the sheaf of principal parts from \cite[16.7]{EGAIVd}. 

Let $f:X\rightarrow S$ be a map of varieties, and let $\calF$ be a quasi-coherent sheaf on $X$. Let $\Delta^{(k)}$ be the $k$-th infinitesimal neighborhood of the diagonal in $X\times_S X$, and let 
\[ p_i:\Delta^{(k)} \rightarrow X, \; i=1,\,2\]
denote the projection onto the $i$th component of $X \times_S X$, restricted to $\Delta^{(k)}$. For $k>0$, we define the sheaf of $k$-principal parts of $\calF$ (relative to $S$) to be the $\calO_X$-module 
\[ \calP^k_{X/S}(\calF) := {p_1}_*p_2^*\calF.  \]
If $\calF$ is coherent, then so is $\calP^k_{X/S}(\calF)$ \cite[16.7.3]{EGAIVd}.

If $K$ is a field, $x \in X(K)$, $s$ is the image of $x$ in $S(K)$, and $\frakm$ denotes the maximal ideal of the local ring $\calO_{X_s,x}$, then the $K$-vector space $\calP^k_{X/S}(\calF)|_x$ is naturally identified with the $K$-vector space $(\calF|_{X_s})_x/ \frakm^{k+1}$. There is a natural $k$th-order expansion map of $f^{-1}\calO_S$-modules
\[ d^k_{X/S,\calF}: \calF \rightarrow \calP^k_{X/S}(\calF) \] 
such that for $x, s$, as above the composition
\[ (\calF|_{X_s})_{x} \xrightarrow{d^k_{X/S,\calF}} \left(\calP^k_{X/S}(\calF)|_{X_s}\right)_x \rightarrow \calP^k_{X/S}(\calF)|_x = (\calF|_{X_s})_x / \frakm^{k+1} \]
is the natural reduction map. 

\subsection{The Taylor expansion map}
Although the expansion map $d^k_{X/S,\calF}$ is a map of $f^{-1}\calO_S$ modules, it is \emph{not} typically a map of $\calO_X$-modules (because when multiplying a local section by a function, the effect on the infinitesimal expansion involves not just the function but also the higher terms in its own $k$-th order expansion). However, we do obtain a natural Taylor expansion map of $\calO_X$-modules
\[ \Taylor^k_{X/S,\calF}: f^*f_* \calF \rightarrow \calP^k_{X/S}(\calF) \]
by adjunction from the map of $\calO_S$-modules $f_* d^k_{X/S,\calF}$. For a pair $x, s$ as above, the fiber of $\Taylor^k_{X/S, \calF}$ at $x$ computes the $k$th order Taylor expansion at $x$ of a global section (all relative to $S$): it is naturally identified with the composition
\[ f^*f_* \calF |_x \rightarrow \Gamma(X_s, \calF|_{X_s}) \rightarrow (\calF|_{X_s})_x \rightarrow (\calF|_{X_s})_x / \frakm_x^{k+1} = \calP^k_{X/S}(\calF)|_x. \]

The construction $\calP^k_{X/S}$ is functorial from quasi-coherent sheaves on $X$ to quasi-coherent sheaves on $X$, and the Taylor expansion maps assemble to a natural transformation 
\[ \Taylor^k_{X/S}: f^*f_* \rightarrow \calP^k_{X/S}. \]

\subsection{Taylor expansions at multiple points} 
Let $\mu$ be a partition. We will also need to consider the variant where we take Taylor expansions simultaneously along a $\mu$-labeled configuration of points. To do so, we first compose the functor
$\Conf^{\mu}_{X/S}$ with $\Taylor^r_{X/S}$ to obtain for any $\calF$ a map
\[ \Conf^{\mu}_{X/S} (f^* f_* \calF) \rightarrow \calP^{r,\mu}_{X/S}(\calF) := \Conf^{\mu}_{X/S}(\calP^r_{X/S}(\calF)). \]
Writing $f_\mu$ for the structure morphism $\Conf^{\mu}_{/S}(X) \rightarrow S$, we note that there is a natural diagonal map 
\[ f_\mu^* f_*\calF \rightarrow \Conf^{\mu}_{X/S}(f^* f_* \calF) \]
and we write $\Taylor^{r,\mu}_{X/S}$ for the composition of this diagonal map with $\Conf^{\mu}_{X/S}(\Taylor^r_{X/S})$, so that $\Taylor^{r,\mu}_{X/S}$ is a natural transformation from $f_\mu^* f_*$ to $\calP^{r,\mu}_{X/S}$. 

\subsection{Generating sets of global sections}

\begin{definition}
We say that a map of quasi-coherent sheaves $\calG \rightarrow f_*\calF$ is $r$-infinitesimally $m$-generating if the following equivalent conditions hold:
\begin{enumerate}
\item For every $\mu$ with $|\mu|\leq m$, $\Taylor^{r,\mu}_{X/S}(\calF):f^*_{\mu}f_*\calF \to \calP^{r,\mu}_{X/S}(\calF)$ restricts to a surjective map 
\[ f_\mu^* \calG \twoheadrightarrow \calP^{r,\mu}_{X/S}(\calF). \]
\item For every algebraically closed $L$, $s \in S(L)$, and $n\leq m$ distinct points $x_1, \ldots, x_n$ in $X_s(L)$, the composition
\[ \calG|_s \rightarrow \Gamma(X_s, \calF) \rightarrow \bigoplus_{i=1}^{n} 
(\calF|_{X_s})_{x_i} / \frakm_{x_i}^{r+1} \]
is surjective, where $\frakm_{x_i}$ denotes the maximal ideal in the local ring $\calO_{X_s,x_i}$. 
\end{enumerate}
\end{definition}

\begin{remark}\label{remark:tensorbygloballygenerated} If $\calG$ is $r$-infinitesimally $m$-generating, then so is $\calG\otimes \calL$ for any globally generated line bundle $\calL$. Indeed, if we fix $m$ geometric points and a subspace surjecting onto the expansions at one point and zero at the others, then we can multiply it by a non-vanishing section of $\calL$ to get a subspace with the same property.
\end{remark}

\begin{example}\label{example:surj-projective-space} We take $S = K$, $X = \bbP^n_K$ and $\calF = \calO(d)$, so that $f_*\calF = \Gamma(\bbP^n,\calO(d))$. We claim that if 
\[ d \geq r + (m-1)(r+1) = mr + m - 1,\]
then $\Gamma(\bbP^n, \calO(d))$ is $r$-infinitesimally $m$-generating. To see this, we will show that the second criterion holds. So, assume $L/K$ is algebraically closed, and let $x_1, \ldots, x_m$ be distinct points in $\bbP^n(L)$ -- note that since $\bbP^n(L)$ is infinite, this also covers the case of $< m$ points. 

We first observe that for any $x \in \bbP^n(L)$, the reduction map from $V_{r}:=\Gamma(\bbP^n, \calO(r))$ to $\calO_{\bbP^n, x}(r)/\frakm_{x}^{r+1}$ is a bijection -- to see this, we may fix local coordinates so that $\calO_{\bbP^n, x}(r)/\frakm_{x}^{r+1}$ is identified with polynomials of degree $\leq r$ in $n$ variables, then homogenize to establish the bijection. Furthermore, if $f$ is a homogeneous polynomial of degree $a$ that does not vanish at $x$, then the reduction map induces a bijection between $f\cdot V_r \subset \Gamma(\bbP^n, \calO(r+a))$ and $\calO_{\bbP^n, x}(r+a)/\frakm_{x}^{r+1}$ (in suitable bases, the multiplication by $f$ modifies the previous bijection by composition with an upper triangular matrix with a non-zero scalar along the diagonal). 

Thus, if we choose for each $i$ a homogeneous polynomial $f_i$ of degree $d-r$ that vanishes to order $\geq r$ at each $x_j,\, j \neq i$, but not at $x_i$ (this is possible because $d - r \geq (m-1)(r+1)$ -- for example taking for each $j$ a linear form vanishing at $x_j$ but no other $x_i$, raise it to the $r+1$st power, and then multiply these together), then, the reduction map induces a bijection between the subspace $\sum_i f_i V_r \subset \Gamma(\bbP^n, \calO(d))$ and $\bigoplus_{i=1}^{m} \calO_{\bbP^n, x_i}(d)/\frakm_{x_i}^{r+1}$, and we conclude. 
\end{example}

\begin{example}\label{example:surj-ample} Suppose $X/K$ is a projective variety and $\calL$ is an ample line bundle. Fix $A>0$ such that $\calL^A$ is very ample and $B>0$ such that $\calL^b$ is globally generated for all $b\geq B$. Let $X \rightarrow \bbP^n$ be the embedding given by $\calL^A$. Then, using the previous example, we find that for 
\[ d = A \cdot \left( r + (m-1)(r+1)\right), \]
$\Gamma(X, \calL^d)$ is $r$-infinitesimally $m$-generating. Then, since $\calL^b$ is globally generated for all $b \geq B$, by Remark \ref{remark:tensorbygloballygenerated} we find that for 
\[ d \geq B + A \cdot \left( r + (m-1)(r+1) \right), \]
$\Gamma(X, \calL^d)$ is $r$-infinitesimally $m$-generating. 
\end{example}

\begin{lemma}\label{lemma:surj-map-generating}
Let $f:X\to S$ be a map of varieties, $\phi: \calF_1 \rightarrow \calF_2$ a surjective map of quasi-coherent sheaves on~$X$, and suppose $g: \calG \rightarrow f_*\calF_1$ is $r$-infinitesimally $m$-generating. Then 
\[ f_*\phi \circ g: \calG \rightarrow f_*\calF_2 \]
is $r$-infinitesimally $m$-generating.
\end{lemma}
\begin{proof}
For any $\mu$, the diagram
\[ \xymatrix{ f^*_{\mu} \calG \ar[rrr]^{\Taylor^{r,\mu}_{X/S}(\calF_1)} \ar[d] & && \calP^{r,\mu}_{X/S}(\calF_1)\ar[d] \\
f^*_{\mu} \calG \ar[rrr]^{\Taylor^{r,\mu}_{X/S}(\calF_2)} &&& \calP^{r,\mu}_{X/S}(\calF_2) } \]
commutes, where the vertical arrows are induced by $\phi$. For $|\mu| \leq m$ the top horizontal arrow is surjective by hypothesis. We claim that the right vertical arrow is surjective for any $\mu$, which will then imply the bottom arrow is also surjective for $|\mu| \leq m$ completing the proof. Indeed, the functor $\calP^{r,\mu}_{X/S}$ is right exact as the composition of two right-exact functors:  $\calP^k_{X/S}$ is right exact by \cite[16.7.3]{EGAIVd} and $\Conf^{\mu}_{X/S}$ is right exact since it is given as a composition of pullbacks and descent.    
\end{proof} 

\begin{example}\label{example:inf-m-gen-coh-sheaf} Suppose $X/K$ is a projective variety, $\calL$ is an ample line bundle on $X$, and $\calF$ is a coherent sheaf on $X$. We can find a surjection $(\calL^{-D})^N \twoheadrightarrow \calF$ by taking $D$ large enough so that $\calL^D \otimes \calF$ is globally generated. Thus, invoking Lemma \ref{lemma:surj-map-generating} and Example \ref{example:surj-ample}, we find that for $A$ such that $\calL^A$ is very ample and $B$ such that $\calL^b$ is globally generated for all $b \geq B$,
\[ d \geq D + B + A \cdot \left( r + (m-1)(r+1)\right), \]
$\Gamma(X, \calF \otimes \calL^d)$ is $r$-infinitesimally $m$-generating. 
\end{example}

\subsection{Geometric realization}
Let $X/S$ be a variety, let $\calF$ be a coherent sheaf on~$X$ and let $\calG$ be a coherent sheaf of global sections of $\calF$. For any partition $\mu$, by composition with $f_\mu^* \calG \rightarrow f_\mu^*f_*\calF$, we obtain a map of coherent sheaves 
\[ \Taylor^{r,\mu}_{X/S}: f_\mu^* \calG \rightarrow \calP^{r,\mu}_{X/S}(\calF). \]
Applying the geometric realization, we obtain a map
\[ \Conf^{\mu}_{/S}(X) \times_S \bbV(\calG) \rightarrow \bbV(\calP^{r,\mu}_{X/S}(\calF)). \]
By Lemma \ref{lemma:compatibility-conf-realization}, this is identified with a map
\[ \Conf^{\mu}_{/S}(X) \times_S \bbV(\calG) \rightarrow \Conf^{\mu}_{X/S}( \bbV(\calP^{r}_{X/S}(\calF)) ). \]
Tracing through the definitions, we find this map has the obvious description on geometric points:
\begin{enumerate}
\item A geometric point on the left is given by taking a geometric point $s$ of $S$, a $\mu$-configuration of geometric points $(x_1, \ldots, x_n)$ in $X_s$, and an element $g$ of the fiber $\calG|_s$. 
\item A geometric point on the right is given by taking a geometric point $s$ of $S$ and a labeled $\mu$-configuration of points $(x_1, \ldots, x_n)$ in $X_s$, where $x_i$ is labeled by an element of the fiber $\calP^{r,\mu}_{X/S}(\calF) |_{x_i} = \left(\calF|_{X_s}\right)_x /\frakm_{x_i}^{r+1}$.
\item The map sends the data $s, (x_1, \ldots, x_n), \; g$ on the left to the configuration where each $x_i$ is labeled by the image of $g$ under the natural $r$-th order expansion map 
\[\calG|_s \rightarrow \Gamma(X_s, \calF|_{X_s}) \rightarrow \left(\calF|_{X_s}\right)_{x_i} /\frakm_{x_i}^{r+1}. \]
\end{enumerate}

\section{Motivic Euler products}\label{sect.motivic-euler-products}
\subsection{Definition}
In \cite{bilu:thesis}, the first author recently defined a notion of \textit{motivic Euler product}, that is, a way to make sense of products of the form
$$\prod_{y\in Y} F_y(\mathbf{t}),$$
where $Y$ is a variety over a field $K$, $(t_i)_{i\in I}$ is a collection of indeterminates indexed by a set $I$, $F\in 1 + (t_i)_{i\in I}K_0(\Var/Y)[[t_i]]_{i \in I}, $
and $F_y(\mathbf{t})$ denotes the power series restricted to $y \in Y$. Here we recall this definition, some properties, and prove some additional properties not contained in \cite{bilu:motiviceulerproducts}. 

\subsubsection{Symmetric powers}
Let $X\rightarrow S$ be a map of quasi-projective varieties over a field~$K$. Recall that we write 
\[ \Sym_{/S}^n X := \underbrace{X \times_S \ldots \times_S X}_{n} / \mathfrak{S}_n \]
for the $n$th symmetric power of $X$ relatively to $S$. Given a variety $Y/X$, we write $\Sym_{X/S}^n Y $ for the element
\[ \Sym_{/S}^n Y \rightarrow \Sym_{/S}^n X \]
of $\Var/\Sym^n_{/S} X$,  which can also be viewed as a $Y$-labeled symmetric power of~$X$. 

We now explain how to extend this construction to elements of $K_0(\Var/X)$, that is, define for any $a\in K_0(\Var/X)$ an element $\Sym^n_{X/S}(a)\in K_0(\Var/\Sym^n_{/S}X)$. To do so, we first observe that the product 
$$K_0(\Var/\Sym_{/S}^{\bullet}X) := \prod_{n\geq 0}K_0(\Var/\Sym_{/S}^n X)$$ has the structure of a graded $K_0(\Var/S)$-algebra. The addition is termwise, and multiplication is given by Cauchy product: for 
\[ a = (a_n)_{n\geq 0},\, b = (b_n)_{n\geq 0}\in \prod_{n\geq 0} K_0(\Var/\Sym_{/S}^\bullet X),\]
the product $ab$ has $n$-th component given by
$$ (ab)_n = \sum_{i=0}^na_i b_{n-i}.$$
Here the product $a_ib_{n-i}$ should be understood as the image of $(a_i,b_{n-i})$ through the composition
\begin{multline*}K_0(\Var/\Sym_{/S}^iX)\times K_0(\Var/\Sym_{/S}^{n-i}X)\xrightarrow{\boxtimes}K_0(\Var/\Sym_{/S}^iX\times \Sym_{/S}^{n-i}X)\\ \to K_0(\Var/\Sym_{/S}^nX)\end{multline*}
where the last map is induced by composition with the natural arrow 
\[ \Sym_{/S}^i X \times \Sym_{/S}^{n-i} X \rightarrow \Sym_{/S}^n X. \]
We denote by 
\[ K_0(\Var/\Sym_{/S}^{\bullet}X)^1\subset K_0(\Var/\Sym_{/S}^{\bullet}X) \]
the subset of families $(a_n)_{n \geq 0}$ with $a_0 = 1.$ It is a subgroup of the multiplicative units in $K_0(\Var/\Sym_{/S}^{\bullet}X).$

\begin{proposition}[\cite{bilu:motiviceulerproducts}, lemma 3.5.1.2 ]\label{prop:sympower} There is a unique group morphism
$$\Sym_{X/S}:K_0(\Var/X) \to K_0(\Var/\Sym_{/S}^{\bullet}X)^1,$$
such for any quasi-projective variety $Y$ over $X$, 
\[ \Sym_{X/S}([Y])= ([\Sym^n_{X/S} Y])_{n\geq 0}. \]
\end{proposition}

\begin{definition}\label{def:sympower} For every $a\in K_0(\Var/X)$, we 
denote by $$\Sym_{X/S}^n(a)\in K_0(\Var/\Sym_{/S}^n X)$$ the $n$-th component of $\Sym_{X/S}(a)$. 
More generally, for a generalized partition $\mu = (m_i)_{i\in I}$, we write
\[ \Sym^\mu_{/S} X := \prod_{i\in I} \Sym^{m_i}_{/S} X \]
(where the product is taken relatively to~$S$), and
\[ \Sym_{X/S}^\mu(a) := \boxtimes_{i\in I} \Sym^{m_i}_{X/S}(a) \in K_0(\Var/\Sym^\mu_{/S} X). \] 
\end{definition}

\begin{example}\label{example:sympower} 
We claim that for $a \in K_0(\Var/X)$,  
\begin{equation}\label{equation.negative-symmetric-power} \Sym_{X/S}^n (-a) = \sum_{\substack{\mu \in \calQ \\|\mu| = n}}(-1)^{||\mu||} \Sym_{X/S}^\mu(a), \end{equation}
where $\calQ$ is the set of partitions defined in \ref{subsect:partitions}, and  $\Sym_{X/S}^\mu(a)$ is viewed as an element of $K_0\left(\Var/\Sym^{|\mu|}_{/S} X\right)$ by composing with the natural map $\Sym^\mu_{/S} X \rightarrow \Sym^{|\mu|}_{/S} X.$
To verify this, we first observe
\begin{align*} \Sym_{X/S}(-a) &= \frac{1}{\Sym_{X/S}(a)} \\
& = \frac{1}{1 + \left(\Sym_{X/S}(a)-1\right)} \\
& = \sum_{k=0}^\infty (-1)^k (\Sym_{X/S}(a) - 1)^k.
\end{align*}
The result then follows from writing 
\[ \Sym_{X/S}(a)-1=(0, a, \Sym_{X/S}^2(a), \Sym_{X/S}^3(a), \ldots) \]
and expanding in each index.  
\end{example}

\begin{remark} If $Y\to Z$ is a radicial surjective morphism between varieties over~$X$, then, computing on geometric points, we find that for any $m$ the induced map 
\[ \Sym^m_{/S} Y\to \Sym^m_{/S} Z \]
is also radicial surjective. Thus, $\Sym^m_{X/S}$ induces a functor
\[ (\Var/X)_{\RS} \rightarrow (\Var/\Sym_{/S}^m X)_{\RS}. \]
In particular, the group morphism~$\Sym_{X/S}$ of proposition \ref{prop:sympower} descends to a morphism
$$\Sym_{X/S}:\widetilde{K}_0(\Var/X) \to \widetilde{K}_0(\Var/\Sym_{/S}^{\bullet}X)^1,$$
where the right-hand side is defined in the obvious way using $\widetilde{K}_0$ everywhere in place of $K_0$. 
\end{remark}
\subsubsection{Symmetric products and motivic Euler products}

Let $I$ be an abelian semigroup, and let $\mu = (m_i)_{i\in I}$ be a generalized partition. Let $\mathscr{A} = (a_i)_{i\in I}$ be a family of classes in $K_0(\Var/X)$. Using definition \ref{def:sympower}, we have the following:

\begin{definition}\label{symproddef} Let $\mu\in \calP(I)$ be a generalized partition. 
\begin{enumerate} \item 
We define $\Sym_{X/S}^{\mu} (\mathscr{A}):=\boxtimes_{i\in I}\Sym_{X/S}^{m_i}(a_i) \in K_0\left(\Var/\Sym_{/S}^{\mu}X\right)$. 
\item We further define  $$\Conf^{\mu}_{X/S}(\mathscr{A}) := \left(\Sym^\mu_{X/S}(\mathscr{A})\right)_{*,X/S}\in K_0\left(\Var/\Conf^{\mu}_{/S}(X)\right),$$
where the subscript ``$*, X/S$" indicates application of the restriction morphism
\begin{equation} \label{morphism.complementdiagonal}K_0\left(\Var/\Sym_{/S}^\mu X\right)\to K_0\left(\Var/\Conf^{\mu}_{/S}(X)\right)\end{equation}
induced by the inclusion $\Conf^{\mu}_{/S} (X) \to \Sym^\mu_{/S} X$. If we fix some $a\in K_0(\Var/X)$ and take $a_i=a$ for all $i \in I$, then we write  
\[ \Conf^{\mu}_{X/S}(a) := \Conf^{\mu}_{X/S}(\mathscr{A}). \]
\end{enumerate}
\end{definition}

\begin{remark} In general, for an element $a$ of a Grothendieck ring of the form $K_0\left(\Var / \Sym^{\mu}_{/S}X \right)$, we will denote by $(a)_{*,X/S}$ or $a|_{\Conf^{\mu}_{/S}(X)}$ the image of $a$ in the Grothendieck ring $K_0\left(\Var/\Conf^{\mu}_{/S}(X)\right)$ via the restriction morphism~(\ref{morphism.complementdiagonal}). 
\end{remark}

\begin{remark} Our notation is consistent with the one introduced in section \ref{sect.confspacesdefinition}, in the sense that if we start with a family $\mathscr{X} = (X_i)_{i\in I}$ of varieties over $X$ and take $\mathscr{A} = ([X_{i}])_{i\in I}$ to be the family of their classes in $K_0(\Var/X)$, then the class of $\Conf^{\mu}_{X/S}(\mathscr{X})$ in $K_0(\Var/\Conf^{\mu}_{/S}(X))$ is $\Conf^{\mu}_{X/S}(\mathscr{A})$. In particular, one can think of the construction $\Conf^\mu_{X/S}$ on $K_0(\Var/X)$ as extending the notion of a labeled configuration space by moving from varieties over $X$ to classes in $K_0(\Var/X)$ as spaces of labels.   
\end{remark}

\begin{definition}[Motivic Euler product] For $\mathscr{A} = (a_i)_{i\in I}$ a family of classes in $K_0(\Var/X)$, we denote
$$\prod_{x\in X/S} \left( 1 + \sum_{i\in I} a_{i,x} t_i\right) := \sum_{\mu\in \mathcal{P}(I)}\Conf^{\mu}_{X/S}(\mathscr{A}) \mathbf{t}^{\mu}\in K_0(\Var/S)[[(t_i)_{i\in I}]],$$
where $ \mathbf{t}^{\mu} = \prod_{i\in I}t_i^{m_i}$, and we consider $\Conf^{\mu}_{X/S}(\mathscr{A})$ as a class in $K_0(\Var/S)$ by applying the forgetful morphism $K_0\left(\Var/\Conf^ {\mu}_{/S}(X)\right) \rightarrow K_0(\Var/S)$. 
\end{definition}

As shown in \cite[Sections 3.8 and 3.9]{bilu:motiviceulerproducts}, this notion of motivic Euler product satisfies many properties that one may wish for: it is multiplicative, it is compatible with finite products, and we can take double products.


\begin{example} \label{example:kapranov-zeta}
For a quasi-projective variety $X/S$, the Kapranov zeta function is defined by
\[ Z_{X/S}^{\Kap}(t) = \sum_{n\geq 0} [\Sym^n_{/S}X]t^n\in 1+t K_0(\Var/S)[[t]].\]
We explain how this can be expressed as a motivic Euler product. From the definition, we compute
\begin{align*} \prod_{x \in X} \frac{1}{1-t} &= \prod_{x \in X} (1 + t + t^2 + \ldots) \\
&= \sum_{\mu \in \calQ} C^\mu_{X/S}([X]) t^{\sum \mu} \\
&= \sum_{\mu \in \calQ} [C^\mu_{/S}(X)] t^{\sum \mu} \\
&= \sum_{n \geq 0} \left[ \bigsqcup_{\substack{ \mu \in \calQ \\ \sum \mu=n }} C^\mu_{/S}(X) \right ] t^n. 
\end{align*}
In characteristic zero, the disjoint union in the $n$th term can be identified with the decomposition of $\Sym^n_{/S} X$  by multiplicity partitions, and we obtain
\begin{equation}\label{equation:euler-kapranov} \prod_{x \in X} \frac{1}{1-t} = \sum_{n\geq 0} [\Sym_{/S}^n X] t^n = Z_{X/S}^\Kap(t). \end{equation}
In positive characteristic, the natural map from the multiplicity partition to the symmetric power is radicial surjective, but not always a piecewise isomorphism (cf., e.g., \cite[Remark at end of section 4]{ekedahl:geometric-invariant}). Thus, in positive characteristic (\ref{equation:euler-kapranov}) only holds after passing to the quotient $\widetilde{K}_0(\Var/S)$.
\end{example}


\subsubsection{Motivic Euler products with other constant coefficients}
By \cite[Section 3.10]{bilu:motiviceulerproducts}, we can also allow motivic Euler products to have a finite number of factors with constant term not equal to one. For later reference, we  give here a formula for this in a special case sufficient for our purposes: we assume $X$ to be a variety with a quasi-finite morphism to a variety~$S$ with fibers of size $m$, and we consider a family of elements $(a_i)_{i\geq 0}$ of $K_0(\Var/X)$. Then
\begin{equation}\label{equation:eulerproductconstantcoef}\prod_{x\in X/S}\left(a_{0,x} + a_{1,x}t_1 + a_{2,x}t_2 + \ldots \right) = \sum_{\substack{\mu = (m_i)_{i\geq 0}\\ \sum m_i = m}}\left(\prod_{i\geq 0}\Sym^{m_i}_{X/S}(a_{i})\right)_{*,X/S} \mathbf{t}^{\mu}.\end{equation}
where $\mathbf{t}^\mu = \prod_{i\geq 1}t_i^{m_i}$ (in particular, $\mathbf{t}^\mu$ does not depend on $m_0$). 
\subsection{Coefficients in localized and completed Grothendieck rings}
By \cite{bilu:motiviceulerproducts}, lemma 3.7.0.2, the above constructions can also be carried out in the localized Grothendieck ring $\M_X$, in such  a way that for every $a\in K_0(\Var/X)$ and any integers $n,m$ with $n\geq 0$,
\begin{equation}\label{eq:grothringlocalisation} \Sym^n_{X/S}(a\LL^{m})= \Sym^{n}_{X/S}(a)\LL^{nm} \textrm { in } \calM_X. 
\end{equation}

\begin{example}\label{example:labeled-conf-localization}
Invoking (\ref{eq:grothringlocalisation}) and working on decompositions of $X$, we find that for any constructible $g: X \rightarrow \bbZ$,
\[ \Sym^n_{X/S}(a\LL^{g})= \Sym^{n}_{X/S}(a)\Sym^n( \bbL^{g} ) \textrm{ and } \Sym^n_{X/S}(\bbL^{-g}) = \left(\Sym^n_{X/S}(\bbL^{g})\right)^{-1}  \]
in  $\calM_{\Sym^n_{/S} X}$. Applying this to the definition of $\Conf^\mu_{X/S}$, we find that, 
\[ \Conf^\mu_{X/S}( a \bbL^g ) = \Conf^\mu_{X/S}(a) \Conf^\mu_{X/S}(\bbL^g) \textrm{ and } \Conf^\mu_{X/S}(\bbL^{-g}) = \left(\Conf^\mu_{X/S}(\bbL^{g})\right)^{-1}  \]
in $\calM_{\Conf^\mu_{/S}(X)}$. When $g: X \rightarrow \bbZ_{\geq 0}$, the first identity holds without inverting~$\bbL$. 
\end{example}

We can also extend our constructions to the completion $\widehat{\M}_X$. Indeed, note first that for an element $c\in \Fil^n\M_X$, we have, for every integer $i\geq 1$, $\Sym_{X/S}^i(c)\in \Fil^{in}\M_{\Sym_{/S}^iX}$. Now, for any  $a\in \Fil^0\widehat{\M}_X$ and any integer $n\geq 1$, we may write $a= b + c$ where $b\in \Fil^0\M_X$ and $c\in \Fil^n\widehat{\M}_X$, and define  $$\Sym_{X/S}^m(a) = \Sym_{X/S}^m(b)\  (\mathrm{mod}{\ \Fil^n\widehat{\M}_{\Sym_{X/S}^mX}}).$$ This can easily be seen to not depend on the choice of $b$ and $c$, and doing this for arbitrarily large $n$ gives a well-defined element $\Sym_{X/S}^m(a)\in \widehat{\M}_{\Sym_{/S}^mX}$. If $a\not\in \Fil^0\widehat{\M}_X$, we can reduce to the previous case by multiplying by an appropriate power of $\LL$.

We can then prove that the map $\Sym_{X/S}$ from proposition \ref{prop:sympower} extends to a group morphism
$$\Sym_{X/S}: \widehat{\M}_X\to \widehat{\M}_{\Sym^{\bullet}_{/S}X}^1$$
where $\widehat{\M}_{\Sym^{\bullet}X}^1$ is defined analogously to $K_0(\Var/\Sym^{\bullet}_{/S}X)^1$. 

\subsection{An important Euler product}
The following formula for a special case of a motivic Euler product will be important for us going forward. It is an immediate consequence of the formula for $\Sym^n_{X/S}(-a)$ in Example \ref{example:sympower}:
\begin{proposition}\label{proposition:fundamental-euler-product}
Let $a \in \calM_X$. Then,
\begin{align*} \prod_{x \in X/S} \left( 1 - a_x t \right) & = \sum_{\mu\in \calQ}(-1)^{||\mu||} \cdot \Conf^{\mu}_{X/S}(a) \cdot t^{|\mu|} \textrm{ in } \calM_S[[t]].
\end{align*}
If $a \in K_0(\Var/X)$, then the identity holds already in $K_0(\Var/S)[[t]].$ 
\end{proposition}

\begin{example}\label{example.kapranov-zeta-inverse}
If $a=1 \in K_0(\Var/X)$, we obtain 
\[ \prod_{x\in X/S}(1-t) = \sum_{\mu\in \calQ}(-1)^{||\mu||} \cdot [\Conf^{\mu}_{/S}(X)] \cdot t^{|\mu|} \]
in $K_0(\Var/S)[[t]]$. By multiplicativity of motivic Euler products, we have 
\[ \prod_{x\in X/S}(1-t) = \left(\prod_{x\in X/S}\frac{1}{1-t}\right)^{-1}\]
and thus, by Example \ref{example:kapranov-zeta}, in characteristic zero we obtain
\begin{equation}\label{equation.inverse-zeta-formula} \left( Z_{X/S}^\Kap(t) \right)^{-1} = \sum_{\mu\in \calQ}(-1)^{||\mu||} \cdot [\Conf^{\mu}_{/S}(X)] \cdot t^{|\mu|} \end{equation}
in $K_0(\Var/S)[[t]]$. In positive characteristic, we obtain this formula after passing to $\widetilde{K_0}(\Var/S)[[t]]$. The formula (\ref{equation.inverse-zeta-formula}) was first established by Vakil-Wood \cite[Proposition 3.7]{vakil-wood:discriminants}.
\end{example}

We highlight the following special case of Proposition \ref{proposition:fundamental-euler-product}: 
\begin{corollary}\label{corollary:fundamental-euler-product-coh-sheaf}
Let $Z/X$, and let $\calF$ be a coherent sheaf on $X$. Then, in $\calM_S$,
\begin{align*}  \prod_{x\in X/S} \left(1-\left(\frac{[Z]}{[\bbV(\calF)]}\right)_x t \right)& = \sum_{\mu\in \calQ}(-1)^{||\mu||} \cdot \Conf^{\mu}_{X/S}\left(\frac{[Z]}{[\bbV(\calF)]} \right) \cdot t^{|\mu|} \\
&= \sum_{\mu\in \calQ}(-1)^{||\mu||} \cdot \frac{[\Conf^{\mu}_{X/S}(Z)]}{[\Conf^{\mu}_{X/S}\bbV(\calF)]} \cdot t^{|\mu|} \\
 &= \sum_{\mu\in \calQ}(-1)^{||\mu||} \cdot \frac{[\Conf^{\mu}_{X/S}(Z)]}{[\bbV(\Conf^{\mu}_{X/S}(\calF))]}  \cdot t^{|\mu|}, \end{align*}
where on the last two lines the division in each term takes place in $\calM_{C^\mu_{/S}(X)}$ before applying the forgetful map to $\calM_S$. 
\end{corollary}
\begin{proof}
By Theorem \ref{theorem:V-structure}, $[\bbV(\calF)]=\bbL^{\dim \calF}$. Thus, the second equality follows from Example \ref{example:labeled-conf-localization}, and the last equality follows from Lemma \ref{lemma:compatibility-conf-realization}.
\end{proof}

\subsection{Convergence and evaluation at $\LL^{-N}$}
\begin{definition} Let $F(t) = \sum_{i\geq 0} a_it^i\in \widehat{\mathcal{M}}_S[[t]]$ be a power series. We define its radius of convergence to be 
$$\sigma_F = \limsup_{i} \frac{\dim_{/S} a_i}{i}.$$
For every integer $N>\sigma_F$, we then say that $F$ converges at $t = \LL^{-N}$, and we may evaluate $F$ at $\LL^{-N}$, to obtain an element $F(\LL^{-N})\in \widehat{\M}_S$. 
\end{definition}

\begin{notation}\label{notation:evaluation} If $F(0) = 1$ and the motivic Euler product $P(t) = \prod_{x\in X/S}F_x(t)$ converges at $t = \LL^{-N}$, we write $$\left. \prod_{x\in X/S}F_x(t)\right|_{t = \LL^{-N}}\ \ \ \text{or}\ \ \ \prod_{x\in X/S}F_x(\LL^{-N})$$
for the value $P(\LL^{-N})$. 
\end{notation}

\begin{example}\label{example:convergence} Let $Z$ be a variety over $X$, $r$ an integer, and consider the motivic Euler product 
$$\prod_{x\in X/S}(1 - [Z]_x\LL^{-m}t).$$
By Proposition \ref{proposition:fundamental-euler-product} and property (\ref{eq:grothringlocalisation}), we have 
$$\prod_{x\in X/S}(1 - [Z]_x\LL^{-m}t) =  \sum_{\mu\in \calQ}(-1)^{||\mu||} \Conf^{\mu}_{X/S}(Z) \LL^{-m|\mu|}t^{|\mu|}$$
so that the dimension of the coefficient of degree $k$ is bounded by 
$k (\dim_{/S} Z-m)$. In particular, the radius of convergence of this motivic Euler product is bounded by $\dim_{/S} Z - m$, and it can be evaluated for $t = \LL^{-\dim Z+m -1}$. 
\end{example}
\begin{remark} One can formulate very general convergence criteria for motivic Euler products, as is done in \cite{bilu:thesis}. However, in this paper we will only need \ref{example:convergence}.
\end{remark}
We now turn to the case of series with several indeterminates, and will deal with evaluating one of them at $\LL^{-N}$. If $F$ is a power series in the variables $t, u_1,u_2,\ldots$ with coefficients in $\widehat{\M}_X$, we may write it uniquely in the form
\begin{equation}\label{eq:seriesint}F(t, (u_i)_{i\geq 1}) = \sum_{\mu\in \calQ_0} F_{\mu}(t)\mathbf{u}^{\mu}\in \widehat{\M}_X[[t, (u_i)_{i\geq 1}]].\end{equation}
Then, for any integer $N\leq \inf_{\mu} \sigma_{F_{\mu}}$, we say that $F$ converges at $t = \LL^{-N}$, and we may evaluate it to obtain an element $\left.F(t, (u_i)_{i\geq 1})\right|_{t = \LL^{-N}} = F(\LL^{-N}, (u_i)_{i\in I})\in \widehat{\M}_X[[(u_i)_{i\geq 1}]].$

We now check that evaluating Euler products with factors of this form works as expected, that is, when evaluating the Euler product one in fact obtains the Euler product where all factors have been evaluated. We start by treating the case when every monomial where $t$ occurs contains also some other indeterminate, so that evaluating $t$ doesn't change the constant terms of the factors.

\begin{proposition} \label{prop:evaluationwith constantcoef1}Keep the above notation and assume that $F_0 = 1$ and that $N$ is an integer such that $N \leq \inf_{\mu} \sigma_{F_{\mu}}.$ Then the Euler product 
$$\prod_{x\in X/S} F_{x}(t, (u_i)_{i\geq 1})$$
converges at $t = \LL^{-N}$, and in $\widehat{\M}_S$, we have
$$\left.\left(\prod_{x\in X/S} F_x(t, (u_i)_{i\geq 1})\right)\right|_{t = \LL^{-N}} = \prod_{x\in X/S} F_x(\LL^{-N},(u_i)_{i\geq 1}).$$
\end{proposition}

\begin{proof} For every non-zero $\mu$, write $F_{\mu}(t) = \sum_{i\geq 0} f_{\mu,i}t^{i}$. Then, denoting by $I$ the semigroup $(\calQ_0\setminus \{0\})\times \bbZ_{\geq 0}$, we have 
\begin{eqnarray*}\prod_{x\in X/S} F_{x}(t, (u_{i})_{i\geq 1}) & = & \prod_{x\in X/S} \left(1 + \sum_{(\mu,i)\in I} f_{\mu, i, x} t^i \mathbf{u}^{\mu}\right) \\
& = & \sum_{(n_{\mu,i})_{\mu,i}\in \mathcal{P}(I)}\left( \prod_{(\mu,i)\in I}\Sym_{X/S}^{n_{\mu,i}}f_{\mu,i}\right)_{*}t^{\sum_{\mu,i}n_{\mu,i}i} \mathbf{u}^{\sum_{\mu,i}n_{\mu,i}\mu}.\\
& = &  \sum_{\nu\in \calQ_0}\left(\sum_{\substack{(n_{\mu,i})_{\mu,i}\in \mathcal{P}(I)\\ \nu = \sum_{\mu} (\sum_{i}n_{\mu,i})\mu}}\left(\prod_{(\mu,i)\in I}\Sym_{X/S}^{n_{\mu,i}}f_{\mu,i}\right)_{*}t^{\sum_{\mu,i}n_{\mu,i}i}\right)\mathbf{u}^{\nu}.
\end{eqnarray*}
On the other hand,
\begin{eqnarray*} \prod_{x\in X/S} F_{x}(\LL^{-N}, (u_i)_{i\geq 1}) & = & \prod_{x\in X/S} \left(1 + \sum_{\mu\in \calQ_0\setminus \{0\}} F_{\mu}(\LL^{-N})_x \mathbf{u}^{\mu}\right) \\
& = & \sum_{\nu\in \calQ_0} \left( \sum_{\substack{(n_{\mu})_{\mu}\\ \nu = \sum_{\mu}n_{\mu}\mu}}\left(\prod_{\mu} \Sym_{X/S}^{n_{\mu}}(F_{\mu}(\LL^{-N}))\right)_{*}\right) \mathbf{u}^{\nu}.\\
\end{eqnarray*}
For every $\mu$ and every $n_{\mu}$, we have
$$\Sym_{X/S}^{n_{\mu}}(F_{\mu}(\LL^{-N})) = \Sym_{X/S}^{n_{\mu}}\left(\sum_{i\geq 0}f_{\mu,i}\LL^{-Ni}\right) = \sum_{\substack{(n_{\mu,i})_{i\geq 0} \\ \sum_{i}n_{\mu,i} = n_{\mu}}}\prod_{i} \Sym_{X/S}^{n_{\mu,i}}(f_{\mu,i}\LL^{-Ni}).$$

We may conclude using the fact that $\Sym_{X/S}^{n_{\mu,i}}(f_{\mu,i}\LL^{-Ni}) = \Sym_{X/S}^{n_{\mu,i}}(f_{\mu,i}) \LL^{-in_{\mu,i}N}$ by (\ref{eq:grothringlocalisation}). 

\end{proof}
Assume now more generally that we are given a series $F(t, (u_i)_{i\geq 1})$ as in (\ref{eq:seriesint}), such that $F_{0}(t)$ is of the form $F_0(t) = 1 + tG_0(t)$ for some $G_0\in \widehat{\mathcal{M}}_X[[t]]$. Denoting $G_{\mu}(t) = \frac{F_{\mu}(t)}{1 + tG_0(t)}\in \widehat{\mathcal{M}}_X[[t]]$ we may then write
$$F(t, (u_i)_{i\geq 1}) = (1 + tG_0(t)) \left(1 + \sum_{\mu\in \calQ_0\setminus\{0\}}G_{\mu}(t) \mathbf{u}^{\mu}\right),$$
so that we reduce to cases treated previously: for every $n$ such that the motivic Euler product $\prod_{x\in X}(1 + tG_0(t))$ converges at~$\LL^{-N}$ and such that $N\leq \inf \sigma_{G_{\mu}}$, we may write, taking Euler products and evaluating at~$\LL^{-N}$, 
$$\left.\prod_{x\in X/S} F_x(t, (u_i)_{i\geq 1})\right|_{t=\LL^{-N}}$$
$$ = \prod_{x\in X/S} (1 + \LL^{-N}G_{0,x}(\LL^{-N}))\prod_{x\in X/S} \left(1 + \sum_{\mu\in \calQ_0\setminus \{0\}}G_{\mu,x}(\LL^{-N}) \mathbf{u}^{\mu}\right),$$
where we used notation \ref{notation:evaluation} on the first factor and proposition \ref{prop:evaluationwith constantcoef1} on the second.


\subsection{Monomial substitutions in motivic Euler products}\label{section:substitutions}  Let  $(a_{i,j})_{\substack{1\leq i \leq n\\ 1 \leq j \leq m}}$ be a family of non-negative integers such that for every $i$, $a_{i,j}$ is non-zero for at least one value of $j$, and let
$$F: \bbZ[[t_1,\ldots,t_n]] \to \bbZ[[s_1,\ldots,s_m]]
$$
be the monomial substitution $t_i \mapsto s_1^{a_{i,1}}\ldots s_m^{a_{i,m}}$. We also denote by $F$ the morphism it induces on any ring of power series in $t_1,\ldots,t_n$ with coefficients in a Grothendieck ring of varieties. We denote by $I = (\bbZ_{\geq 0})^{n}\setminus \{0\}$ and $J =   (\bbZ_{\geq 0})^{m}\setminus\{0\}$.
\begin{lemma}\label{lemma:monomialsubstitution} Let $\mathscr{X} = (X_{\i})_{\i\in I}$ be a family of elements of $\M_X$. Then
$$F\left(\prod_{x\in X/S} \left( 1 + \sum_{\i\in I} X_{\i,x} \t^{\i}\right)\right) = \prod_{x\in X/S} F\left( 1 + \sum_{\i\in I} X_{\i} \t^{\i}\right)_x.$$
\end{lemma}
\begin{proof} We denote by $A$ the matrix $(a_{i,j})_{i,j}$. For every line vector $\j = (j_1,\ldots,j_m)\in J$, we denote by $I_{\j}$ the set $\{\i = (i_1,\ldots,i_n)\in I,\ \i \cdot A = \j\}$. 

We first expand the left-hand side:
\begin{eqnarray*}F\left(\prod_{x\in X/S} \left( 1 + \sum X_{\i,x} \t^{\i}\right)\right)& = & F\left( \sum_{\lambda = (\ell_{\i})_{\i}} \Sym^{\lambda}_{X/S}(\mathscr{X}) \prod_{\i} (t^{\i})^{\ell_{\i}}\right)\\
& = & \sum_{\lambda = (\ell_{\i})_{\i}} \Sym^{\lambda}_{X/S}(\mathscr{X})\prod_{\i}\left(\prod_{k=1}^n (s_1^{a_{k,1}}\ldots s_{m}^{a_{k,m}})^{i_k}\right)^{\ell_{\i}} \\
& = & \sum_{\lambda = (\ell_{\i})_{\i}} \Sym^{\lambda}_{X/S}(\mathscr{X}) \prod_{j=1}^m s_j^{\sum_{\i, k} a_{k,j}i_k \ell_{\i}}\\
& = & \sum_{\k= (k_1,\ldots,k_m)} \s^{\k} \left(\sum_{\substack{(\ell_{\i})_{\i} \\ \sum_{\i} \ell_{\i} \i\cdot A = \k}} \left(\prod_{\i}\Sym^{\ell_{\i}}_{X/S}(X_{\i}) \right)_{*}\right)
\end{eqnarray*}

On the other hand,
\begin{eqnarray*}\prod_{x\in X/S} F\left( 1 + \sum_{\i\in I} X_{\i} \t^{\i}\right)_x& = & \prod_{x\in X/S} \left( 1 + \sum_{\i\in I} X_{\i,x} \prod_{k=1}^n s_1^{a_{k,1}i_k}\ldots s_m^{a_{k,m}i_k}\right)\\
& = & \prod_{x\in X/S} \left( 1 + \sum_{\i\in I} X_{\i,x}\prod_{j=1}^m s_j^{\sum_{k=1}^n a_{k,j} i_k }\right)\\
& = & \prod_{x\in X/S} \left( 1 + \sum_{\j\in J}\left( \sum_{\substack{\i \\ \i \cdot A = \j}} X_{\i,x}\right) \s^{\j}\right)\\
& = & \sum_{\nu  = (n_{\j})_{\j\in J}}\left( \prod_{\j} \Sym_{X/S}^{n_{\j}}\left(\sum_{\i\cdot A = \j} X_{\i}\right)\right)_{*}\prod_{\j} \s^{n_{\j}\, \j}\\
& = & \sum_{\k = (k_1,\ldots,k_m)} \s^{\k} \left( \sum_{\substack{(n_{\j})_{\j}\\ \sum_{\j}n_{\j}\, \j =\k}} \left(\prod_{\j} \Sym_{X/S}^{n_{\j}}\left(\sum_{\i\cdot A = \j} X_{\i}\right)\right)_{*}\right)\\
& = &  \sum_{\k = (k_1,\ldots,k_m)} \s^{\k} \left( \sum_{\substack{(n_{\j})_{\j}\\ \sum_{\j}n_{\j}\j =\k}} \left(\prod_{\j} \sum_{\substack{(n_{\j,\i})_{\i\in I_{\j}}\\ \sum n_{\j,\i} = n_{\j}}} \prod_{\i\in I_{\j}} \Sym_{X/S}^{n_{\j,\i}}(X_{\i})\right)_* \right).\\
\end{eqnarray*}
This is seen to be equal to the left-hand side by putting $\ell_{\i}:= n_{\i \cdot A,\i}$ for every $\i$.
\end{proof}
\begin{remark} According to \cite[Section 3.8, Property 4]{bilu:thesis}, we may in fact more generally make substitutions of the form $t_i \mapsto \LL^{n_i}s_1^{a_{i,1}}\ldots s_m^{a_{i,m}}$ where  $n_i\in \bbZ$. This is specific to affine spaces and relies on property (\ref{eq:grothringlocalisation}).
\end{remark}
\begin{remark} Motivic Euler products do not allow for non-monomial substitutions (e.g., such that an indeterminate is sent to a polynomial which is not a single monomial). As an easy counterexample, consider the series
$$\prod_{x\in X}( 1+ s_1 + s_2)  = \sum_{n_1,n_2} \left(\Sym^{n_1}(X)\times \Sym^{n_2}(X)\right)_*s_1^{n_1}s_2^{n_2} \in K_0(\Var/K)[[s_1,s_2]]
$$

The coefficient of $s_1s_2$ in this series is $(X^2)_{*}$ (that is, the complement of the diagonal in $X^2$).
On the other hand, 
$$\prod_{x\in X}(1 + t) = \sum_{n\geq 0} (\Sym^n(X))_{*}t^n,$$
so that if we substitute $t = s_1 + s_2$, we get that the coefficient of $s_1s_2$ will be $2(\Sym^2(X))_*$. This is in general not equal to $(X^2)_*$ in $K_0(\Var/K)$: for example, if $X = \LL$, we have
$(X^2)_* = \LL^2 - \LL$, and $2(\Sym^2(X))_* = 2(\Sym^2X - X) = 2(\LL^2 - \LL)$. 

In the same way, Euler products are usually not compatible with the operation $t\mapsto -t$, which can easily be seen by comparing the series
$$\prod_{x\in \bbA^1}(1-t) = Z_{\bbA^1}^{\Kap}(t)^{-1} = 1-\LL t$$
and $$\prod_{x\in \bbA^1}(1+t) = Z_{\bbA^1}^{\Kap}(t)Z_{\bbA^1}^{\Kap}(t^2)^{-1} = \frac{1-\LL t^2}{1-\LL t}.$$
\end{remark}


\section{Motivic inclusion-exclusion}\label{sect.motivic-inclusion-exclusion}
In this section, we give a general motivic inclusion-exclusion principle (Theorems \ref{theorem:mot-inc-exc-exact} and \ref{theorem:mot-inc-exc-approx} below) that will play a central role in the proofs of our main results.

\subsection{Inclusion-exclusion for finite sets}
\newcommand{\finset}{\mathrm{Fin.-Set}}

We begin by revisiting the classical inclusion-exclusion formula for finite sets from a categorical perspective; there is no mathematical necessity for this detour, but it will be helpful in guiding us to the motivic version. 

In its simplest form, inclusion-exclusion gives a way to compute the cardinality of a union of finite sets:  one starts with the naive guess given by adding up the cardinalities of each set, then iterates corrections for overcounting caused by overlap between the sets. The resulting statement is useful because it is often easier to count the elements of an intersection than the elements of a union. 
\begin{theorem}[Inclusion-exclusion of finite sets]
Let $Y$ be a finite set, and let $Y_i \subset Y,\, i \in \{1, \ldots, n\} $ be a collection of subsets. Then,
\begin{multline*}
\left| \bigcup_{i=1}^n Y_i \right| = \left( |Y_1| + |Y_2| + \ldots + |Y_n| \right)\\ - \left( |Y_1 \cap Y_2| + |Y_1 \cap Y_3| + \ldots +  |Y_{n-1} \cap Y_n| \right) \\ + \left( |Y_1 \cap Y_2 \cap Y_3| + |Y_1 \cap Y_2 \cap Y_4| + \ldots  \right) \\
\ldots
\\ + (-1)^n \left( |Y_1 \cap Y_2 \cap \ldots \cap Y_n|  \right) \\
= \sum_{k=1}^n (-1)^{k-1} \sum_{\substack{J \subset \{1, \ldots, n \} \\ |J|=k}} \left| \bigcap_{j \in J} Y_j \right|. 
\end{multline*}
\end{theorem}
 
Note that cardinality determines the class of a finite set in the Grothendieck ring of finite sets $K_0(\finset)=\bbZ$, so that we can view inclusion-exclusion as a statement in this Grothendieck ring. In that light, it admits the following refinement:

\begin{theorem}\label{theorem:cat-inc-exc-finite-set}[Categorified inclusion-exclusion of finite sets]
Let $f: X \rightarrow Y$ be a map of finite sets. Then, in $K_0(\finset/Y)$, 
\[ [f(X)] = \sum_{k\geq 1} (-1)^{k-1} [ \Conf^k_{/Y} (X)] \]
where $\Conf^{k}_{/Y}(X)$ denotes the configuration space of $k$ unordered elements of $X$ relatively to $Y$. 
\end{theorem} 

To obtain the original inclusion-exclusion formula from the categorified inclusion-exclusion formula, one applies the categorified version to the natural map $\bigsqcup Y_i \rightarrow Y$, and then applies the forgetful map $K_0(\finset/Y) \rightarrow K_0(\finset)=\bbZ$. 

\begin{proof}[Proof of Theorem \ref{theorem:cat-inc-exc-finite-set}]
We first observe that it suffices to work over each $y \in Y$ individually, since the assignment 
\[ [X] \mapsto ([X_y])_{y\in Y} \] 
gives an isomorphism 
\[ K_0(\finset/Y) \xrightarrow{\sim} \prod_{y \in Y} K_0(\finset);\] in other words,
\[ K_0(\finset/Y)= \bbZ^{Y}. \] 
Thus, we may assume that $Y$ is a point, and work in $K_0(\finset)$. The statement then says that if $X$ is a finite set,
\[ \sum_k (-1)^{k-1} \binom{ |X|}{k} = \begin{cases} 1 & \textrm{ if } X \neq \emptyset \\ 0 & \textrm{ if } X = \emptyset. \end{cases} \]
and indeed this is the case: 
\[ \sum_k (-1)^{k-1} \binom{ |X|}{k} = \left. \left(1 - (1-t)^{|X|} \right)\right|_{t=1}. \]

\end{proof} 

\subsection{Motivic inclusion-exclusion: statements}

The statement of motivic in\-clu\-sion-exclusion is similar to the categorified inclusion-exclusion for finite sets, however, one must keep track of \emph{all} labeled configuration spaces:

\begin{theorem}[Motivic inclusion-exclusion: exact version]\label{theorem:mot-inc-exc-exact}
Let $f: X \rightarrow Y$ be a quasi-finite map of varieties over $K$. Then, in $K_0\left((\Var / Y)_\RS\right)$, 
\[ [f(X)] = \sum_{\mu \in \calQ\bs \{\emptyset\}}  (-1)^{||\mu|| - 1}  [\Conf^\mu_{/Y}(X) ], \]
or, equivalently,
\[ 1 - [f(X)] = \sum_{\mu \in \calQ} (-1)^{||\mu||}  [\Conf^\mu_{/Y}(X) ] \]
where the terms in the sums are identically zero for $|\mu|$ sufficiently large. 
\end{theorem}

\begin{example}
We give two important examples to illustrate this Theorem and its relation with inclusion-exclusion for finite sets:
\begin{enumerate}
\item Suppose $Y$ is a finite set, and $X = \bigsqcup_{i=1}^n Y_i$ for subsets $Y_i \subset Y$. We may view these finite sets as varieties over $K$ in the natural way (i.e. each point is a copy of $\Spec K$). Then, the motivic inclusion-exclusion formula collapses to the classical inclusion-exclusion formula: for example, if $n=2$, then motivic inclusion-exclusion gives the following identity in $K_0(\Var/K)$:
\begin{align*}
 [Y_1 \cup Y_2] & = \underbrace{[\Conf^{1^1}_{/Y}(X)]}_{=[Y_1] + [Y_2]} + \underbrace{[\Conf^{1^2}_{/Y}(X)]}_{=[Y_1 \cap Y_2]} - \underbrace{[\Conf^{1^1 2^1}_{/Y}(X)]}_{=2 [Y_1 \cap Y_2]} \\
 &= [Y_1] + [Y_2] - [Y_1 \cap Y_2]
\end{align*}
For general $n$, a combinatorial proof of the collapse of terms can be deduced by applying \cite[Lemma 6.1.5]{howe:mrv2} when $\mu_1'=\ldots=\mu_k'=1^1$ for each $k \leq n$ in order to cancel the terms with $|\mu|=k$. 

\item For general maps of varieties, we do not see the same collapse to a simpler formula: for example, when $X=\Spec \bbF_{q^2}$ and $Y= \Spec \bbF_q$, we obtain
\[ [\Spec \bbF_q] = \underbrace{[\Conf^{1^1}_{/Y}(X)]}_{=[\Spec \bbF_{q^2}]} + \underbrace{[\Conf^{1^2}_{/Y}(X)]}_{=[\Spec \bbF_{q}]} - \underbrace{[\Conf^{1^1 2^1}_{/Y}(X)]}_{= [\Spec \bbF_{q^2}]}. \]
In particular, the terms with $|\mu|=2$ do not collapse here to give a single copy of $\Conf^{1^2}_{/Y}(X)$. 
\end{enumerate}
\end{example}

\begin{remark}\label{remark:mot-inc-exc-PWvRS}
One could ask whether motivic inclusion-exclusion holds already in $K_0(\Var/Y)$ (it does in characteristic zero, since then $K_0(\Var/Y)=K_0((\Var/Y)_\RS)$, however, it seems likely that this fails in positive characteristic. For example, if we apply it to the map $\Spec \bbF_2(t^{1/2}) \rightarrow \Spec \bbF_2(t)$, then motivic inclusion-exclusion has only one non-zero term, $\mu=1^1$, and thus if it held it would imply
\[ [\Spec \bbF_2(t^{1/2})] = [\Spec \bbF_2(t)] \textrm { in } K_0\left(\Var/ \bbF_2(t)\right). \]
We do not know whether this equality holds or not -- cf. Remark \ref{remark:RSvsPW}. \end{remark}

It will also be useful to consider the following approximate version of motivic inclusion-exclusion, which applies even when the map is not quasi-finite. To state it, we need the following notation: for any partition $\mu$, and non-negative integer~$k$, let $\Conf^{\mu,\geq k}_{/Y}(X)$ be the constructible subset of $\Conf^\mu_{/Y}(X)$ whose geometric points are described by giving a geometric point $y$ of $Y$ having at least $|\mu|+k$ pre-images in~$X$ and a $\mu$-configuration of points in the pre-image of $y$. Concretely, this is the constructible subset of $\Conf^{\mu}_{/Y}(X)$ obtained as the image of $\Conf^{\mu \cdot *^k}_{/Y}(X)$ under the map ``forget the $*$-labeled points". Noting that $\Conf^{\emptyset}_{/Y}(X)=Y$, we also write $Y^{\geq k}$ for the space $\Conf^{\emptyset, \geq k}_{/Y}(X)$. 

\begin{theorem}[Motivic inclusion-exclusion: approximate version]\label{theorem:mot-inc-exc-approx}
Let ${f:X \rightarrow Y}$ be a map of varieties over $K$. For $m>0$, in $K_0\left((\Var/Y)_\RS \right)$,
\[ [f(X)] - [Y^{\geq m}] = \sum_{\substack{\mu \in \calQ\bs \{\emptyset\} \\ |\mu| < m}} (-1)^{||\mu|| - 1} \left( [\Conf^\mu_{/Y}(X)] - [\Conf^{\mu,\geq m-|\mu|}_{/Y}(X)] \right), \]
or, equivalently,
\[ 1 - [f(X)] = \sum_{\substack{\mu \in \calQ \\ |\mu| < m}} (-1)^{||\mu||} \left( [\Conf^\mu_{/Y}(X)] - [\Conf^{\mu,\geq m-|\mu|}_{/Y}(X)] \right). \]
\end{theorem}

\begin{remark}
By applying the forgetful map, the identities in Theorems \ref{theorem:mot-inc-exc-exact} and \ref{theorem:mot-inc-exc-approx} also hold in $K_0(\Var/K)$ (or, more generally, if $f:X \rightarrow Y$ is a map of varieties over $S$, in $K_0(\Var/S)$). This is particularly useful for the approximate version, Theorem \ref{theorem:mot-inc-exc-approx}, which we will use in situations where the error terms can be shown to have small dimension over the base $S$ compared to the other terms, so that they become negligible in the dimension topology.  
\end{remark}

Theorem \ref{theorem:mot-inc-exc-exact} is a direct consequence of Theorem \ref{theorem:mot-inc-exc-approx} since for a quasi-finite map the error terms will vanish for $m$ sufficiently large. In the other direction, Theorem \ref{theorem:mot-inc-exc-approx} can be deduced from Theorem \ref{theorem:mot-inc-exc-exact} by restricting to the locus of points in $Y$ with fewer than $m$ pre-images under $f$; we will discuss this implication in more detail in \ref{subsec:second-proof-mot-inc-exc}. 

\subsection{First proof of motivic inclusion-exclusion}\label{subsection:mot-inc-exc-first-proof}

Motivic inclusion-exclusion was first used by Vakil-Wood in the proof
 of Theorem 1.1.3 of \cite{vakil-wood:discriminants}, where they applied it in the special case that $Y$ is the space of hypersurface sections of a fixed projective variety and $X$ is the space of hypersurface sections together with a marked singular point. The argument they give is a motivic analog of the proof of classical inclusion-exclusion of finite sets by starting with the naive guess and then correcting for overlap; in the motivic setting, however, there is less cancellation, and one must keep track of arbitrary configurations in the overlap. We reproduce their argument carefully here to obtain a proof of Theorem \ref{theorem:mot-inc-exc-approx}.

Before beginning the proof, we introduce one more piece of notation. For ${f:X \rightarrow Y}$ a map of varieties over $K$, we write 
\[ \Conf^{\mu, k}_{/Y}(X) = \Conf^{\mu, \geq k}_{/Y}(X) \backslash \Conf^{\mu, \geq k+1}_{/Y}(X), \]
the constructible subset of $\Conf^\mu_{/Y}(X)$ whose geometric points are described by giving a point $y \in Y$ having \emph{exactly} $|\mu|+k$ pre-images in $X$ and a $|\mu|$-configuration of points in the pre-image of $y$. In particular, by looking at geometric points we find
\begin{lemma}\label{lemma:conf-iso-RS} The natural map $\Conf^{\mu\cdot *^k,0}_{/Y}(X) \rightarrow \Conf^{\mu, k}_{/Y}(X)$ is an isomorphism in $(\Var/Y)_\RS$.   
\end{lemma}

\begin{remark} \label{remark:vakil-wood-error}
The map in Lemma \ref{lemma:conf-iso-RS} is not in general a piecewise isomorphism if the map $f:X\rightarrow Y$ induces inseparable residue field extensions. This leads to a gap in the positive characteristic case of Vakil-Wood's proof of Theorem 1.1.3 in \cite{vakil-wood:discriminants}: to obtain equation (3.1) of \cite{vakil-wood:discriminants} they tacitly assume that
\[ [ \Conf^{\mu \cdot *^k,0}_{/Y}(X)]  = [\Conf^{\mu,k}_{/Y}(X)] \textrm{ in } K_0(\Var/K). \]
Even in their specific setup this is not clearly the case, as the map $X \rightarrow Y$ in question can involve inseparable extensions: for example, one case they consider has $K=\bbF_2$, $Y$ the affine space $H^0(\bbP^1, O(2))$, and $X \subset \bbP^1 \times Y$ the set of $(x, F)$ such that $V(F)$ is singular at $x$. Then, the $\bbF_2(t)$-point $u^2 - t v^2$ of $Y$ has a single pre-image in $X$ with residue field extension $\bbF_2(t^{1/2})/\bbF_2(t)$, and at this point motivic inclusion-exclusion in the regular/piecewise Grothendieck ring could fail as described in Remark \ref{remark:mot-inc-exc-PWvRS}. This gap can be fixed by passing to the quotient $K_0\left( (\Var/K)_\RS \right).$ 
\end{remark}

\begin{proof}[First proof of Theorem \ref{theorem:mot-inc-exc-approx}]
To keep our formulas clean, we will write $C^\bullet$ for $\Conf^\bullet_{/Y}(X)$. We first observe that, for any $\mu$ and $k \in \bbZ_{\geq 1}$, 
\[ C^\mu = C^{\mu,0} \sqcup C^{\mu,1} \sqcup \ldots \sqcup C^{\mu, k-1} \sqcup C^{\mu, \geq k}. \]
Applying Lemma \ref{lemma:conf-iso-RS}, we obtain 
\begin{equation}\label{eqn:Cmuexpression} [C^\mu ] =  [C^{\mu,0}] + [C^{\mu \cdot *,0}] + \ldots + [C^{\mu \cdot *^{k-1},0}] + [C^{\mu, \geq k}] \textrm{ in } K_0\left( (\Var/Y)_\RS \right). \end{equation}
We then apply (\ref{eqn:Cmuexpression}) iteratively to replace terms of the form $[C^{\mu,0}]$ with $C^{\mu}$ and terms involving larger partitions:
\begin{multline} 1 - [f(X)/Y] = [C^{\emptyset, 0}] \\
= [C^{\emptyset}] - ([C^{*,0}] + [C^{*^2,0}] + \ldots + [C^{*^{m-1},0}] + [C^{\emptyset, \geq m}]) 
 \\
= [C^\emptyset] - (([C^{*}] - ([C^{* \cdot \star, 0}] + \ldots + [C^{* \cdot \star^{m-2},0}] + [C^{*, \geq m-1}]) ) + \\ ([C^{*^2}] - ([C^{*^2 \cdot \star, 0}] + \ldots + [C^{*^2 \cdot \star^{m-3},0}] + [C^{*^2, \geq m-2}]) ) + \ldots ) \\
 = \ldots = \sum_{\mu \in \calQ, |\mu| < m} (-1)^{||\mu||} \left( [C^\mu] - [C^{\mu,\geq m-|\mu|}] \right). 
\end{multline}
Solving for $[f(X)/Y]$, we obtain the result. 

\end{proof}

As observed earlier, Theorem \ref{theorem:mot-inc-exc-approx} easily  implies Theorem \ref{theorem:mot-inc-exc-exact}, so that we have now also proven the latter result. 

\subsection{Second proof of motivic inclusion-exclusion.}\label{subsec:second-proof-mot-inc-exc} We now give a proof of motivic inclusion-exclusion modeled on the proof of the categorified inclusion-exclusion for finite sets, Theorem \ref{theorem:cat-inc-exc-finite-set}, given above. In particular, motivic Euler products intervene in this proof (through motivic powers), which helps to explain why there is a connection between motivic Euler products and motivic inclusion-exclusion in the proofs of our main theorems. 

The following lemma plays a key role: 
\begin{lemma}\label{lemma:finite-variety-mot-zeta-pole}
Let $X \neq \emptyset $ be a finite variety over a field $K'$. Then
\[ \left.(1-t)^{[X]}\right|_{t=1} = 0 \] 
in $K_0\left( (\Var/K')_\RS \right)$. 
\end{lemma}
\begin{proof}
There is an integer $n>0$ such that for any algebraically closed field $L/K'$, $\# X(L) = n$ -- indeed, if we write 
\[ X=\bigsqcup \Spec K_i \] for a finite set of finite extensions $K_i/K'$, then $n$ is the sum of the separable degrees of $K_i/K'$. In particular, by computing on geometric points, we find
\begin{enumerate}
\item For $|\mu| > n$, $\Conf^\mu X=\emptyset$.
\item For $|\mu| < n$, the natural map 
\[ \Conf^{\mu \cdot *^{n-|\mu|}} (X) \rightarrow \Conf^{\mu}(X) \]
is radicial surjective. 
\end{enumerate}

Thus, in $K_0\left( (\Var/K')_\RS \right)$, 
\begin{align*} 
\left.(1-t)^{[X]}\right|_{t=1} & = \sum_{\mu \in \calQ, \; |\mu|\leq n} (-1)^{||\mu||} [\Conf^{\mu} X] \\
& = \sum_{\mu \in \calQ, \; |\mu|< n} (-1)^{||\mu||} [\Conf^{\mu} X] +\sum_{\mu \in \calQ, \; |\mu|= n} (-1)^{||\mu||} [\Conf^{\mu} X]  \\
& = \sum_{\mu \in \calQ, \; |\mu| < n} (-1)^{||\mu||} \left( [\Conf^\mu X] - [\Conf^{\mu \cdot *^{n-|\mu|}} X] \right) \\
& = 0. \end{align*}
\end{proof}

\begin{remark} Lemma \ref{lemma:finite-variety-mot-zeta-pole} can be read as saying that the motivic zeta function of a non-empty zero-dimensional variety always has a pole at $1$. 
\end{remark}

\begin{remark}
If $X$ is finite \'{e}tale (or even if a single $K_i$ in the decomposition $X=\bigsqcup_i \Spec K_i$ is separable over $K'$), then the argument for Lemma \ref{lemma:finite-variety-mot-zeta-pole} can be modified to show $(1-t)^{[X]}|_{t=1} = 0$ in $K_0(\Var/K')$. For $X=\Spec K_1$ for $K_1/K'$ inseparable, it is not known whether this identity holds; it again boils down to the question of equality between a field and a totally inseparable extension in the Grothendieck ring (cf. Remarks \ref{remark:vakil-wood-error}, \ref{remark:mot-inc-exc-PWvRS}, and \ref{remark:RSvsPW}.) 
\end{remark}

\begin{proof}[Proof of Theorem \ref{theorem:mot-inc-exc-exact}]
By Lemma \ref{lemma:injection-relative-product} and compatibility with base change for configuration spaces and images, it suffices to consider the case where the base is a field, $K'$. Then, $X$ quasi-finite is either equal to the empty set or non-empty and finite over $K'$, and the desired statement is that 
\[ \sum_{\mu \in \calQ\bs\{0\}} (-1)^{||\mu||-1} [\Conf^{\mu} (X)] = \begin{cases} 1 \textrm { if } X \neq \emptyset \\
0 \textrm{ if } X=\emptyset \end{cases} \textrm{ in } K_0\left( (\Var / K')_\RS \right). \]
The left-hand side is equal to $1 - (1-t)^{[X]}|_{t=1}$. We have $(1-t)^{[\emptyset]}=(1-t)^0=1$, so the desired result holds if $X=\emptyset.$ If $X$ is non-empty, then the desired result holds by Lemma \ref{lemma:finite-variety-mot-zeta-pole}. 
\end{proof}

\begin{remark}
It should be possible to axiomatize the properties satisfied by the categories $(\Var/K)_\RS$ and $\finset$ and then show that any such category admits a theory of Euler products and an inclusion-exclusion principle which are related as above. We hope to return to this question in future work.
\end{remark}

It remains to deduce Theorem \ref{theorem:mot-inc-exc-approx} from Theorem \ref{theorem:mot-inc-exc-exact} in order for this proof to stand as a complete alternative to the proof of motivic inclusion-exclusion offered in \ref{subsection:mot-inc-exc-first-proof}. 

\begin{proof}[Proof that Theorem \ref{theorem:mot-inc-exc-exact} implies Theorem \ref{theorem:mot-inc-exc-approx}]
Given $f:X \rightarrow Y$, consider the constructible subset $Y^{<m}= Y \backslash Y^{\geq m}$ whose geometric points are those with fewer than $m$ pre-images under $f$. We write $X^{<m} := X \times_Y Y^{<m}$ (which makes sense in $(\Var/K)_\RS$ ), and $f^{<m}: X^{<m} \rightarrow Y^{<m}$. Then, Theorem \ref{theorem:mot-inc-exc-exact} implies
that 
\begin{equation}\label{eqn:approx-inc-exc}
[f^{<m}(X^{<m})] = \sum_{\mu \in \calQ, |\mu|<m} (-1)^{||\mu||-1} [\Conf^\mu_{/Y^{<m}}(X^{<m})]. \end{equation}
This identity holds in $K_0(\Var/Y^{<m})$, and thus also in $K_0(\Var/Y)$. On the other hand, we find
\[ [f^{<m}(X^{<m})]=[f(X)]-[Y^{\geq m}] \]
and the natural map 
\[ \Conf^\mu_{/Y^{<m}}(X^{<m}) \rightarrow \Conf^\mu_{/Y}(X) \]
induces an isomorphism in $(\Var/Y)_\RS$ 
\[ \Conf^\mu_{/Y^{<m}}(X^{<m}) \rightarrow \Conf^\mu_{/Y}(X) \backslash \Conf^{\mu, \geq m-|\mu|} _{/Y}(X). \]
Plugging into (\ref{eqn:approx-inc-exc}), we obtain the desired formula:
\[ [f(X)]-[Y^{\geq m}] = \sum_{\mu \in \calQ, |\mu|<m} (-1)^{||\mu||-1} \left( [\Conf^\mu_{/Y}(X)] - [\Conf^{\mu, \geq m-|\mu|} _{/Y}(X)] \right). \]

\end{proof}

\section{Independent events in motivic probability} \label{section.indepedent-events}

In this section we develop some basic formalism for studying independence in motivic probability theory. The fundamental question we want to make precise (and then answer) is, given an algebraic family of $m$-wise (e.g. pair-wise, triple-wise, etc.) independent events, how can we approximate the motivic probability of their union? Before delving into the motivic setting, we give a brief overview of the classical phenomenon we are hoping to imitate. We only develop the minimum amount of material on motivic probability needed to treat the problem at hand with conceptual clarity; other foundational issues will be treated in more detail in \cite{bilu-howe:motivic-random-variables}.

\subsection{The underlying probabilistic mechanism}
\subsubsection{Crash course in probability}
Recall that a (classical) probability space is a measure space $\Omega$ equipped with a measure $\nu$ such that $\nu(\Omega)=1$. Measurable sets of $\Omega$ are called \emph{events}, and the probability of an event $E \subset \Omega$ is $\nu(E)$. Measurable functions are called random variables, and given a random variable $X$, its expectation is defined to be
\[ \bbE [ X ] := \int_\Omega X d\nu. \]
A family of events $E_1$, \ldots, $E_n$ are \emph{independent} if, for any $I \subset \{ 1, \ldots, n \}$, 
\begin{equation}\label{eq:independence} \nu \left( \bigcap_{i \in I} E_i \right) = \prod_{i \in I} \nu(E_i). \end{equation}
 They are said to be \emph{$m$-wise independent} if (\ref{eq:independence}) is true for all $I$  with $|I|\leq m$. 

\subsubsection{Unions and independence}
Let $(\Omega, \nu)$ be a probability space, and suppose $E_1, \ldots, E_n$ are events. If the $E_i$ are independent, then we can compute $\nu( \bigcup_{i=1}^n E_i )$ in terms of the $\nu(E_i)$: by de Moivre,
\[ \bigcup_{i=1}^n E_i = \left( \bigcap_{i=1}^n E_i^c \right )^c \]
and, since the complementary events $E_i^c$ are also independent,
\begin{align}\label{eqn:union-ind-formula} \nu\left(\bigcup_{i=1}^n E_i\right) &= 1 - \prod_{i=1}^n \nu(E_i^c) \\ 
 & = 1 - \prod_{i=1}^n (1-\nu(E_i)). \end{align}

\subsubsection{Categorification}
What if the $E_i$ are not independent? We can still obtain a similar formula by categorifying de Moivre's theorem: for an event $E \subset \Omega$, we consider the indicator random variable 
\[ X_{E}(\omega) = \begin{cases} 1 & \textrm { if } \omega \in E \\
							0 & \textrm{ if } \omega \not\in E. \end{cases} \]
We then have:
\begin{enumerate}
\item $\bbE[X_E]=\nu(E)$,
\item $X_{E} \cdot X_{E'} = X_{E \cap E'}$, and
\item $X_{E^c} = 1 - X_E$. 
\end{enumerate}
Using these rules, de Moivre can be written as:
\begin{equation}\label{eqn:deMoivre-categorified}
X_{\bigcup E_i}  = 1 - \prod_{i=1}^n (1- X_{E_i})
\end{equation}
If the $E_i$ are independent, then taking expectation (which commutes with multiplication for independent random variables), we find 
\begin{eqnarray*} \nu\left(\bigcup_{i=1}^n E_i\right)=\bbE\left[X_{\bigcup_{i=1}^n E_i}\right] &=& 1 - \bbE\left[ \prod_{i=1}^n (1-X_{E_i}) \right] \\&=& 1- \prod_{i=1}^n \bbE[(1-X_{E_i})]\\
&=& 1 - \prod_{i=1}^n \left( 1- \nu(E_i) \right),\end{eqnarray*}
and thus we recover (\ref{eqn:union-ind-formula}) as a specialization of (\ref{eqn:deMoivre-categorified}).   

\subsubsection{Inclusion-exclusion formula}
In general, when the $E_i$ are not necessarily independent, we can still expand (\ref{eqn:deMoivre-categorified}) and take the expectation of each term to obtain the inclusion-exclusion formula 
\begin{equation}\label{eqn:inclusion-exclusion-measure} \nu\left( \bigcup_{i=1}^n E_i \right) = \sum_{\substack{I \subset \{1, \ldots, n\},\\I \neq \emptyset}} (-1)^{|I|-1} \nu \left(  \bigcap_{i \in I} E_i \right). 
\end{equation}

\subsubsection{$m$-wise independence}
We are particularly interested in the situation where the $E_i$ are not necessarily independent, but are $m$-wise independent for some $m$ (e.g. pair-wise independent, triple-wise independent...). Then, we will have a nice product formula for the terms in the inclusion-exclusion formula (\ref{eqn:inclusion-exclusion-measure}) with $|I| \leq m$, and we can treat the rest as an error term to be estimated in comparing $\nu\left( \bigcup_{i=1}^n E_i \right)$ with $\prod_{i=1}^n \nu(E_i)$. A convenient way to keep track of the information provided by $m$-wise independence is to introduce a dummy variable $t$ and consider the polynomial
\[ 1 - \prod_{i=1}^n (1 - X_{E_i}t). \]
If we evaluate at $t=1$, we obtain $X_{\bigcup_{i=1}^n E_i}$; thus, if we take the expectation coefficient by coefficient, we obtain a polynomial
\[ 1 - \bbE\left[\prod_{i=1}^n (1 - X_{E_i}t)\right] \]
whose value at $t=1$ is $\nu\left(\bigcup_{i=1}^n E_i\right).$ On the other hand, if the $E_i$ are $m$-wise independent, then we have the following equality (where the subscript $\leq m$ denotes truncation of the polynomial after the term $t^m$):
\begin{align}\nonumber \left( 1 - \bbE\left[ \prod_{i=1}^n (1 - X_{E_i}t) \right] \right)_{\leq m} & = \left( 1 - \prod_{i=1}^n \bbE\left[1-X_{E_i} t \right] \right)_{\leq m}  \\ 
\label{eqn:k-wise-ind-poly-classical} & = \left( 1 - \prod_{i=1}^n (1-\nu(E_i) t) \right)_{\leq m} \end{align}
i.e., these polynomials have the same coefficients for $ t, t^2, \ldots, t^m$. Thus, in this case the error in the approximation 
\begin{equation} \label{eqn:independence-approx} \nu\left(\bigcup_{i=1}^n E_i\right) \approx 1 - \prod_{i=1}^n (1-\nu(E_i)) \end{equation}
is bounded by the sum of the absolute values of the higher order coefficients of 
\[ 1 - \bbE\left[\prod_{i=1}^n (1 - X_{E_i}t)\right] \textrm{ and } 1 - \prod_{i=1}^n (1-\nu(E_i) t). \]
In particular, if other information allows us to show this sum is small, then (\ref{eqn:independence-approx}) is a good approximation. 

\begin{remark} A crucial point for us is that, although we first built these polynomials using product formulas, \emph{the error in the final approximation has a simple expression in terms of the coefficients}. In the motivic setting, we will be taking the union of an infinite family of events parameterized by a variety, and will have no reasonable interpretation of what it means to form the corresponding products except via a formal description of the coefficients of the power series. However, since we will be using approximations obtained from $m$-wise independence, this is exactly what we need! 
\end{remark}

\subsection{A motivic analog}
\subsubsection{Crash course in motivic probability}
Let $\Omega$ be a variety over $S$, and assume $[\Omega]$ is invertible in $\calM_S.$ We consider the ring $K_0(\Var/\Omega)$ of \emph{motivic random variables} (cf. also \cite{howe:mrv1, howe:mrv2}) over $\Omega$, equipped with the expectation function
\[ \bbE: K_0(\Var/\Omega) \rightarrow \calM_S \]
given by forgetting the structure morphism and dividing by $[\Omega/S]$, so that, e.g., 
\[ \bbE\left[ F/\Omega\right] = \frac{[F/S]}{[\Omega/S]}. \]
The pair $(\Omega, \bbE)$ is a \emph{motivic probability space}. 

An \emph{event} on $\Omega$ is a constructible subset $A \subset \Omega$, and the corresponding indicator random variable is $[A] \in K_0(\Var/\Omega)$.

\subsubsection{Families and linearity of expectation}
A \emph{family of motivic random variables} on $\Omega$, parameterized by $J/S$, is an element of $K_0(\Var/ J \times_S \Omega)$. A \emph{family of events} on $\Omega$, parameterized by a variety $J/S$, is a constructible subset $\Phi \subset J \times_S \Omega$. The corresponding family of indicator random variables is $[\Phi] \in K_0(\Var/ J \times_S \Omega).$ 

There is a natural fiberwise expectation map
\begin{align*} \bbE_{J}: K_0(\Var/J \times_S \Omega) & \rightarrow K_0(\Var/J) \\
[F/J \times_S \Omega] & \mapsto \frac{[F/J]}{[J \times_S \Omega / J]} \in \calM_J.
\end{align*}
The map $\bbE_{J}$ should be thought of as taking a family of random variables parameterized by $J$ to the function on $J$ that sends a point to the expectation of the random variable parameterized by that point. 

In this context, we will write $\int_J$ for both of the natural forgetful maps
\[ K_0(\Var/J) \rightarrow K_0(\Var/S) \textrm{ and } K_0(\Var/J \times_S \Omega) \rightarrow K_0(\Var/\Omega), \]
which are morphisms of $K_0(\Var/S)$-modules. 
The map $\int_J$ should be viewed as summing up values of a function on $J$ or summing up random variables in a family parameterized by $J$. We have the following linearity of expectation:
\begin{equation}\label{equation.linearity-of-expectation} \bbE\left[ \int_J a \right] = \int_J \bbE_{J}[a]. \end{equation}

\subsubsection{Independence in algebraic families}

\begin{definition}
A family of motivic random variables $X$ on $\Omega$, parameterized by $J/S$, is \emph{strongly $m$-wise independent} if, for all $\mu\in \calQ$ such that $|\mu| \leq m$, 
\[ \bbE_{\Conf^\mu_{/S} (J)}\left[ \Conf^\mu_{J \times_S \Omega / \Omega}(X)\right] = \Conf^{\mu}_{J/S} \left ( \bbE_{J}[X] \right)\]
in $\calM_{\Conf^{\mu}_{/S}(J)}$. 
\end{definition}

\begin{remark}\label{remark:linearity-of-expectation-m-wise}
Invoking linearity of expectation as in (\ref{equation.linearity-of-expectation}), we find that if $X$ is strongly $m$-wise independent, then
\[ \bbE\left [ \int_{C^\mu_{/S}(J)} \Conf^\mu_{J \times_S \Omega / \Omega}(X) \right] = \int_{C^\mu_{/S}(J)} C^\mu_{J/S}( \bbE_J[X] ) \textrm{ in }\calM_S. \]
For this kind of identity, we will usually drop the integral signs indicating the application of a forgetful map between Grothendieck rings. This will not cause any ambiguity because our notation for expectation functions indicates where the argument should be taken to live. Thus, we write simply,
\[ \bbE\left [ \Conf^\mu_{J \times_S \Omega / \Omega}(X) \right] =  C^\mu_{J/S}( \bbE_J[X] ) \textrm{ in }\calM_S. \] 
\end{remark}

\begin{remark} If we base change to a single configuration, strong $m$-wise independence implies that the density of the intersection of the events attached to the points in the configuration is is equal to the product of their individual densities, as one would expect from the terminology. However, the full statement is stronger than just requiring this independence for every set of distinct points, as it implies some coherence of the punctual $m$-wise independence as we algebraically vary the points we are considering -- thus the adjective ``strongly" in the terminology strongly $m$-wise independent. 
\end{remark}

\subsubsection{Generating functions and $m$-wise independence}
Applying the following theorem to the indicator random variable $[\Phi]$ for a family of events $\Phi$ gives a motivic analog of (\ref{eqn:k-wise-ind-poly-classical}):
\begin{theorem}\label{theorem:strongly-independent-truncation}
If $X$ is a strongly $m$-wise independent family of motivic random variables on $\Omega$, parameterized by $J$, then 
\begin{equation}\label{equation.strongly-independent-truncation}  \left( \bbE\left[ \prod_{j \in J \times_S \Omega / \Omega} (1 - X_j t) \right] \right)_{\leq m} = \left( \prod_{j \in J/S}  (1 - \bbE[X_j] t) \right)_{\leq m} \end{equation}
as polynomials with coefficients in $\calM_S$. Note that here $\bbE[X_j]$, to be interpreted as a family over $J$ in the motivic Euler product, really means $(\bbE_J[X])_j$.
\end{theorem}
\begin{proof}
By Proposition \ref{proposition:fundamental-euler-product}, the coefficient of $t^k$ on the left side of (\ref{equation.strongly-independent-truncation}) is 
\[ \bbE\left[ \sum_{\substack{\mu \in \calQ \\ |\mu|=k}} (-1)^{||\mu||} \left[\Conf^\mu_{J \times_S \Omega / \Omega} (X) \right] \right] = \sum_{\substack{\mu \in \calQ \\ |\mu|=k}} (-1)^{||\mu||} \bbE\left[ \Conf^\mu_{J \times_S \Omega / \Omega} (X) \right].  \] 
Thus, for $k \leq m$, invoking strong $m$-wise independence and linearity of expectation as in Remark \ref{remark:linearity-of-expectation-m-wise} allows us to simplify the coefficient of $t^k$ on the left side of (\ref{equation.strongly-independent-truncation}) to 
\begin{equation}\label{equation.strongly-ind-trunc-proof-intermediate}\sum_{\substack{\mu \in \calQ\\ |\mu|=k}} (-1)^{||\mu||}\Conf^{\mu}_{J/S}(\bbE_J([X])). \end{equation}
Invoking Proposition \ref{proposition:fundamental-euler-product} again shows that (\ref{equation.strongly-ind-trunc-proof-intermediate}) is equal to the coefficient of~$t^k$ on the right-hand side of (\ref{equation.strongly-independent-truncation}), and we conclude.  
\end{proof}
\begin{remark}
Another notion of independence and its connection with moment generating functions is used in \cite{howe:mrv1,howe:mrv2}. We will give a more complete study of independence and moment generating functions for motivic random variables in \cite{bilu-howe:motivic-random-variables}, encompassing the results presented here and in \cite{howe:mrv1,howe:mrv2}.
\end{remark}

\section{Motivic probability problems from coherent sheaves} \label{section.motivic-prob-coherent-sheaves}
In this section we apply the formalism of the previous section to answer some natural motivic probability problems on global sections of coherent sheaves. 

\subsection{The problems}

Let $f: X \rightarrow S$ be a map of varieties, and let $\calB$ be a coherent sheaf on $X$. Let $\calA$ be a coherent sheaf on $S$ equipped with a map $\phi: \calA \to f_* \calB$. 

We can use the map $\phi$ to cut out a space of sections of $\calA$ by imposing local conditions along the fibers of $\calB$. Concretely, let $T$ be a constructible subset of $\bbV(\calB)$. Then, for each geometric point $x: \Spec L \rightarrow X$, the fiber $T_x$ is a constructible subset of the affine space 
\[ \bbV(\calB)|_x = \bbA(\calB|_x) \cong \bbA_{L}^{\dim \calB|_x}. \]
We view this set $T_x$ as the set of conditions on $\bbV(\calA)$. In particular, given a geometric point $s$ of $S$ and a geometric point in $X$ lying above $s$, we say that $a \in \calA|_s$ \emph{satisfies~$T$ at $x$} if the image of $a$ under the natural map
\begin{equation}\label{eq:mapphix}\phi_x: \calA|_s \xrightarrow{\phi} \Gamma(X_s, \calB|_{X_s}) \rightarrow \calB|_{x} \end{equation}
is in $T_x$. We say that $a$:
\begin{enumerate}
\item \emph{satisfies $T$ everywhere} if it satisfies $T$ at all such~$x$.
\item \emph{satisfies $T$ somewhere} if it satisfies $T$ at some $x$.
\item \emph{satisfies $T$ nowhere} or \emph{avoids $T$} if it does not satisfy $T$ at any $x$.
\end{enumerate}

\begin{example}
If $S=\Spec K$, $X=\bbP^1_K$, $\calB=O(d)$, and $\calA=\Gamma(\bbP^1_K, O(d))$, the space of degree $d$ homogeneous polynomials in two variables, then we can consider the condition $T \subset \bbV(\calO(d))$ such that 
\begin{enumerate}
\item $T_{\infty}=\{0\}$,
\item $T_{0}=\bbA(\calO(d)|_0) \bs \{0 \}$, 
\item and $T_{x}=\bbA(\calO(d)|_x)$ for all $x \neq 0, \infty$. 
\end{enumerate}
Then, the set of sections satisfying $T$ everywhere is the set of degree $d$ homogeneous polynomials in two variables which vanish at $\infty$ and do not vanish at $0$.  
\end{example}

\newcommand{\nw}{\mathrm{nowhere}}
\newcommand{\ew}{\mathrm{everywhere}}
\newcommand{\sw}{\mathrm{somewhere}}

We may consider the set $\bbV(\calA)^{T-\sw}$ of $a \in \bbV(\calA)$ such that the corresponding section satisfies $T$ somewhere. Similarly, we may consider analogous sets for nowhere and everywhere, but these are all related in the obvious way by complements (denoted with a superscript $c$):
\[ \bbV(\calA)^{T-\ew} = \bbV(\calA)^{T^c-\nw} = \left(\bbV(\calA)^{T^c-\sw}\right)^c. \]
Thus, in the following discussion we will focus on the ``somewhere" condition. 

The set $\bbV(\calA)^{T-\sw}$ is constructible: indeed, if we consider the geometric realization of the map $\phi_X: f^*\calA \rightarrow \calB$ adjoint to $\phi$,  
\[ \phi_X: \bbV(\calA) \times_S X \rightarrow \bbV(\calB) \]
and we write $\pi_\calA$ for the projection $\bbV(\calA) \times_S X \rightarrow \bbV(\calA)$, then
\[ \bbV(\calA)^{T-\sw} = \pi_\calA( \phi_X^{-1}(T) ) \]
By Chevalley's theorem, we conclude that $\bbV(\calA)^{T-\sw}$ is constructible. 
 
Thus, $\bbV(\calA)^{T-\sw}$ is an event on the motivic probability space $\bbV(\calA)$. 
 
\begin{question}\label{question:coherent-sheaf-motivic-problem} What is the motivic probability 
\[ \bbE[[\bbV(\calA)^{T-\sw}]] = \frac{[\bbV(\calA)^{T-\sw}]}{[\bbV(\calA)]} \] 
in $\calM_S$?
\end{question}

\subsection{An approximate answer}

To answer this question, one should think of $\bbV(\calA)^{T-\sw}$ as being the union of the family $\phi_X^{-1}(T)$ of events ``satisfies $T$ at $x$'' over $x$ in $X$. If we fix a geometric point $x$, then, as long as the map $\phi_x$ in (\ref{eq:mapphix}) is surjective, the motivic density of sections satisfying $T$ at $x$ will be 
\begin{equation}\label{eq:expectedmotivicdensity} \frac{[T_x]}{[\bbV(\calB)_x]}=\frac{[T_x]}{\bbL^{\dim \calB|_x}}. \end{equation}
Moreover, if $\phi_{x_1}, \ldots, \phi_{x_k}$ are jointly surjective, then these events will behave as if they are independent: the motivic density of sections satisfying $T$ at each of the distinct points $x_1, x_2, \ldots, x_k$ will be the product of the densities for each point. If such joint surjectivity holds for any choice of $x_1, \ldots, x_k$ for $k \leq m$, we say that $\phi$ is $m$-generating.  We will see that when $\phi$ is $m$-generating, the family of events is $m$-wise independent, and we can try to carry out the strategy outlined in the classical case for computing the density of the union. More precisely, under a natural condition on~$T$ that allows us to control the error terms, we will show:

\begin{theorem}\label{theorem:main-theorem-coherent} If $T$ is $M$-admissible (to be defined below) and $\phi$ is $m$-generating, then, in $\widehat{\calM}_S$, 
\[ 1 - \frac{[\bbV(\calA)^{T-\sw}]}{[\bbV(\calA)]} = \left.\prod_{x \in X/S}\left( 1-\frac{[T]_x}{[\bbV(\calB)]_x} t \right) \right|_{t=1} \mod \Fil_{-m + M} . \]
\end{theorem} 

The rest of this section is devoted to the proof of this theorem. In \ref{subsection:mgeneration} we introduce the notion of $m$-generating map, and prove in Lemma \ref{lemma:coherent-strongly-m-wise-independent} that if $\phi$ is $m$-generating, then the family of events $\phi_X^{-1}(T)$ is strongly $m$-wise independent with the expected motivic density (\ref{eq:expectedmotivicdensity}) relatively to $X$. The notion of admissibility and its consequences are studied in \ref{subsect:admissibility}. In particular, Lemma  \ref{lemma:M-admissible-euler-product} shows that  the motivic Euler product in the statement of theorem \ref{theorem:main-theorem-coherent} converges, and is equal modulo $\Fil_{-m+M}$ to its truncation up to degree $m$. We apply the approximate version of  motivic inclusion-exclusion in $(\ref{eqn:inc-exc-app})$ to estimate $1 - \frac{[\bbV(\calA)^{T-\sw}]}{[\bbV(\calA)]}$. The error terms are bounded appropriately in Lemma \ref{lemma:admissible-inc-exc-error-terms} thanks to the admissibility condition. For the other terms, we use strong $m$-wise independence in the form of Theorem \ref{theorem:strongly-independent-truncation}, together with the value of the motivic density, to rewrite them as the aforementioned truncation of the Euler product in the right-hand side of the statement of Theorem \ref{theorem:main-theorem-coherent}. 
\subsubsection{$m$-generation and independence} \label{subsection:mgeneration}

Recall that given a geometric point $x: \Spec L \rightarrow X$ with image $s \in S(L)$, we wrote $\phi_x$ for the composition
\[ \calA|_s \xrightarrow{\phi} (f_*\calB)|_s \rightarrow \Gamma(X_s, \calB|_{X_s}) \rightarrow \calB|_x. \] 

\begin{definition}
The map $\phi$ is $m$-generating if, for any algebraically closed $L/K$, any $s \in S(L)$, and any distinct $x_1, \ldots, x_k \in X_s(L)$, $k \leq m$, the map
\[ \calA|_s \xrightarrow{ \left( \phi_{x_1}, \ldots, \phi_{x_k} \right) } \bigoplus_{i=1}^k \calB_{x_i} \]
is surjective.  

Equivalently, $\phi$ is $m$-generating if for any$\mu\in \calQ$ such that  $|\mu| \leq m$, the map
\[ \phi_{\mu}: f_\mu^* \calA \rightarrow \Conf_{X/S}^\mu(\calB) \]
is a surjective map of coherent sheaves on $\Conf^{\mu}_{/S}(X)$. 
\end{definition}

\begin{lemma}\label{lemma:coherent-strongly-m-wise-independent}
If $\phi$ is $m$-generating for $m \geq 1$, then, for any constructible subset $T \subset \bbV(\calB)$, viewing $\phi_X^{-1}(T)$ as a family of events on $\bbV(\calA)$ parameterized by~$X$, we have
\[ \bbE_X \left[ \phi_X^{-1}(T) \right] = \frac{[T/X]}{[\bbV(\calB)/X]} \in \calM_X \]
and $\phi_X^{-1}(T)$ is strongly $m$-wise independent. 
\end{lemma}

\begin{proof} Let $\mu\in \calQ$ be such that $|\mu|\leq m$. We note that because $\phi_X^{-1}(T)$ is a constructible subset of $X \times_S \bbV(\calA)$, we have that
\[ \Conf^\mu_{X \times_S \bbV(\calA)/\bbV(\calA)} (\phi_X^{-1}(T)) = \Conf^{\mu}_{/\bbV(\calA)} (\phi_X^{-1}(T)). \]
We now give another characterization of the constructible subset $\Conf^\mu_{/ \bbV(\calA)}(\phi_X^{-1}(T) )$ inside $\Conf^{\mu}_{/S}(X) \times_S \bbV(\calA)$. Taking the geometric realization of $\phi_\mu$ and invoking Lemma \ref{lemma:compatibility-conf-realization}, we obtain a map
\[ \phi_\mu: \Conf^{\mu}_{/S}(X) \times_S \bbV(\calA) \rightarrow \bbV(\Conf^{\mu}_{X/S}(\calB)) = \Conf^\mu_{X/S}( \bbV(\calB) ) \]
in $\left(\Var/C^{\mu}_{/S}(X)\right)_{C^{\mu}_{/S}(X)-\PW}$. 
Then, comparing geometric points and using Lemma \ref{lemma:rec-princ-cons}, we find
\[ \Conf^\mu_{/ \bbV(\calA)}(\phi_X^{-1}(T)  ) = \phi_\mu^{-1} \left( \Conf^{\mu}_{X/S}(T) \right) \]
as constructible subsets of $\Conf^{\mu}_{/S}(X) \times_S \bbV(\calA) $.

By Corollary \ref{corollary:presentation-of-surjection}, 
\begin{align*} \left[\phi_\mu^{-1} \left( \Conf^{\mu}_{X/S}(T) \right)\right] & = \frac{\left[\Conf^{\mu}_{/S}(X) \times_S \bbV(\calA)\right]}{\left[\Conf^\mu_{X/S}( \bbV(\calB) )\right]} \left[ \Conf^{\mu}_{X/S}(T) \right] \\
&= \left[\Conf^{\mu}_{/S}(X) \times_S \bbV(\calA)\right] \cdot \Conf^\mu_{X/S} \left( \frac{[T/X]}{[\bbV(\calB)/X]} \right)
\end{align*}
in $\calM_{\Conf^{\mu}_{/S}(X)}$. Dividing, we obtain
\begin{align*}
\frac{
\left[\Conf^\mu_{/ \bbV(\calA)}(\phi_X^{-1}(T)  )\right] } {[ \Conf^{\mu}_{/S}(X) \times_S \bbV(\calA) ] } = \Conf^\mu_{X/S} \left( \frac{[T/X]}{[\bbV(\calB)/X]} \right)
\end{align*}
in $\calM_{\Conf^{\mu}_{/S}(X)}$. For $\mu = 1^1$, this gives us the first statement of the lemma. Now, for general $\mu$ such that $|\mu|\leq m$, the left-hand side is 
\[ \bbE_{\Conf^{\mu}_{/S}(X)}\left[ \Conf^{\mu}_{/\bbV(\calA)}(\phi_X^{-1}(T)) \right] \]
and using the expression we obtained for $\bbE_X[\phi^{-1}_X(T)]$, the right-hand side is 
\[ \Conf^{\mu}_{X/S} \left( \bbE_X [ \phi_X^{-1}(T) ] \right),\]
thus we conclude. 
\end{proof}

\subsubsection{Admissibility} \label{subsect:admissibility}
\newcommand{\codim}{\mathrm{codim}}

We now introduce the notion of an admissible condition, which will allow us to control the error terms in Theorem \ref{theorem:main-theorem-coherent}. 

Qualitatively, a condition $T$ is admissible if the dimension of $T_x$ is large only along a low-dimensional subvariety of $X$. To turn this into a useful definition, we need a precise relationship between these dimensions: when $S$ is a point, we will say that $T$ is admissible if, after removing a finite subvariety, $X$ can be decomposed into locally closed subvarieties $X_i$ such that for each $i$ and each $x \in X_i$ the codimension of $T_x$ in $\bbV(\calB)_x$ is at least one more than the dimension of $X_{i}$ at $x$. 

For general $S$, we will extend this definition in the natural way. For applications, it will also be important to remember the size of the finite set we remove; thus, in full generality we obtain:

\begin{definition} \label{def:m-admissible}
A condition $T \subset \bbV(\calB)$ is $M$-\emph{admissible} if
there exists a constructible $A \subset X$ such that $A/S$ is quasi-finite of degree bounded by $M$, and a piecewise isomorphism $X' \rightarrow X\bs A$ such that for every $x \in X'$,
\[ \dim T_x \leq \dim \calB|_x - (\dim_x X'_{f(x)} + 1). \]
\end{definition}

\begin{example}\label{example:admissibility-codimension} In this example we explain how this definition of $M$-admissible reduces in a special case to the simpler definition given in the context of Theorem \ref{theorem.taylorconditions} (in Definition \ref{def:intro-m-admissible} of the introduction). Suppose $K$ is a field, $X/K$ is equidimensional of dimension $n$, and $\calB$ is locally free of rank $r$; then, $\bbV(\calB)$ is equidimensional of dimension $n+r$. We claim that $T \subset \bbV(\calB)$ is $M$-admissible in the sense of Definition \ref{def:m-admissible} if and only if there exists a finite subscheme $A \subset X$ of degree $\leq M$ over $K$ such that $T|_{X\bs A}$ has codimension at least $n+1$ in $\bbV(\calB)|_{X\bs A}$.

By replacing $X$ with $X\bs A$ we may assume $M=0$ (going either direction). If $T$ is 0-admissible in the sense of Definition \ref{def:m-admissible}, then we obtain a piecewise isomorphism $X'\rightarrow X$ such that for every $x \in X'$, 
\[ \dim T_x \leq r - (\dim_x X'_{f(x)} + 1). \]
We may write $X'$ as a finite disjoint union of irreducible locally closed subvarieties $X_i \subset X$ -- then it suffices to show that $T|_{X_i}$ has codimension at least $n+1$ in $\bbV(\calB)|_{X\bs A}$. We still have that for all $x \in X_i$. 
\[ \dim T_x \leq r - (\dim_x {X_i}_{f(x)}+1 ), \]
thus we conclude that $T_{X_i}$ has dimension $\leq r - 1$, and thus codimension $n+1$ in $\bbV(\calB)|_{X\bs A}$. 

To go the other direction, we can stratify $X$ according to the dimension of fibers of $T$, and this stratification will witness admissibility in the sense of Definition \ref{def:m-admissible} (otherwise there will be an entire stratum where the condition on $\dim T_x$ fails, and this will contradict the hypothesis on the codimension). 
\end{example}

\begin{lemma}\label{lemma:M-admissible-euler-product}
If $T$ is $M$-admissible, then 
\[ \prod_{x \in X/S} \left( 1 - \frac{[T]_x}{[\bbV(\calB)]_x} \right) \]
converges at $t=1$ in $\widehat{\calM}_S$. More precisely, for each $k$ the coefficient of $t^k$ is in $\Fil_{-k + M}$, so in particular, for any $m>0$,  
\[ \left.\left(\prod_{x \in X/S} \left( 1 - \frac{[T]_x}{[\bbV(\calB)]_x} t \right) 
\right)\right|_{t=1}  \equiv \left. \left( \prod_{x \in X/S} \left( 1 - \frac{[T]_x}{[\bbV(\calB)]} \right) \right)_{\leq m} \right|_{t=1} \mod \Fil_{-m + M} \]  
\end{lemma}
\begin{proof}

We first show the result when $M=0$.  By Corollary \ref{corollary:fundamental-euler-product-coh-sheaf}, the coefficient of~$t^k$ is a sum of terms of the form 
\[ \frac{ [\Conf^\mu_{X/S} ( T )] }{ [\bbV( \Conf^\mu_{X/S}(\calB) )] } = [\Conf^\mu_{X/S} ( T )] \cdot \bbL^{- \dim \Conf^\mu_{X/S}(\calB)} \]
over all $\mu \in \calQ$ such that $|\mu|=k$. We examine each of these terms:

Above a geometric point corresponding to the configuration $c=(x_1, \ldots, x_k)$ in~$X_s$, we find, by admissibility, 
\begin{align*} \dim \Conf^{\mu}_{X/S}(T)_c &= \sum_{i=1}^k \dim T_{x_i}\\
& \leq \sum_{i=1}^k\left( \dim \calB|_{x_i}  - (\dim_{x_i} X'_{s} + 1)\right)\\ & \leq \dim \Conf^\mu_{X/S}(\calB)|_c - \dim_{c} \left(C^{\mu}_{/S}(X)_s\right) - k. \end{align*}

Thus, by invoking Lemma \ref{lemma:dimension-fibration} with 
\[ g(c)=\dim \Conf^\mu_{X/S}(\calB)|_c - k, \] we find this term is in $\Fil_{-k}$, as desired.   

In general, we may take $A$ and $X'$ witnessing the $M$-admissibility of $T$, then factor the product as
\[ \left(\prod_{x \in A/S} \ldots \right) \cdot \left(\prod_{x \in X'/S}\ldots\right). \] 
$T|_{X'}$ is 0-admissible, so the previous result applies to the right-hand product. On the other hand, for the product over $A$, we find all the coefficients are in $\Fil^0$: indeed, 
\[ \dim \Conf^{\mu}_{A/S}(T|_A/A)_c \leq \dim \Conf^\mu_{X/S}(\calB)|_c,  \]
and since $\Conf^{\mu}_{A/S}$ is of relative dimension zero over $S$, we can again apply Lemma \ref{lemma:dimension-fibration} to conclude that the corresponding term is in in $\Fil_0$. 
Because the fiber size is bounded by $M$, only coefficients up to $t^M$ are non-zero. We conclude that in the full product, the coefficient of $t^k$ is congruent to zero mod $\Fil_{-k + M}$. In particular, the radius of convergence is at most $-1$, and the series can be evaluated at $t = 1$. 

\end{proof}

\subsubsection{Estimating the error terms}
We can now prove Theorem \ref{theorem:main-theorem-coherent}. First, we observe that by approximate motivic inclusion-exclusion (Theorem \ref{theorem:mot-inc-exc-approx}),

\begin{align} \label{eqn:inc-exc-app} 1 - [\bbV(\calA)^{T-\sw}] = & \sum_{\substack{\mu \in \calQ \\ |\mu| < m}} (-1)^{||\mu||} \left[\Conf^\mu_{/\bbV(\calA)}( \phi_X^{-1}(T))\right] \\
\nonumber & - \sum_{\substack{\mu \in \calQ \\ |\mu| < m}} (-1)^{||\mu||} \left[\Conf^{\mu,\geq m-|\mu|}_{/\bbV(\calA)}(\phi_X^{-1}(T))\right].
\end{align} 

We will eliminate the error terms appearing on the second line of (\ref{eqn:inc-exc-app}) using
\begin{lemma}\label{lemma:admissible-inc-exc-error-terms} For $|\mu| < m$, 
\[ \frac{\left[\Conf^{\mu,\geq m-|\mu|}_{/\bbV(\calA)}(\phi_X^{-1}(T))\right]}{[\bbV(\calA)]} \in \Fil_{-m + M}. \]
\end{lemma}

We finish the proof of Theorem \ref{theorem:main-theorem-coherent} before returning to prove Lemma \ref{lemma:admissible-inc-exc-error-terms}: Invoking Lemma \ref{lemma:admissible-inc-exc-error-terms} to kill the error terms on the right hand side of (\ref{eqn:inc-exc-app}), we find
\begin{multline*}
\frac{1 - [\bbV(\calA)^{T-\sw}]}{\bbV(\calA)} =\\
\left. \left(\bbE\left[ \prod_{x \in X \times_S \bbV(\calA)/\bbV(\calA)} \left(1 - [\phi^{-1}_X(T)]_x\right) \right] \right)_{\leq m} \right|_{t=1} \mod \Fil_{-m + M}.\end{multline*}
By Lemma \ref{lemma:coherent-strongly-m-wise-independent} and Theorem \ref{theorem:strongly-independent-truncation}, we can rewrite this as
\[ \frac{1 - [\bbV(\calA)^{T-\sw}]}{\bbV(\calA)} = \left.  \prod_{x \in X} \left(1-\frac{[T]_x}{[\bbV(\calB)]_x} t \right)_{\leq m} \right|_{t=1} \mod \Fil_{-m + M}. \]
And, by Lemma \ref{lemma:M-admissible-euler-product}, we obtain
\[ \frac{1 - [\bbV(\calA)^{T-\sw}]}{\bbV(\calA)} = \left. \prod_{x \in X} \left(1-\frac{[T]_x}{[\bbV(\calB)]_x} t \right) \right|_{t=1} \mod \Fil_{-m + M},\]
proving Theorem \ref{theorem:main-theorem-coherent}. 

\begin{proof}[Proof of Lemma \ref{lemma:admissible-inc-exc-error-terms}]
Arguing as in the proof of Lemma \ref{lemma:coherent-strongly-m-wise-independent}, we find that for any $|\mu| \leq m$, we have, in $\left(\Var/C^{\mu}_{/S}(X)\right)_{\PW},$
\begin{align*} \Conf^{\mu}_{/\bbV(\calA)}(\phi_X^{-1}(T)) & = \phi_\mu^{-1} \left( \Conf^{\mu}_{X/S}(T) \right)\\
&=\bbA^{\dim f_\mu^* \calA - \dim \Conf^\mu_{X/S} \calB} \times_{\Conf^\mu_{/S}(X)} \Conf^\mu_{X/S}(T) 
\end{align*}

We assume first that $T$ is 0-admissible, and replace $X$ with a piecewise decomposition $X'$ witnessing the admissibility. Computing above $c \in \Conf^{\mu}_{/S}(X)_s$ as in the proof of Lemma \ref{lemma:M-admissible-euler-product}, we find
\[ \dim \Conf^{\mu}_{/\bbV(\calA)}(\phi_X^{-1}(T))_c \leq \dim f_\mu^* \calA - \dim_{c} \Conf^\mu_{/S}(X)_s - |\mu|. \]
 Decomposing $S$ we may assume $\dim f_\mu^* \calA = f_\mu^*\dim \calA=r$ is constant. Using the method of the proof of Lemma \ref{lemma:dimension-fibration}, we can thus decompose $\Conf^{\mu}_{/\bbV(\calA)}(\phi_X^{-1}(T))$ into irreducible subvarieties such that each subvariety has dimension $\leq r - |\mu|$ over $S$.  

In particular, taking $|\mu|<m$ we may apply this computation to $\mu \cdot *^{m-|\mu|}$. Thus, $\Conf^{\mu, \geq m-\mu}_{/\bbV(\calA)}(\phi_X^{-1}(T))$  is the image of a disjoint union of irreducible varieties of dimension $\leq r - m$ over $S$, thus  
\[ \dim_{/S} \Conf^{\mu, \geq m-\mu}_{/\bbV(\calA)}(\phi_X^{-1}(T)) \leq r - m.\]
Dividing by $[\bbV(\calA)]$ gives an element of $\Fil_{-m}$, completing the proof in the 0-admissible case. 

When $T$ is $M$-admissible for $M > 0$, we argue similarly, first decomposing configuration spaces into products according to points in $A$ and $X'$ witnessing the admissibility.

\end{proof}

\subsection{Taylor conditions}

As a special case of Theorem \ref{theorem:main-theorem-coherent}, we obtain:

\begin{theorem}\label{theorem.general}
Let $f:X \rightarrow S$ be a map of varieties, let $\calF$ be coherent sheaf on~$S$, and let $\calG \rightarrow f_*\calF$ be an $r$-infinitesimally $m$-generating set of global sections. If $T$ is a constructible subset of $\calP^r_{X/S}(\calF)$ such that $T^c$ is $M$-admissible, viewed as a local condition on $\calG$ via $\Taylor^r: f^*\calG \rightarrow \calP^r_{X/S}(\calF)$, then 
\[ \frac{ [\bbV(\calG)^{T-\ew} ]}{[\bbV(\calG)]} = \left. \prod_{x \in X/S} \left(1 - \frac{[T^c]_x}{[\bbV(\calP^r_{X/S}(\calF))]_x} t \right)\right|_{t=1} \mod \Fil_{-m + M} \]
\end{theorem}
\begin{proof}
We apply Theorem \ref{theorem:main-theorem-coherent} with $\calA=\calG$, $\calB=\calP^r_{X/S}(\calF)$, and $\phi = \Taylor^r$; the map $\phi$ is $m$-generating by definition of $r$-infinitesimally $m$-generating. We obtain
\begin{align*} \frac{ [\bbV(\calG)^{T-\ew} ]}{[\bbV(\calG)]} & = 1 - \frac{[\bbV(\calG)^{T^c-\sw}]}{[\bbV(\calG)]} \\
&= \left.\prod_{x \in X/S} \left(1 - \frac{[T^c]_x}{[\bbV(\calP^r_{X/S}(\calF))]_x} t \right)\right|_{t=1} \mod \Fil_{-m + M}. \end{align*}

\end{proof}

We obtain Theorem \ref{maintheorem.general} from Theorem \ref{theorem.general} by using Example \ref{example:inf-m-gen-coh-sheaf} to control how large $d$ needs to be for $f_*\calF(d)$ to be $r$-infinitesimally $m$-generating. Theorem~\ref{theorem.taylorconditions} follows from Theorem \ref{maintheorem.general} after taking into account Example \ref{example:admissibility-codimension} to compare the two notions of admissibility and using Example \ref{example:surj-projective-space} to obtain a precise value of~$\epsilon$.

\section{Applications}\label{sect.applications}

In this section we present some applications of Theorem \ref{maintheorem.general} (and Theorem \ref{theorem.general}) to concrete problems in motivic statistics. Our goal is not to give an exhaustive list, but rather to discuss some interesting examples that illustrate how the theorem can be used in practice. Many of our examples are motivic analogs of results in arithmetic statistics that can be found in the literature, and it seems likely that for any natural\footnote{of course one can cook up examples where Poonen's general theorem applies but the Taylor conditions are not constructible, but we claim these are not natural!} statement in arithmetic statistics that can be proven via Poonen's sieving method, our Theorem~\ref{maintheorem.general} (or Theorem \ref{theorem.general}) can be used to give a motivic analog. 

\subsection{Hypersurfaces transverse to a quasi-projective variety}

The following theorem computes the asymptotic motivic density of hypersurfaces transverse to a fixed quasi-projective variety with isolated singularities. It is a straightforward application of Theorem \ref{theorem.taylorconditions} and the Euler product for the Kapranov zeta function (as in Example \ref{example:kapranov-zeta}). 

\begin{theorem}
Let ${X \subset \bbP^n_\bbC}$ be an irreducible locally closed subvariety with isolated singularities. Let $U_d$ be the constructible subset of 
\[ V_d:=H^0(\bbP^n_{\bbC}, \calO(d)) \]
consisting of sections that intersect $X$ transversely. Then, 
\[ \lim_{d \rightarrow \infty} \frac{[U_d]}{[V_d]} = \left(Z^{\Kap}_{X_\mathrm{smooth}}\left(\LL^{-(\dim X+1)}\right) \right)^{-1} \cdot (1-\bbL^{-1})^{\# X_\mathrm{sing}},\]
where we have written $X_{\mathrm{sing}}$ for the singular locus of $X$, a finite collection of points, and $X_{\mathrm{smooth}}=X-X_{\mathrm{sing}}$ for the smooth locus. 
\end{theorem}

\subsection{Smooth complete intersections in $\bbP^n$}
Let $K$ be a field. The following theorem gives the motivic analog of a point-counting result of Bucur and Kedlaya \cite{bucur-kedlaya:complete-intersections} on smooth intersections in $\bbP^n_K$. 

\begin{theorem}\label{theorem.complete-intersection}
Fix $1\leq k \leq n$ and, for $\underline{d}=(d_1, \ldots, d_k)$, let $U_{\underline{d}}$ be the constructible subset of $V_{\underline{d}}=H^0(\bbP^n_K, \calO(d_1) \oplus \ldots \oplus \calO(d_k))$ consisting of $k$-tuples of hypersurfaces intersecting transversely (i.e. cutting out a smooth variety of dimension $n-k$). Define 
$$L(n,k): = \prod_{j=0}^{k-1}(1 - \LL^{-n + j})\in K_0(\Var/K).$$
Then, 
$$ \lim_{\min(d_i)\rightarrow \infty} \frac{[U_{\underline{d}}]}{[V_{\underline{d}}]} =\\ \left.\prod_{x\in \bbP^n}\left(1 - \LL^{-k}(1-L(n,k)) \cdot t\right)\right|_{t=1} \textrm{ in } \widehat{\widetilde{\calM}}_K$$ 
\end{theorem}

\begin{remark} Using numerical computations, Bucur and Kedlaya have observed that the corresponding arithmetic Euler product does not appear to be a rational number when $n \geq k \geq 2$; if this holds, then the Euler product cannot have a finite formula in terms of special values of zeta functions. Nevertheless, one could have used the power structure \cite{glm:power-structure} on the Grothendieck ring of varieties to express the motivic probability in Theorem \ref{theorem.complete-intersection}; however, just as in the proof of the main theorem on motivic random variables in \cite{howe:mrv2}, the power structure alone does not seem to be a sufficiently versatile tool to prove Theorem \ref{theorem.complete-intersection} without the intervention of motivic Euler products in computations. 
\end{remark} 

\begin{proof} For any $\underline{d}$, we consider the Taylor condition $$T\subset \calP^1(\calO(d_1) \oplus \ldots \oplus \calO(d_k)) = \calP^1(\calO(d_1)) \oplus \ldots \oplus \calP^1(\calO(d_k))$$  given by asking that for a $k$-tuple of sections $(f_1,\ldots,f_k)$, at any point, either $f_1,\ldots,f_k$ do not all vanish, or they all vanish and their gradients are linearly independent. Then, for any point $x$, $T^c_x$ corresponds to those $k$-tuples where $f_1,\ldots,f_k$ all vanish at $x$, and with linearly dependent gradients. 

We first check that $T^c$ is 0-admissible. The set of $k$-tuples of linearly dependent vectors in a vector space of dimension $n$ is of dimension $(n+1)(k-1)$. Thus, for every $x$, the set $T^c_x$ is of codimension $(n+1)k - (n+1)(k-1) = n+1$ in the corresponding fiber of the bundle of principal parts. 

We compute
\begin{multline*}\frac{[T^c]_x}{[\calP^1(\calO(d_1) \oplus \ldots \oplus \calO(d_k))]_x} = \LL^{-k}\times \frac{\LL^{nk}-(\LL^{n}-1)(\LL^{n} - \LL) \ldots (\LL^n - \LL^{k-1})}{\LL^{nk}}\\
= \LL^{-k} (1-L(n, k)).
\end{multline*}

By example \ref{example:surj-projective-space},  for any $m$, $\Gamma(\bbP^n_K, \calO(d_i))$ is 1-infinitesimally $m$-generating for $d_i$ sufficiently large. We may deduce from this that  $\Gamma(\bbP^n_K, \oplus_{i=1}^k\calO(d_i))$ is 1-infinitesimally $m$-generating when $\min(d_i)$ is sufficiently large, and we may apply Theorem \ref{theorem.general} to conclude. 
\end{proof}

\subsection{Hypersurfaces with exactly $m$ singular points}
Let $Y$ be a quasi-pro\-jec\-tive variety over $K$. Following \cite{vakil-wood:discriminants}, for every integer $n\geq 0$, denote by $\Sym^n_{[m]}(Y)$ the locally closed subset of $\Sym^{n}(Y)$ consisting of unordered $n$-tuples of points supported on exactly $m$ distinct geometric points. Denote by $$Z_{Y}^{\Kap, [m]} (t) = \sum_{n\geq 0} \Sym^{n}_{[m]}(Y)t^n\in K_0(\Var_{K})[[t]]$$
the corresponding generating function. 
In \cite{vakil-wood:discriminants}, Vakil and Wood prove the following theorem (modulo the correction of working in the modified Grothendieck ring):
\begin{theorem}[Vakil-Wood]\label{theorem.VWsingularpoints} Let $Y$ be a smooth projective variety of pure dimension $n>0$, with an ample line bundle $L$. Denote by $U_d^{[m]}$ the subset of $V_d = H^0(Y, L^d)$ of sections singular at exactly $m$ distinct geometric points. Then
$$\lim_{d\to \infty} \frac{\left[U_d^{[m]}\right]}{[V_d]} = \frac{Z_{Y}^{\Kap, [m]}(\LL^{-(n+1)})}{Z_{Y}^{\Kap}(\LL^{-(n+1)})} \textrm{ in } \widehat{\widetilde{\M}}_K$$
\end{theorem}
The existence of this result moreover prompted them to conjecture the arithmetic analog of the latter, which was later proved by Gunther \cite{gunther:random-hypersurfaces}. While the proof of Vakil and Wood's result in \cite{vakil-wood:discriminants} is quite lengthy and involves some intricate combinatorics, Gunther's proof of the arithmetic analog is surprisingly simple: for every possible configuration of $m$ singular geometric points, he sums the corresponding limit in Poonen's theorem with Taylor conditions, where the imposed Taylor condition is to be singular at the $m$ chosen points and smooth elsewhere. Since we are working over finite fields, the resulting sum is finite and gives the desired limit.

In the motivic setting, there is no longer a finite set of possible configurations, so we cannot simply add up the Euler products. Nonetheless, there is a natural way to integrate over all possible configurations: we may take the relative Euler product, which gives a power series with coefficients in the relative Grothendieck ring $K_0(\Var / \Conf^m(Y))$; to integrate, we simply forget the map to $\Conf^m(Y)$, to get a power series with coefficients in $K_0(\Var/K)$. 

In this section, using this idea and theorem \ref{theorem.general}, we will give a proof of theorem \ref{theorem.VWsingularpoints} closer to the spirit of Gunther's proof. 
 We start by giving a formula for $Z_{Y}^{\Kap, [m]} (t)$ in terms of motivic Euler products. Given a variety $S$ and a power series ${f \in  K_0(\Var / S)[[t]]}$, we write $\int_S f$ for the power series in $K_0(\Var / K)[[t]]$ given by forgetting the structure map to $S$ for each coefficient. We write 
\[ \bc_m = \{(y, \bc)\in Y \times \Conf^m Y,\ y\in |\bc|\}  \]
for the universal configuration of $m$-points over $\Conf^m Y$. 
\begin{example} We have
\begin{equation}\label{equation.Kapranovzetam}Z_{Y}^{\Kap, [m]} (t) = \int_{\Conf^m Y} \prod_{y\in \bc_m/\Conf^{m} Y} \frac{t}{1-t},\end{equation}
in $K_0(\Var/K)$, since
$$\prod_{y\in \bc_m/\Conf^m Y}(t + t^2 + \ldots ) = \sum_{\substack{\mu = (m_i)_{i\geq 1}\\ \sum m_i = m}} \Conf^{\mu}_{/\Conf^m Y} (\bc_m) t^{\sum \mu}$$
by (\ref{equation:eulerproductconstantcoef}), so that, for every given configuration $\bc$ of $m$ geometric points of $Y$ and for every integer $n$, the right-hand side of (\ref{equation.Kapranovzetam}) sums over all effective zero-cycles of degree $n$ on $Y$ supported on exactly these $m$ points.
\end{example}

We now turn to our proof of theorem \ref{theorem.VWsingularpoints}.
Put $S = \Conf^{m} Y$ and  $X = Y\times \Conf^{m} Y$, viewed as a variety over $S$ via the second projection. We denote by $\tilde{L}$ the pullback of the line bundle $L$ to~$X$. 
In terms of bundles of principal parts, we have $\calP^1_{/S}(\tilde{L}^d) = \mathrm{pr}_1^* \calP^1(L^d)$. We define the Taylor condition $\mathcal{T} \subset \calP^1_{/S}(\tilde{L}^d)$ by :
$$\mathcal{T}_{|\bc_m} = \text{zero section of}\  \calP^1_{/S}(\tilde{L}^d)_{|\bc_m},$$
and $$\mathcal{T}_{|X-\bc_m} = \text{complement of zero section in}\  \calP^1_{/S}(\tilde{L}^d)_{|X-\bc_m}.$$
We clearly have that $\mathcal{T}^c$ is $m$-admissible. Applying Theorem \ref{maintheorem.general}, we deduce that, in $K_0(\Var/S)$, 
\begin{equation}\label{eq:theoremBoverconfspace} \lim_{d \rightarrow \infty} \frac{[\mathrm{pr}_{2,*}(\tilde{L}^d)^{\mathcal{T}-\text{everywhere}}]}{[\mathrm{pr}_{2,*}(\tilde{L}^d)]} =\left. \prod_{x\in X/S}\left(1-\frac{[\mathcal{T}^c]_x}{[P^1(\tilde{L}^d)]_x}t\right)\right|_{t = 1}\end{equation}
On the left-hand side of (\ref{eq:theoremBoverconfspace}), we have $\mathrm{pr}_{2,*}(\tilde{L}^d) = H^0(X,L^d)\times_K S,$ and 
$$\mathrm{pr}_{2,*}(\tilde{L}^d)^{\mathcal{T}-\text{everywhere}} = \{(f,\bc)\in H^0(Y,L^d)\times S,\ \text{$f$ is singular exactly at  $|\bc|$}\},$$
so that after applying the forgetful map to $K_0(\Var/K)$, this limit is the one we are looking for. 
The Euler product on the right-hand side of (\ref{eq:theoremBoverconfspace}) may be rewritten
$$\prod_{x\in X-\bc_m/S} \left( 1-\LL^{-(n+1)}t\right) \prod_{x\in \bc_m/S} (1- (1-\LL^{-(n+1)})t).$$
$$ = \prod_{x\in X/S} (1-\LL^{-(n+1)}t)\prod_{x\in \bc_m/S}\frac{(1- (1-\LL^{-(n+1)})t)}{(1 - \LL^{-(n+1)}t)}.$$
Applying $\int_S$ (which is a morphism of $K_0(\Var/K)$-modules on each coefficient), we get
$$Z_{Y}^{\Kap}(\LL^{-(n+1)}t)^{-1} \int_S \prod_{x\in \bc_m/S}\frac{(1- (1-\LL^{-(n+1)})t)}{(1 - \LL^{-(n+1)}t)}.$$
The latter evaluated at $t = 1$ is equal to $$Z_{Y}^{\Kap}(\LL^{-(n+1)})^{-1}\left.\int_{\Conf^m Y} \prod_{y\in \bc_m/\Conf^{m} Y} \frac{t}{1-t}\right|_{t = \LL^{-(n+1)}},$$
which, using (\ref{equation.Kapranovzetam}), gives us the result. 

\begin{remark} This method allows us to compute the limiting probability of the following more general problem: given a partition $\tau$ and subsets $U, T\subset \calP^{s}(L^d)$ such that $U^c$ is admissible, what is the motivic density of the constructible set
$V_{U, T,d}^{\tau}$ of pairs $$(\bc, F) \in \Conf^{\tau} Y \times V_d  .$$
such that the Taylor expansion of $F$ lies in $T$ at the points of $|\bc|$ and in $U$ at all other points? Here one should think of $U$ as describing the ambient condition, and $T$ as describing a special condition only required along a finite set of points. 
\end{remark}

\subsection{Smooth hypersurfaces containing a subvariety}
\newcommand{\calI}{\mathcal{I}}
Let $V \subset \bbP^n_K$ be an arbitrary subvariety with ideal sheaf $\calI$. Applying Theorem \ref{maintheorem.general} with $X=\bbP^n_K$, $\calL=\calO(1)$, and $\calF=\calI$, we obtain a motivic analog of \cite[Corollary 1.4]{gunther:random-hypersurfaces} (see also \cite{wutz:thesis}); we leave the details to the interested reader. 

\subsection{Semi-ample Bertini}
The following theorem is a motivic analog of the semi-ample Bertini theorem of Erman and Wood \cite{erman-wood:semiample-bertini}.  
\newcommand{\calS}{\mathcal{S}}

\begin{theorem} Let $X/\bbC$ be a projective variety equipped with a globally generated line bundle $\calS$, and consider the induced map
\[ f_{\calS}: X \rightarrow \bbP( H^0(X, \calS)^* ). \]
Let $\calL$ be an ample line bundle on $X$, let $U_{n,d}$ denote the constructible subset of $V_{n,d} = H^0(X, \calL^n \otimes \calS^d)$ consisting of sections with smooth vanishing locus, and let $T_{n,d} \subset V_{n,d} \times \bbP( H^0(X, \calS)^* )$ 
be the constructible subset consisting of pairs $(a, s)$ such that the vanishing locus of $a$ is smooth along the fiber $f_{\calS}^{-1}(s)$. Then, for any fixed sufficiently large $n$, 
\[ \lim_{d \rightarrow \infty} \frac{[U_{n,d}]}{[V_{n,d}]} =\left. \lim_{d \rightarrow \infty} \prod_{x\in \bbP( H^0(X, \calS)^* )} \left( 1 - \frac{[T_{n,d}^c]_x}{[V_{n,d}]}t \right)\right|_{t=1}. \]
\end{theorem}

\begin{remark} If $\calS$ is not the trivial bundle then the sequence of Euler products on the right-hand side stabilizes for $d$ sufficiently large. If $\calS$ is the trivial bundle, then the statement is a tautology as for each $d$ the $d$th terms of both sides agree. 
\end{remark}

\begin{proof}
The strategy of the proof is to invoke Theorem \ref{maintheorem.general} after pushing forward the entire situation to $\bbP := \bbP(H^0(X,\calS)^*)$ (in particular, because the pushforward of a line bundle is in general just a coherent sheaf and not a vector bundle, it is crucial that we allowed this generality in Theorem \ref{maintheorem.general}). Smoothness along the fiber is remembered by the first sheaf of principal parts after the pushforward, because the fiber of this sheaf at a point is given by sections over the first infinitesimal neighborhood. The only subtlety is then to verify that the complementary Taylor condition of being singular along the fiber is admissible, which requires passing to a sufficiently large power of $\calL$. 

We write $\bbP := \bbP(H^0(X,\calS)^*)$ and $\calM = f_{\calS *} \calL^n$, a coherent sheaf on $\bbP$. Then, by the projection formula, 
\[ f_{\calS *}(\calL^n \otimes \calS^d) = \calM(d), \]
so that we have
\[ H^0(\bbP, \calM(d)) = H^0( \bbP, f_{\calS *}(\calL^n \otimes \calS^d) ) = H^0 (X, \calL^n \otimes \calS^d) = V_{n,d}. \]

 Theorem \ref{maintheorem.general} applied to $\calM$ implies that for any Taylor condition $T$ on global sections of $\calM(d)$ with admissible complement, we have
$$\frac{[H^0(\bbP, \calM(d))^{T-\text{everywhere}}]}{[H^0(\bbP,\calM(d))]}  = \left.\prod_{x\in \bbP} \left(1 - \frac{[T^c]_x}{[\bbV(\calP^1(\calM(d)))]_x} t\right)\right|_{t = 1} + O(\bbL^{-\epsilon d})$$
The idea is to deduce the result from this by applying it to a well-chosen Taylor condition $T$.

 Take $n$ large enough so that $\calL^n$ is 1-infinitesimally 1-generating. Then, since~$\calS$ is globally generated, by Remark \ref{remark:tensorbygloballygenerated} we have that $\calL^n\otimes \calS^d$ is 1-infinitesimally 1-generating for any~$d>0$, so the Taylor expansion map
$$\mathrm{Taylor}: H^0(X, \calL^n\otimes \calS^d)\times X\to \calP^1(\calL^n\otimes \calS^d)$$
is surjective. We consider the Taylor condition $T'\subset \calP^1(\calL^n\otimes \calS^d)$ given by the complement of the zero section. Note that via the map
$$\id\times f_{\calS}: V_{n,d}\times X\to V_{n,d}\times \bbP,$$
we have
$$T_{n,d}^c = (\id\times f_{\calS})(\mathrm{Taylor}^{-1}(T'^c)).$$
By definition, we have $\dim \mathrm{Taylor}^{-1}(T')_x = \dim V_{n,d} - \dim X-1$ for any $x\in X$. 
This allows us to bound the codimension of $T_{n,d}^c$.  Indeed, up to stratifying $f_{\calS}$ into irreducible pieces with fibers of constant dimension, assume all of the fibers of $f_{\calS}$ have the same dimension. Then we have
$$\dim (T_{n,d}^c)_{|s} \leq \dim T'^c_{x}+ \dim f_{\calS}^{-1}(s) = \dim V_{n,d} - (\dim X + 1) + \dim f_{\calS}^{-1}(s),$$
so that
$$\codim (T_{n,d}^c)_{|s} \geq \dim X +1 - \dim f_{\calS}^{-1}(s) = \dim \bbP + 1$$
On the other hand, $V_{n,d}\times \bbP$ may be rewritten as
$H^0(\bbP, \calM(d))\times \bbP$, and we have the Taylor expansion map
$$\mathrm{Taylor}: H^0(\bbP, \calM(d))\times \bbP \to \bbV(\calP^1(\calM(d))),$$
which is surjective for $d$ sufficiently large. Consider the image $T$ of $T_{n,d}$. Because smoothness along the fiber is remembered by $\calP^1(\calM(d))$, we have $\mathrm{Taylor}^{-1}(T^c) = T_{n,d}^c$. This shows that $T^c$ and $T_{n,d}^c$ have the same codimension, which proves admissibility of $T^c$. Moreover, by Corollary \ref{corollary:presentation-of-surjection}, this gives us the equality
$$\frac{[T_{n,d}^c]}{[V_{n,d}]} = \frac{[T^c]}{[\bbV(\calP^1(\calM(d)))]}$$
in $\M_{\bbP}$ which allows us to conclude.

\end{proof}

\subsection{Surjections onto vector bundles}

The following application of Theorem~\ref{maintheorem.general} was communicated to us by Benjamin Fayyazuddin Ljungberg; a special case is used in his proof \cite{ljungberg:thesis} of instances of the motivic Tamagawa number conjecture. 

\begin{theorem}
\newcommand{\calV}{\mathcal{V}}
Let $f:X \rightarrow S$ be a smooth projective curve, and let $\calV$ be a vector bundle of rank $n$ on $X$. Let $U_d/S$ be the open subset of surjections from $\calO(-d)^{n+1}$ onto $\calV$ inside the space of all maps
\[ V_d:=\bbV \left(f_* \mathcal{H}om(\calO(-d)^{n+1}, \calV)\right). \]
Then, in $\widehat{\widetilde{\calM}}_S$, 
\[ \lim_{d\rightarrow \infty} \frac{[U_d]}{[V_d]}=Z_{X/S}(n+1) \cdot \ldots \cdot Z_{X/S}(2) \]
where $Z_{X/S}(k)$ denotes $Z_{X/S}^{\Kap}(\LL^{-k})$. 
\end{theorem}
\begin{proof}

We write
\[ H_d = \calO(d) \otimes \calO^{n+1} \otimes \mathcal{V} = (\calO(-d)^{n+1})^* \otimes \mathcal{V} = \underline{\mathrm{Hom}} ( \calO(-d)^{n+1}, \mathcal{V}) \]
Writing $H_d$ also for the total space, we define $T_d$ to be open subset given by surjections in each fiber, which we view as a 0-infinitesimal Taylor condition on $H_d$. An equivalent definition of $T_d$ is as the locus where the dual map is an injection, and using this description, by Noetherian induction we find that in $K_0(\Var / X)$,
\[ [T_d/X] = (\bbL^{n+1} - 1) \cdot \ldots \cdot (\bbL^{n+1} - \bbL^{n-1}). \]
As $[H_d/X] = \bbL^{n(n+1)}$, we obtain
\[ \frac{[T_d/X]}{[H_d/X]}= (1 - \bbL^{-(n+1)})\cdot \ldots \cdot (1- \bbL^{-2}). \]
We write $\mathrm{GS}_d$ for the space of global surjections, i.e. the constructible subset of the total space of $f_* H_d$ satisfying the condition $T_d$ at each point of $X$. Because the complement of $T_d$ in $H_d$ has codimension $2$ (relative over $S$), Theorem \ref{maintheorem.general} gives that in $\widehat{\widetilde{\calM}}_S$,
\[ \lim_{d \rightarrow \infty} \frac{[\mathrm{GS}_d]}{[f_*H_d]} = \left. \prod_{x \in X / S} \left( 1 - (1 - (1 - \bbL^{-(n+1)})\cdot \ldots \cdot (1- \bbL^{-2})) t\right)\right|_{t=1} \]
We show now that the value on the right-hand side is equal to the product of motivic zeta values 
\[ Z_{X/S}(n+1) \cdot \ldots \cdot Z_{X/S}(2). \]
This product of zeta values can be expressed using motivic Euler products as
\begin{multline} \left.\prod_{x \in X/S}(1 - \bbL^{-(n+1)}t_1)\right|_{t_1=1} \cdot \ldots \cdot \left.\prod_{x \in X/S}(1 - \bbL^{-2}t_n)\right|_{t_n=1} = \\ \left. \left( \prod_{x \in X/S} (1 - \bbL^{-(n+1)}t_1) \cdot \ldots \cdot (1 - \bbL^{-2}t_n) \right)\right|_{t_1=1,\ldots,t_n=1}.
\end{multline}
Now, by expanding the products, we find that both \[ (1 - \bbL^{-(n+1)}t_1) \cdot \ldots \cdot (1 - \bbL^{-2}t_n)  \textrm{ and }  1 - (1 - (1 - \bbL^{-(n+1)})\cdot \ldots \cdot (1- \bbL^{-2})) t \]
can be obtained via different monomial substitutions from a single power series in variables $s_I$ for $I$ a non-empty subset of $\{1, \ldots, n\}$. The Euler product for this power series converges when all of the $s_I$ are evaluated at $1$, and thus, by applying Lemma \ref{lemma:monomialsubstitution} and evaluating all variables to $1$, we find that
\[  \left.\prod_{x \in X / S}  \left(1 - (1 - (1 - \bbL^{-(n+1)})\cdot \ldots \cdot (1- \bbL^{-2})) t\right)\right|_{t=1} = Z_{X/S}(n+1) \cdot \ldots Z_{X/S}(2). \]

\end{proof}

\section{An analytic unification of arithmetic and motivic statistics}\label{sect.conjecture}
\newcommand{\bW}{\mathbf{W}}
In this section, we explain how Poonen's finite field Bertini theorem and Vakil-Wood's motivic analog (in the Grothendieck ring of varieties over a finite field) can be unified by a single analytic conjecture about Hasse-Weil zeta functions. This discussion can be extended to a unification of Poonen's Bertini theorem with Taylor coefficients and our Theorem \ref{theorem.taylorconditions}, however, as the conjecture is already interesting in the simpler case, we constrain ourselves to just a few remarks about the generalization at the end of this section (cf. Remark \ref{remark.general-unification} below). 

The key observation is that Poonen's and Vakil-Wood's results can both be interpreted as giving the convergence of a sequence of renormalized zeta functions, but for two different topologies on the space of rational functions. In Poonen's theorem, the convergence is for the Witt topology, where a rational function is considered to be small if the first $N$ coefficients of its Taylor expansion at $0$ are small. In Vakil-Wood's theorem, the convergence is for the weight topology, where a rational function is considered to be small if its poles and zeroes are all large. 

The Witt topology and the weight topology are incompatible, but they admit a natural common refinement, which we call the Hadamard topology, and our unifying conjecture (Conjecture \ref{conj.zeta-convergence-general} below) then states that the sequence of zeta functions converges in the Hadamard topology. The Hadamard topology is closely related to the topology of uniform convergence, and in the case of $\bbP^n$ our conjecture has a particularly nice form:

\begin{conjecture}\label{conj.zeta-convergence-Pn}
Let $U_d \subset V_d := \Gamma(\bbP^n, \calO(d))$ be the space of smooth hypersurface sections of degree $d$ in $\bbP^n$. Then, 
\[ \lim_{d\rightarrow \infty} Z_{U_d/\bbF_q}\left(q^{-\dim V_d} t\right) = Z_{\mathrm{GL}_{n+1}/\bbF_q}\left(q^{-(n+1)^2}t\right) \]
in the Hadamard topology. In particular, the sequence of meromorphic functions 
\[ \frac{ Z_{U_d/\bbF_q}(q^{-\dim V_d} t) }{Z_{\mathrm{GL}_{n+1}/\bbF_q}(q^{-(n+1)^2}t)} \]
converges uniformly on compact sets in $\bbC$ to the constant function $1$. 
\end{conjecture}

\begin{example}\label{example:P1-conjecture} Conjecture \ref{conj.zeta-convergence-Pn} holds for $n=1$. Indeed, in this case, the sequence in the Grothendieck ring is constant with the correct value after $d=3$. To see this, first note that for $d \geq 2$, in $K_0(\Var/\bbF_q)$,
\[ [\Conf^d \bbA^1] = \bbL^d - \bbL^{d-1} \]
(cf. \cite[Proposition 3]{ellenberg:AWSnotes}, credited there to Mike Zieve; the proof presented there is for point counts, but in fact it lifts to the Grothendieck ring, as claimed in Exercise~12 of the same notes) and thus, for $d \geq 3$, 
\[ [\Conf^d \bbP^1] = [\Conf^d \bbA^1] + [\Conf^{d-1} \bbA^1] = \bbL^d - \bbL^{d-2}, \]
and, still for $d \geq 3$,
\[ \frac{[U_d]}{[V_d]} = \bbL^{-(d+1)}(\bbL - 1)[\Conf^d \bbP^1]= (1 - \bbL^{-1})(1 - \bbL^{-2}). \]
On the other hand,
\[ [\mathrm{GL}_2]=(\bbL^2 - 1)(\bbL^2 - \bbL), \]
and we conclude that for $d\geq 3$, 
\[ \frac{[U_d]}{\bbL^{\dim V_d}} =  \frac{ [\mathrm{GL}_2] }{\bbL^4}. \]
Passing to zeta functions, we obtain the desired identity (when passing to zeta functions, dividing by $\bbL$ corresponds to a change of variable by $t \mapsto q^{-1}t$, as we shall see below).   
\end{example}

\begin{remark}Tommasi \cite{tommasi:stable-cohomology-hypersurfaces} has shown for any $n$ that the cohomology of $U_d/\mathbb{C}$ stabilizes to the cohomology of $\mathrm{GL}_{n+1}$ (which is included via the orbit map of any point). A corresponding result in \'{e}tale cohomology over $\mathbb{F}_q$ combined with sufficient control over the size of unstable cohomology would imply Conjecture \ref{conj.zeta-convergence-Pn}. More generally, one can hope to attack the full Conjecture \ref{conj.zeta-convergence-general} by an analysis of the Vassiliev spectral sequence. 
\end{remark}

The fact that the limiting value in Conjecture \ref{conj.zeta-convergence-Pn} is the zeta function of a variety is a peculiarity of this special case. When $\bbP^n$ is replaced with an arbitrary smooth projective $X/\bbF_q$, the limiting value is the inverse of a special value of the Kapranov zeta function which is not, in general, a rational function (here the Kapranov zeta function must be thought of as a power series whose coefficients are rational functions with the Witt ring structure -- thus this special value is given by the evaluation of a power series with zeta function coefficients at a zeta function!). Nonetheless, the Kapranov zeta function is known to converge at this special value for both the Witt topology and the weight topology, and we show that it also converges in the Hadamard topology so that a general conjecture can be formulated.  

Here something nice occurs: the completion of the space of rational functions for the Witt topology is a space of power series centered at 0, but with no condition on the radius of convergence, whereas the completion for the weight topology is a space of formal divisors. In both cases, we have no natural interpretation of an arbitrary element of the completion as a function.  The completion for the Hadamard topology, on the other hand, is naturally identified with the space of meromorphic functions $f$ on $\bbC$ such that $f(0)=1$ and such that $f$ can be represented as the quotient of two entire functions of genus 0. In particular, special values of the Kapranov zeta function are naturally meromorphic functions on $\bbC$, and thus in the general conjecture we can still interpret the convergence analytically as a statement about uniform convergence on compact sets of a sequence of functions.     

\subsection{The Witt ring structure on rational functions}
\newcommand{\calR}{\mathcal{R}}
We now explain how to construct a ring structure on the set $\calR_1$ of complex rational functions $f$ with $f(0)=1$ by identifying this set of rational functions with a Grothendieck ring: let $\mathbf{Rep}_{\bbZ}$ be the the category of pairs $(V, \rho)$ where $V$ is a finite dimensional complex vector space and $\rho$ is a representation of $\bbZ$ on $V$ (to give $\rho$ it is equivalent to give the automorphism $\rho(1)$ of $V$). The characteristic power series of a linear map induces an injective map 
\begin{align*}
K_0(\mathbf{Rep}_{\bbZ}) & \rightarrow 1 + t \bbC[[t]] \\
[(V, \rho)] & \mapsto \frac{1}{\det(1-t \rho(1))} 
\end{align*}
with image $\calR_1$. The induced addition law on $\calR_1$ is multiplication of power series, called Witt addition; the induced multiplication law is called Witt multiplication, and the set $\calR_1$ equipped with this ring structure is also known as the rational Witt ring of $\bbC$. 

We note that $K_0(\mathbf{Rep}_{\bbZ})$ is also naturally isomorphic to the group ring $\bbZ[\bbC^\times]$, where the class $[a]$ in the group ring is matched with the class of the 1-dimensional representation with $\rho(1)$ given by multiplication by $a$. At the level of rational functions, this isomorphism is just the factorization of a rational function into its linear factors (which are uniquely determined because we require the value at $0$ to be $1$). 

\begin{remark}
The symmetric monoidal structure on $\mathbf{Rep}_{\bbZ}$ equips the Grothendieck ring with the structure of a $\lambda$-ring with $\sigma$-operations induced by symmetric powers and $\lambda$-operations induced by exterior powers. Under the isomorphism with~$\bbZ[\bbC^\times]$, the $\sigma$ and $\lambda$-operations are group homomorphisms $$\sigma_s, \lambda_s: K_0(\mathbf{Rep}_{\bbZ})\to 1 + sK_0(\mathbf{Rep}_{\bbZ})[[s]]$$ determined by
\[ \sigma_s( [a] )= 1 + \sigma_1([a])s + \sigma_2([a])s^2 + \ldots = 1 + [a]s+ [a^2] s^2 + \ldots \]
and 
\[ \lambda_s( [a] ) = 1 + \lambda_1([a])s + \lambda_2([a])s^2 + \ldots = 1 + [a]s. \]
Recall that in a $\lambda$-ring the $\sigma$ and $\lambda$-operations are related by $\sigma_s(x)\cdot \lambda_{-s}(x)= 1$, so that one can be deduced from the other. 
\end{remark}

\subsection{Topologies on rational functions}
\newcommand{\calW}{\mathcal{W}}
In this section we describe three natural topologies on $\calR_1$. 

\subsubsection{The Witt topology} \label{subsubsect:Witt-topology} There is a natural injective map 
\[ \calR_1 \rightarrow 1+t\bbC[[t]] \]
given by taking the power series expansion at zero. The Witt topology on $\calR_1$ is induced by the product topology on the coefficients of
\[ 1 + t\bbC[[t]] = \bbC^{\bbN}, \]
and $\calR_1$ is dense when viewed as a subset of $1+t\bbC[[t]]$ so that the completion of $\calR_1$ for the Witt topology is naturally identified with 
\[ 1 + t\bbC[[t]]. \]
The addition, multiplication, and lambda-ring structure are all continuous for the Witt topology, so that they extend to $1+ t\bbC[[t]]$ which is thus a complete topological $\lambda$-ring. With this ring structure, it is also known as the big Witt ring $W(\bbC)$. 

Instead of taking the power-series expansion of a rational function $f$, one could instead take the power-series expansion of $d\log f$, and we would obtain the same topology. In fact, $d\log f$ gives a bijection 
\[ 1 + t\bbC[[t]] \rightarrow t \bbC[[t]] = \bbC^{\bbN} \]
that is an isomorphism of topological rings when $\bbC^{\bbN}$ on the right is equipped with the product topology and ring structure. The coefficients of $d \log f$ are also called the ghost coordinates on the big Witt ring. 

Using this observation, we can also describe the Witt topology in terms of $\bbZ[\bbC^\times]=\calR_1.$ It is induced by the family of semi-norms $|| \cdot ||_j,\, j=1,2, \ldots$ 
\[ \left|\left| \sum_a k_a [a] \right|\right|_j = \left| \sum_a k_a a^j \right|. \]
Indeed, this follows from the above discussion and the computation
\[ d\log \left(\prod_{a}(1-ta)^{-k_a}\right) = \sum_{j=1}^\infty \left(\sum_a k_a a^j\right) t^{j-1}. \]
We note that there is no natural description of the completion of $\bbZ[\bbC^\times]$ for the Witt topology in terms of divisors on $\bbC^\times$. 

\subsubsection{The weight topology}
In the weight topology, a basis of open neighborhoods of $f \in \calR_1$ is given by, for each $r>0$, the set of all rational functions $g$ with the same zeroes and poles as $f$ on the ball $|t| \leq r$. In particular, a sequence converges if and only if on every bounded set, the zeroes and poles eventually stabilize. 

Viewed as the group ring $\bbZ[\bbC^\times]$, a basis of open neighborhoods of zero is given by the set of all finite sums $\sum_{a \in \bbC^\times} k_a [a]$ supported on the closed ball of radius $r$ around $0 \in \bbC$ (here we are using that $[a]$, as a rational function, has a pole at $a^{-1}$), and a basis of open neighborhoods at any other point is given by translation. 

The completion of $\bbZ[\bbC^\times]$ for the weight topology can be described as the set of formal sums 
\[ \sum_{a \in \bbC^\times} k_a [a] \]
whose support is a discrete subset of $\bbC$ and whose set of accumulation points in $\bbC \sqcup \infty$ is contained in $\{0\}$. The addition, multiplication, and lambda-ring structure are all continuous for the weight topology, and extend in the obvious way to these formal sums. 

We note that there is no natural description of the completion of $\calR_1$ for the weight topology in terms of the power series expansion at $0$. 

\subsubsection{The Hadamard topology}

The Hadamard topology is most simply described under the isomorphism $\calR_1 \rightarrow \bbZ[\bbC^\times]$, where it is the topology induced by the norm 
\[ \left|\left| \sum k_a [a] \right|\right|= \sum |k_a||a|. \]
It is easy to see the Witt and weight topologies on $\calR_1$ are not comparable (i.e., neither is finer than the other). However:
\begin{lemma} The Hadamard topology refines both the Witt and weight topology. 
\end{lemma}
\begin{proof}
Each of the semi-norms $|| \cdot ||_j$ defining the Witt topology is continuous for the norm $|| \cdot ||$, and thus the Hadamard topology refines the Witt topology. 

To compare with the weight topology, it suffices to observe that if $f=\sum k_a [a]$ is supported inside the closed ball of radius $r$, then so is any $g$ with $||f - g|| < r$. 
\end{proof}

We define the Hadamard-Witt ring $\calW$ to be the completion of $\bbZ[\bbC^\times]$ for the norm $||\cdot||$.  It can be identified with the set of discretely supported divisors
\[ \sum_{a \in \bbC^\times} k_a [a] \] 
such that $\sum_{a \in \bbC^\times}  |k_a| |a| < \infty$.
It is an elementary computation to check that the multiplication and $\sigma$ (or $\lambda$) operations are continuous, so that they extend to $\calW$ which is thus a complete topological $\lambda$-ring. 

\newcommand{\calH}{\mathcal{H}}
A \emph{Hadamard} function is a meromorphic function $f$ on $\bbC$ such that $f$ can be written as a quotient $f=\frac{g}{h}$ where $g$ and $h$ are both entire functions of genus zero. In the next lemma, we extend the identification of $\bbZ[\bbC^\times]$ with $\calR_1$ to an identification of $\calW$ with the set $\calH_1$ of Hadamard functions $f$ such that $f(0)=1$. 

\begin{lemma}
If $\sum_{a \in \bbC^\times} k_a [a] \in \calW$ and 
\[ \sum_{a \in \bbC^\times} k_a [a] = \sum_{a \in \bbC^\times} k_a^+ [a] + \sum_{a \in \bbC^\times} k_a^- [a]\]
is the unique decomposition with $k_a^+ \geq 0$ and $k_a^- \leq 0$, then the infinite products 
\[ \prod_a \left(\frac{1}{1-ta}\right)^{k_a^-} \textrm{ and } \prod_a \left(\frac{1}{1-ta}\right)^{-k_a^+} \]
converge uniformly on compacts to entire functions of genus zero $f^-$ and $f^+$. Furthermore, the map
\[ \sum_{a \in \bbC^\times} k_a[a] \mapsto \frac{f^-}{f^+} \]
induces a bijection $\calW \rightarrow \calH_1$ extending the bijection $\bbZ[\bbC^\times] \rightarrow \calR_1$. 
\end{lemma}
\begin{proof}
This is an immediate consequence of the Hadamard factorization theorem in the case of genus zero entire functions. 
\end{proof}

Because the Hadamard topology refines the Witt and weight topologies, there are natural maps between the completions, and these maps have natural function-theoretic interpretations: 
\begin{enumerate} 
\item The map from the Hadamard completion to the Witt completion is given by taking $f(t) \in \calH_1$ to its power series at 0. 
\item The map from the Hadamard completion to the weight completion is given by taking $f(t) \in \calH_1$ to the divisor of $\frac{1}{f(t^{-1})}$ (here the inverses are a consequence of the various normalizations we have chosen). 
\end{enumerate}

\subsection{Hasse-Weil zeta functions}
For $X/\bbF_q$ a variety, the Hasse-Weil zeta function is 
\[ Z_X(t) := \prod_{x \in |X|} \frac{1}{1-t^{\deg x}} \in 1 + t\bbZ[[t]] \]
where $|X|$ denotes the set of closed points of $X$. 

By the rationality of the zeta function, we may view this instead as a map to~$\calR_1$. In fact, it extends to a map of (pre-)$\lambda$-rings
\begin{equation} \label{eq:zetamap}K_0(\Var / \bbF_q) \rightarrow \calR_1, \end{equation}
where $\calR_1$ is equipped with the Witt ring structure.

\begin{remark} The most natural way to see that this map respects the ring and (pre-)$\lambda$ structures is to realize $1 + t \bbZ[[t]]$ with the Witt ring structure as the Grothendieck ring of almost-finite cyclic sets\footnote{An almost-finite cyclic set, or an admissible $\bbZ$-set, is a set $A$ equipped with an action of $\bbZ$ such that each point has non-trivial stabilizer and, for each $n>0$, the set of fixed points $A^{n\bbZ}$ is finite.}; the zeta function is then just induced by the natural functor
\[ X \mapsto X(\overline{\bbF_q}) \actsonl \mathrm{Frob}_q^{\bbZ}. \]
This observation is due originally to Witt\footnote{We thank Lars Hesselholt for pointing out to us that Witt was aware of this construction of the big Witt vectors.} and was first elaborated in print by Dress and Siebeneicher \cite{dress-siebeneicher:burnsidering}. We learned of this construction from a survey of Ramachandran~\cite{ramachandran:wittring}.
\end{remark}

Because $[q]=\frac{1}{1-tq}$ is invertible in $\calR_1$ (with inverse $[q^{-1}]=\frac{1}{1-tq^{-1}}$), we find that (\ref{eq:zetamap}) extends to a morphism of (pre-)$\lambda$-rings
\begin{equation}\label{eqn.zeta-map-L-inverted} \calM_{\bbF_q} \rightarrow \calR_1. \end{equation}

We make two important observations:
\begin{enumerate}
\item For $X/\bbF_q$ a variety, $d\log Z_X(t) = \sum_{k \geq 1} \#X(\bbF_{q^k}) t^{k-1}$, so that the $k$th ghost coordinate of $Z_X(t)$ is the number of points in $X(\bbF_{q^k})$. We will use this to interpret Poonen's results on convergence of point counts as a statement about convergence of zeta functions in the Witt topology.
\item It is a consequence of the Weil conjectures that the map \ref{eqn.zeta-map-L-inverted} is continuous when the left-hand side is equipped with the dimension topology and the right-hand side is equipped with the weight topology. In particular, convergence of a sequence of classes in the Grothendieck ring in the dimension filtration implies convergence of the corresponding sequence of zeta functions in the weight topology on $\calR_1$. We will use this to interpret Vakil-Wood's result in the Grothendieck ring as a statement about convergence of zeta functions in the weight topology
\end{enumerate}

\begin{remark}
If Grothendieck's standard conjectures hold, then the natural map from the Grothendieck ring of Chow motives over $\bbF_q$ to $\calR_1$ is injective. In particular, the zeta map (\ref{eqn.zeta-map-L-inverted}) should be thought of as retaining much of the essential geometric information. 
\end{remark}

\subsection{The Kapranov zeta function}

The zeta function appearing in Poonen's stabilization is the Hasse-Weil zeta function, while the zeta function appearing in Vakil-Wood's stabilization is the Kapranov zeta function: for $X/\bbF_q$ a variety,
 \[ Z^{\mathrm{Kap}}_X(s) = 1 + [X]s + [\Sym^2 X]s^2 + \ldots \in 1 + s K_0(\Var/\bbF_q)[[s]]. \]
Passing from $K_0(\Var/\bbF_q)$ to $\calR_1$ via the zeta map, the Kapranov zeta function becomes
\begin{equation}\label{eqn.kap-zeta-series} Z^{\mathrm{Kap}}_X(s) = 1 + Z_X(t) s + Z_{\Sym^2 X}(t) s^2 + \ldots \in 1 + s\calR_1[[s]]. \end{equation}
From now on, we will consider only this version of the Kapranov zeta function. 

Because the map (\ref{eqn.zeta-map-L-inverted}) is a map of (pre-)$\lambda$ rings, we also have
\[ Z^{\mathrm{Kap}}_X(s) = \sigma_s(Z_X(t)).\]
It follows from the formula for the $\sigma$-operations on $\calR_1$ that for \emph{any} rational function $f \in \calR_1$, if we factorize 
\[ f=\frac{ \prod (1-t a_i)}{\prod(1-t b_i)}, \]
then 
\begin{equation}\label{eqn.sigma-rational} \sigma_s(f)= \frac{\prod( 1- s[a_i])}{\prod (1-s[b_i])} \end{equation}
and, in particular, the (specialization to $\calR_1$ of) the Kapranov zeta function is a rational function (suitably interpreted). 

\subsubsection{Special values}
For $X/\bbF_q$ a variety, we find that if $a \in \bbC^\times$ is not a pole of $Z_X(t)$, then the Kapranov zeta function has a well-defined value in $\calW$ at $[a]$ given by evaluating the expression \ref{eqn.sigma-rational} as a rational function at $[a]$ (indeed, the only thing to check is that for $x \in \bbC^\times,\, x \neq 1$, $1-[x]$ has an inverse in $\calW$). 

\begin{remark} When $|a| < q^{-\dim X}$, the series expression (\ref{eqn.kap-zeta-series}) also converges at $[a]$. In particular, for such an $a$, we find that in $\calH_1$,
\[ Z_X^{\Kap}([a]) = \prod_{k=1}^\infty Z_{\Sym^k X}(a^k t) \]
where convergence of the infinite product should be interpreted as saying that the numerators and denominators of the finite products converge uniformly on compacts to entire functions. 
\end{remark}

Recall that the map (\ref{eqn.zeta-map-L-inverted}) is continuous when $\calM_{\bbF_q}$ is equipped with the dimension topology and $\calW=\calH_1$ with the weight topology. Thus, the following result is a consequence of the motivic Bertini theorem of Vakil-Wood (suitably modified for positive characteristic, cf. Remark~\ref{remark:vakil-wood-error}), which takes place in the completed Grothendieck ring $\widehat{\widetilde{\calM}}_{\bbF_q}$, and the observation that, after passing to zeta functions, special values of the Kapranov zeta function are contained in 
\[ \calW =\calH_1 \subset 1 + t\bbC[[t]].\]

\begin{theorem}[Vakil-Wood]\label{theorem.vakil-wood-revisited} Let $X/\bbF_q$ be a smooth projective variety equipped with an ample line bundle $\calL$, and let $U_d \subset \Gamma(X, \calL^d)$ be the open subvariety of smooth sections. Then, in the weight topology on $\calH_1$, 
\begin{align*} \lim_{d \rightarrow \infty} Z_{U_d}( q^{-\dim U_d} t)&  = \left( Z_X^\Kap\left([q^{-(\dim X +1)}]\right) \right)^{-1} \\
& = \left(\prod_{k=1}^\infty Z_{\Sym^k X}\left(q^{-k(\dim X + 1)} t\right)\right)^{-1}. 
\end{align*}
\end{theorem}

\subsubsection{Kapranov zeta and point-counting}
As can be verified either from the series or rational function expansion of $Z_X^\Kap(s)$, applying the $k$th ghost coordinate map~$g_k$ (recall that, as explained in \ref{subsubsect:Witt-topology}, this is the ring homomorphism which sends a rational function $f(t)$ to the coefficient of $t^{k-1}$ in the power series expansion at 0 of $d\log f$) to each coefficient, we obtain
\[ g_k(Z_X^\Kap(s))=Z_{X_{\bbF_{q^k}}}(s). \]
Thus, one way to think of the Kapranov zeta function of $X/\bbF_q$ is as working simultaneously with the Hasse-Weil zeta functions of the base changes of $X$ to all finite extensions of $\bbF_q$. 

Applying this to special values, we find that for $a$ not a pole of $Z_X(t)$ (and recalling that the ghost vector map $g_k$ extends in the obvious way to $\calW$)
\[ g_k(Z_X^\Kap(a)) = Z_{X_{\bbF_{q^k}}}(a^k). \]

Putting everything together, we obtain the following reinterpretation of Poonen's Bertini theorem, which is identical to Theorem \ref{theorem.vakil-wood-revisited} except with the weight topology replaced by the Witt topology. 

\begin{theorem}[Poonen]\label{theorem.poonen-revisited} Let $X/\bbF_q$ be a smooth projective variety equipped with an ample line bundle $\calL$, and let $U_d \subset \Gamma(X, \calL^d)$ be the open subvariety of smooth sections. Then, in the Witt topology on $\calH_1$, 
\begin{align*} \lim_{d \rightarrow \infty} Z_{U_d}( q^{-\dim U_d} t)&  = \left( Z_X^\Kap\left([q^{-(\dim X +1)}]\right) \right)^{-1} \\
& = \left(\prod_{k=1}^\infty Z_{\Sym^k X}\left(q^{-k(\dim X + 1)} t\right)\right)^{-1}. 
\end{align*}
\end{theorem}

\subsection{The general conjecture}

Given Theorems \ref{theorem.vakil-wood-revisited} and \ref{theorem.poonen-revisited}, it is natural to conjecture that the convergence in fact occurs in the Hadamard topology:

\begin{conjecture}\label{conj.zeta-convergence-general} Let $X/\bbF_q$ be a smooth projective variety equipped with an ample line bundle $\calL$, and let $U_d \subset \Gamma(X, \calL^d)$ be the open subvariety of smooth sections. Then, in the Hadamard topology on $\calH_1$, 
\begin{align*} \lim_{d \rightarrow \infty} Z_{U_d}( q^{-\dim U_d} t)&  = \left( Z_X^\Kap\left([q^{-(\dim X +1)}]\right) \right)^{-1} \\
& = \left(\prod_{k=1}^\infty Z_{\Sym^k X}\left(q^{-k(\dim X + 1)} t\right)\right)^{-1}. 
\end{align*}
\end{conjecture}

\newcommand{\GL}{\mathrm{GL}}
\begin{remark}\label{remark:conj-general-pn}
When $X=\bbP^n$, we explain how to evaluate the special value of the Kapranov zeta function appearing in \ref{conj.zeta-convergence-general} to obtain the simpler form given in Conjecture \ref{conj.zeta-convergence-Pn}:
Note that 
\[ Z_{\bbP^n}(t)=\frac{1}{(1-t)(1-qt)\cdot \ldots \cdot (1-q^nt)} \]
and thus
\[ Z^{\Kap}_{\bbP^n}(s)=\frac{1}{(1-s)(1-[q]s)\cdot \ldots \cdot (1-[q^n]s)}. \]
Evaluating, we obtain
\[ \left(Z_{\bbP^n}^\Kap([q^{-(n+1)}])\right)^{-1} = (1-[q^{-(n+1)}])(1-[q^{-n}])\cdot \ldots \cdot (1-[q^{-1}]). \]
On the other hand, in $K_0(\Var /\bbF_q)$, 
\[ [\GL_{n+1}]=(\bbL^{n+1}-1)(\bbL^{n+1} - \bbL) \ldots  (\bbL^{n+1} - \bbL^{n}) \]
and thus 
\[ [\GL_{n+1}]\bbL^{-(n+1)^2} = (1-\bbL^{-n+1})(1- \bbL^{-n}) \ldots (1 - \bbL^{-1}). \]
And, finally, passing through the map (\ref{eqn.zeta-map-L-inverted})
\begin{eqnarray*} Z_{\GL_{n+1}}(q^{-(n+1)^2} t)&=& (1-[q^{-(n+1)}])(1- [q^{-n}]) \ldots (1 - [q^{-1}]) \\&=& \left(Z_{\bbP^n}^\Kap([q^{-(n+1)}])\right)^{-1}. \end{eqnarray*}
\end{remark}

\begin{example}
Combining Example \ref{example:P1-conjecture} and Remark \ref{remark:conj-general-pn}, we find that Conjecture \ref{conj.zeta-convergence-general} holds when $X=\bbP^1$. In fact, whenever $X$ is a curve, a straightforward argument with the Abel-Jacobi map can be used to prove Conjecture \ref{conj.zeta-convergence-general} (we thank Ronno Das for a discussion that led to this observation). Unlike the case of $\bbP^1$, when $X$ has genus at least $1$ the inverse special value of the Kapranov zeta function appearing will no longer be a rational function, and the stabilization occurs only in the limit, not at a finite step. 
\end{example}

\begin{remark}\label{remark.general-unification}
As indicated in the introduction to this section, it is also possible to give a more general conjecture unifying Poonen's theorem with general Taylor conditions and our Theorem \ref{theorem.taylorconditions}. It is interesting to do so: for example, along the way, one finds that over finite fields motivic Euler products have a nice interpretation in terms of classical Euler products after passing to the ghost coordinates of their zeta functions. However, for the sake of brevity, we leave the details to future work, as they would take us too far afield here.
\end{remark}

\bibliography{references}

\begin{thebibliography}{10}

\bibitem{bilu:thesis}
Margaret Bilu.
\newblock {\em Produits eul\'{e}riens motiviques}.
\newblock PhD thesis, Orsay, 2017.

\bibitem{bilu:motiviceulerproducts}
Margaret Bilu.
\newblock Motivic {E}uler products and motivic height zeta functions.
\newblock {\em \url{https://arxiv.org/abs/1802.06836}}, 2018.

\bibitem{bilu-howe:motivic-random-variables}
Margaret Bilu and Sean Howe.
\newblock Motivic random variables.
\newblock {\em In preparation}.

\bibitem{bucur-kedlaya:complete-intersections}
Alina Bucur and Kiran~S. Kedlaya.
\newblock The probability that a complete intersection is smooth.
\newblock {\em J. Th\'{e}or. Nombres Bordeaux}, 24(3):541--556, 2012.

\bibitem{chambert-loir-et-al:motivic-integration}
Antoine Chambert-Loir, Johannes Nicaise, and Julien Sebag.
\newblock {\em Motivic integration}, volume 325 of {\em Progress in
  Mathematics}.
\newblock Birkh\"{a}user/Springer, New York, 2018.

\bibitem{dress-siebeneicher:burnsidering}
Andreas W.~M. Dress and Christian Siebeneicher.
\newblock The {B}urnside ring of the infinite cyclic group and its relations to
  the necklace algebra, {$\lambda$}-rings, and the universal ring of {W}itt
  vectors.
\newblock {\em Adv. Math.}, 78(1):1--41, 1989.

\bibitem{ekedahl:geometric-invariant}
Torsten Ekedahl.
\newblock A geometric invariant of a finite group.
\newblock {\em arXiv: 0903.3148v1}, 2009.

\bibitem{ellenberg:AWSnotes}
Jordan Ellenberg.
\newblock Arizona winter school 2014 notes: geometric analytic number theory.
\newblock {\em
  \url{http://swc.math.arizona.edu/aws/2014/2014EllenbergNotes.pdf}}.

\bibitem{erman-wood:semiample-bertini}
Daniel Erman and Melanie~Matchett Wood.
\newblock Semiample {B}ertini theorems over finite fields.
\newblock {\em Duke Math. J.}, 164(1):1--38, 2015.

\bibitem{ljungberg:thesis}
Benjamin Fayyazuddin~Ljungberg.
\newblock {\em Moduli spaces of bundles via motivic probabilities}.
\newblock PhD thesis, Stanford University, 2019.

\bibitem{EGAIVc}
A.~Grothendieck.
\newblock \'{E}l\'{e}ments de g\'{e}om\'{e}trie alg\'{e}brique. {IV}. \'{E}tude
  locale des sch\'{e}mas et des morphismes de sch\'{e}mas. {III}.
\newblock {\em Inst. Hautes \'{E}tudes Sci. Publ. Math.}, (28), 1966.

\bibitem{EGAIVd}
A.~Grothendieck.
\newblock \'{E}l\'{e}ments de g\'{e}om\'{e}trie alg\'{e}brique. {IV}. \'{E}tude
  locale des sch\'{e}mas et des morphismes de sch\'{e}mas {IV}.
\newblock {\em Inst. Hautes \'{E}tudes Sci. Publ. Math.}, (32), 1967.

\bibitem{gunther:random-hypersurfaces}
Joseph Gunther.
\newblock Random hypersurfaces and embedding curves in surfaces over finite
  fields.
\newblock {\em J. Pure Appl. Algebra}, 221(1):89--97, 2017.

\bibitem{glm:power-structure}
S.~M. Gusein-Zade, I.~Luengo, and A.~Melle-Hern\'{a}ndez.
\newblock A power structure over the {G}rothendieck ring of varieties.
\newblock {\em Math. Res. Lett.}, 11(1):49--57, 2004.

\bibitem{howe:mrv1}
Sean Howe.
\newblock Motivic random variables and representation stability {I}:
  Configuration spaces.
\newblock {\em \url{https://arxiv.org/abs/1610.05723}}.

\bibitem{howe:mrv2}
Sean Howe.
\newblock Motivic random variables and representation stability {II}:
  {H}ypersurface sections.
\newblock {\em Adv. Math.}, 350:1267--1313, 2019.

\bibitem{mustata:zeta}
Mircea Mustata.
\newblock Zeta functions in algebraic geometry.
\newblock \url{http://www.math.lsa.umich.edu/~mmustata/zeta_book.pdf}.

\bibitem{poonen:bertini}
Bjorn Poonen.
\newblock Bertini theorems over finite fields.
\newblock {\em Ann. of Math. (2)}, 160(3):1099--1127, 2004.

\bibitem{ramachandran:wittring}
Niranjan Ramachandran.
\newblock Zeta functions, {G}rothendieck groups, and the {W}itt ring.
\newblock {\em Bull. Sci. Math.}, 139(6):599--627, 2015.

\bibitem{stacks-project}
The {Stacks project authors}.
\newblock The {S}tacks project.
\newblock \url{https://stacks.math.columbia.edu}, 2019.

\bibitem{tommasi:stable-cohomology-hypersurfaces}
Orsola Tommasi.
\newblock Stable cohomology of spaces of non-singular hypersurfaces.
\newblock {\em Adv. Math.}, 265:428--440, 2014.

\bibitem{vakil-wood:discriminants}
Ravi Vakil and Melanie~Matchett Wood.
\newblock Discriminants in the {G}rothendieck ring.
\newblock {\em Duke Math. J.}, 164(6):1139--1185, 2015.

\bibitem{wutz:thesis}
Franziska Wutz.
\newblock {\em Bertini theorems for hypersurface sections containing a
  subscheme over finite fields}.
\newblock PhD thesis, Universit{\"a}t Regensburg, 2014.

\end{thebibliography}
\bibliographystyle{plain}
\end{document}